\documentclass[11pt,letterpaper]{article}
\usepackage{amssymb,amsfonts,amsmath,amsthm,amsopn,amstext,amscd,latexsym,stmaryrd}
\usepackage[margin=0.8in]{geometry}
\usepackage[latin1]{inputenc}
\usepackage{amsmath, amssymb, amsbsy, mathdots, mathabx}
\usepackage{mathtools}
\usepackage{mathrsfs}
\usepackage{amsthm}
\usepackage{enumitem,makeidx}
\usepackage[all]{xy}
\usepackage{tikz}
\usepackage{changepage, comment}
\usepackage[mathcal]{euscript}
\usepackage{thmtools, thm-restate}
\usepackage{url}

\title{Derived Hecke action at $p$}
\author{Chandrashekhar Khare, Niccol\`o Ronchetti}

\date{\today}

\theoremstyle{theorem}
\newtheorem*{thm*}{Theorem}
\newtheorem{thm}{Theorem}[section]
\newtheorem{lem}[thm]{Lemma}
\newtheorem{prop}[thm]{Proposition}
\newtheorem{hypothesis}[thm]{Hypothesis}
\newtheorem{fact}[thm]{Fact}
\newtheorem{claim}[thm]{Claim}
\newtheorem{cor}[thm]{Corollary}
\newtheorem{conj}[thm]{Conjecture}
\newtheorem*{conj*}{Conjecture}
\theoremstyle{definition}
\newtheorem{defn}{Definition}[section]

\theoremstyle{remark}
\newtheorem{rem}{Remark}

\newcommand{\End}{\ensuremath{\mathrm{End}}}
\newcommand{\Hom}{\ensuremath{\mathrm{Hom}}}
\newcommand{\Ext}{\ensuremath{\mathrm{Ext}}}

\newcommand{\Mod}{\ensuremath{\mathrm{Mod}}}
\newcommand{\Lie}{\ensuremath{\mathrm{Lie} \,}}

\newcommand{\Ad}{\ensuremath{\mathrm{Ad}}}
\newcommand{\coad}{\ensuremath{\mathrm{coad}}} %coadjoint
\newcommand{\Gal}{\ensuremath{\mathrm{Gal}}}
\newcommand{\Rep}{\ensuremath{\mathrm{Rep}}}

\newcommand{\cInd}[2]{\ensuremath{\iota^{#1}_{#2}}} %compact induction from #2 to #1

\newcommand{\unr}{\ensuremath{\mathrm{unr}}} % unramified
\newcommand{\Stab}{\ensuremath{\mathrm{Stab}}}

\newcommand{\Frob}{\ensuremath{\mathrm{Frob}}}

\newcommand{\coker}{\ensuremath{\mathrm{coker}}}

\newcommand{\rk}{\ensuremath{\mathrm{rk}}}
\newcommand{\projdim}{\ensuremath{\mathrm{projdim}}} %projective dimension
\newcommand{\dpth}{\ensuremath{\mathrm{depth}}} %depth

\newcommand{\Lift}{\ensuremath{\mathrm{Lift}}}
\newcommand{\Def}{\ensuremath{\mathrm{Def}}}
\newcommand{\Sets}{\ensuremath{\mathrm{Sets}}}

\newcommand{\ab}{\ensuremath{\mathrm{ab}}}
\newcommand{\ordi}{\ensuremath{\mathrm{ord}}} %ordinary
\newcommand{\crys}{\ensuremath{\mathrm{crys}}} %crystalline

\newcommand{\ad}{\ensuremath{\mathrm{ad}}}
\newcommand{\im}{\ensuremath{\mathrm{Im}}}

\newcommand{\lra}{\ensuremath{\longrightarrow}}

\newcommand{\Q}{\ensuremath{\mathbb Q}}
\newcommand{\Qp}{\ensuremath{\mathbb Q_p}}
\newcommand{\Gm}{\ensuremath{\mathbb G_m}}

\newcommand{\Fp}{\ensuremath{\mathbb F_p}}

\newcommand{\R}{\ensuremath{\mathbb R}}
\newcommand{\Z}{\ensuremath{\mathbb Z}}
\newcommand{\Zp}{\ensuremath{\mathbb Z_p}}

\newcommand{\A}{\ensuremath{\mathbb A_{\Q}}}
\newcommand{\Af}{\ensuremath{\mathbb A^{\infty}_{\Q}}}

\newcommand{\Gl}[1]{\ensuremath{\mathrm{GL}_{#1} }}

\newcommand{\Sl}[1]{\ensuremath{\mathrm{SL}_{#1} }}

\newcommand{\res}[2]{\ensuremath{\mathrm{res}^{#1}_{#2} }} % restriction in cohomology or of representations. The larger group goes first
\newcommand{\cores}[2]{\ensuremath{\mathrm{cores}^{#1}_{#2} }} % corestriction in cohomology. The larger group first
\newcommand{\OO}{\ensuremath{\mathcal O}}

\begin{document}

\title{Derived Hecke action at $p$ and the ordinary $p$-adic cohomology of arithmetic manifolds}
\maketitle

\begin{abstract}
 We study the derived Hecke action at $p$ on the ordinary $p$-adic cohomology of arithmetic subgroups of reductive groups $\mathrm G(\Q)$, i.e., we study the derived version of Hida's theory for ordinary Hecke algebras. This is the analog at $\ell=p$ of derived Hecke actions studied by Venkatesh in the tame case. We show that properties of the derived Hecke action at $p$ are related to deep conjectures in Galois cohomology which are higher analogs of the classical Leopoldt conjecture.
\end{abstract}

\tableofcontents

\section{Introduction}

Let $F$ be a number field and $p$ a rational prime.   The classical conjecture of Leopoldt  asserts that the kernel of the map $\widehat{\mathcal O_{F}^*} \rightarrow \overline{\mathcal O_{F}^*}$, from the pro-$p$ completion of  the units of $\mathcal O_F^*$ to its $p$-adic completion, is trivial. Denote by $\delta_{F,p}$ the $\Z_p$-rank of the kernel which is conjectured to be 0.  The classical Leopoldt conjecture has several equivalent formulations. It can also be cast as the assertion that the map $H^1_f(G_{F,S},\Q_p(1)) \rightarrow H^1(G_p,\Q_p(1))$ is injective; the $\Q_p$-dimension of the kernel of the map is  the Leopoldt defect $\delta_{F,p}$.  Note also that  ${\rm dim}_{\Q_p} H^1((G_{F,S}, \Q_p)=1+r_2+\delta_{F,p}$  (where $[F:\Q]=r_1+2r_2$ and $r_2$ is the number of  complex places of $F$). One also knows that the Selmer group $H^1_f(G_{F,S}, \Q_p)=0$ which is a consequence of  (the $p$-part of  the) class group of $F$  being finite.

 In this paper we want to consider higher dimensional, non-abelian analogs of these conjectures and relate them to derived Hecke actions (cf. \cite{akshay}). We recall the folklore conjecture that if $V$ is a $p$-adic representation arising from a  motive $M$ over $F$, $S$ a finite set of places  of $F$ containing the infinite and $p$-adic places of $F$ and places at which $M$ has bad reduction,   such that  $M$ is generic at all places $v \in S$, i.e., $H^2(G_v,V)=0$ for  places $v \in S$ of $F$, then $H^2(G_{F,S}, V)=0$.   (Furthermore the genericity is expected to hold as long as $V$ is not 1-dimensional, more precisely does not arise from $\Q_p(1)$ up to twist by a finite order character.)   
  
  The rest of the paper is focused on cases when $V$  is the $p$-adic representation $\ad(\rho_\pi)$, the adjoint of the representation $\rho_\pi$ arising  a cohomological cuspidal  automorphic representation $\pi$ of $\mathrm G(\mathbf A_\Q)$ with $\mathrm G$ a split, connected, reductive $\Q$-group of $\Q$-rank $r$ and we consider  the conjectured injection  $H^1_f(\Z[{ 1 \over S}],\ad(\rho_\pi)(1))  \hookrightarrow H^1_f(G_p,\ad(\rho_\pi)(1))$.   (More precisely, $\rho_\pi=\rho_{\pi,\iota}$ is the Galois representation associated to $\pi$ and  an embedding $\iota:E_\pi  \hookrightarrow \overline \Q_p$ of the
  Hecke field $E_\pi$ of $\pi$: for the introduction we assume $\iota(E_\pi) $ is contained in $\Q_p$.)     The Poitou-Tate sequence shows that such an injection  implies that $H^2(G_{F,S}, \ad(\rho_{\pi}))=0$ if we assume further that  $\pi_v$ is generic at places $v \in S$.  We relate this deep conjecture in Galois cohomology  to considering cohomology of arithmetic groups and derived Hecke actions which  have a topological definition. The work of Hida in \cite{Hida} is an important precursor to our work as he drew a line  between the Leopoldt  conjecture
  (and its higher dimension analogs)   and structure of ordinary cohomology of arithmetic groups. Our paper may be viewed as an elaboration  of  the study of ordinary $p$-adic Hecke algebras in \cite{Hida} (see also \cite{KT})    using   the beautiful idea introduced in Venkatesh's work \cite{akshay} of studying (degree-shifting)  derived Hecke actions.  We  now turn to describing more concretely the work we do in this paper.
  
  \subsection*{Main results}
  
It is well-known that in good situations, the tempered cohomology of an arithmetic manifold $Y$  arising from congruence subgroups $\Gamma$ of  $\mathrm G(\Z)$  is concentrated around the middle degree and has dimension resembling that of an exterior algebra: \[ \dim_{\Q} H^{q_0+i} \left( Y, \Q \right)_{\pi }= \binom{l_0}{i}. \]
A thorough discussion of this phenomenon can be found for instance in \cite{akshay2} - here  $l_0, q_0$ are quantities associated to the specific arithmetic manifold $Y$, defined in section \ref{notationsection}.

In a recent series of papers (\cite{akshay,GV,PV,HV}) the above phenomenon of  redundancy in the cohomology  of $Y$ is explained  by introducing a \emph{degree-increasing} action of certain \emph{derived Hecke operators} on the Hecke-eigenspace $H^* \left( Y, \Q_p \right)_{\pi}$,  which conjecturally should be the extra symmetries that produce the redundancy.  In loc. cit. it is  shown (under certain assumptions) that this action generates the entire eigenspace over the lowest nonzero degree, as well as that this action has a Galois interpretation, using simultaneous actions of derived Hecke operators at many good primes $l \neq p$. 
 
 We consider in this  paper a derived Hecke action at $l=p$ under the assumption 
 that $\pi$ is {\it ordinary at all the places above $p$} which allows us to use Hida theory to realize the derived Hecke action at $p$ as arising from the covering group of the Hida tower. Thus we construct the algebra of  \emph{derived diamond operators} at $p$,  $\bigwedge^* \Hom \left( \mathrm T(\Zp)_p, \Qp \right)$,  where $\mathrm T$ is a maximal split torus of $\mathrm G$, and study its  action on the ordinary tower and on $H^* \left( Y, \Qp \right)_{\pi}$. The ordinarity assumption is helpful to us as  Hida theory  gives that the ordinary part of the cohomology of arithmetic manifolds is controlled by the Iwasawa algebra of the diamond operators $\Z_p[[\mathrm T(\Z_p)]]$ (see \S \ref{sectionordinarycohomology} for precise statements).    We use the notation $Y_{1,1}$ below  for the arithmetic manifold attached to a congruence subgroup $\Gamma$ that at $p$ is a pro-$p$ Iwahori subgroup.

\begin{thm*}[Generation of cohomology over the bottom degree, Theorem \ref{cohomologygeneratedbybottomdegree}] \label{generationcohomologythmintro} Suppose that the complex $F_{\pi}^{\bullet} \left[ \frac{1}{p} \right]$ interpolating the $\pi$-primary eigenspace in ordinary cohomology is a quotient of $\Zp \left[ \left[ \mathrm T(\Zp)_p \right] \right] \left[ \frac{1}{p} \right]$ by a system of parameters of length $l_0$.

Then the action of the graded algebra $\bigwedge^* \Hom \left( \mathrm T(\Zp)_p, \Qp \right)$ of  derived diamond operators generates $H^* \left( Y_{1,1}, \Qp \right)_{\pi} $ over the bottom degree $H^{q_0} \left( Y_{1,1}, \Qp \right)_{\pi} $.
\end{thm*}
We prove this theorem in section \ref{generationofcohomologysection}.
Notice that this result makes no mention of Galois representations and in fact we do not need them in the picture yet, to prove theorem \ref{cohomologygeneratedbybottomdegree}.

Let now $\rho_{\pi}: \Gal_{\Q} \lra \check {\mathrm G} (\Qp)$ be the Galois representation conjecturally  associated to $\pi$  satisfying the usual compatibility relations (see conjecture \ref{Galoisrepexists1} for a detailed list of its properties).

We show in  Lemma \ref{keylemma} that there exists a natural morphism from the space of derived diamond operators to the dual of the  dual Selmer group for $\rho_{\pi}$: \[ \Phi:  \Hom \left( \mathrm T(\Zp)_p, \Qp \right) \lra H^1_f \left( \Z \left[ \frac{1}{S} \right], \left( \Lie \check {\mathrm G} \right)^* (1) \right)^{\vee}, \] and we conjecture the following relation between our derived diamond action and the map above.
\begin{conj*}[Dual Selmer action, conjecture \ref{derivedactionthruSelmerconj}] \label{factorsthrudualselmerintro} 

The derived diamond action of $\bigwedge^* \Hom \left( \mathrm T(\Zp)_p, \Qp \right)$ on $H^* \left( Y_{1,1}, \Qp \right)_{\pi} $ factors through  the map induced by $\Phi$:  \[ \bigwedge^* \Phi: \bigwedge^* \Hom \left( \mathrm T(\Zp)_p, \Qp \right) \lra \bigwedge^* H^1_f \left( \Z \left[ \frac{1}{S} \right], \left( \Lie \check {\mathrm G} \right)^* (1) \right)^{\vee}, \]  to yield an action of  the exterior algebra on the dual of the dual Selmer group above, which acts faithfully on $H^* \left( Y_{1,1}, \Qp \right)_{\pi} $.
\end{conj*}
We can prove this conjecture assuming that the ordinary deformation ring of $\rho$ is smooth (conjecture \ref{smoothnessconj}) and the crystalline deformation ring is trivial.

\begin{thm*}[Theorem \ref{dimensionality} and Theorem \ref{main}]  $ $

\begin{itemize}

\item  Suppose that the ring $R_{\rho}^{\ordi}$ classifying deformations of $\rho$ which are ordinary at $p$ and unramified outside $S$ is of the expected dimension $r-l_0$, and that the ring $R_{\rho}^{\crys}$ classifying deformations crystalline at $p$ and unramified outside $S$ is isomorphic to $\Qp$. Then the derived  diamond action of $\bigwedge^i \Hom \left( \mathrm T(\Zp)_p, \Qp \right)$ on $H^* ( Y_{1,1}, \Qp )_{\pi}$ factors through the map \[\bigwedge^* \Phi:  \bigwedge^* \Hom \left( \mathrm T(\Zp)_p, \Qp \right) \lra \bigwedge^* H^1_f \left( \Z \left[ \frac{1}{S} \right], \left( \Lie \check {\mathrm G} \right)^* (1) \right)^{\vee}. \]

\item Suppose further that the ring $R_{\rho}^{\ordi}$ classifying deformations of $\rho$ which are ordinary at $p$ and unramified outside $S$ is smooth, and that the ring $R_{\rho}^{\crys}$ classifying deformations crystalline at $p$ and unramified outside $S$ is isomorphic to $\Qp$. Then $H^* \left( Y_{1,1}, \Qp \right)_{\pi} $ as a module over the   graded $\Q_p$-algebra  $\bigwedge^* \Hom \left( \mathrm T(\Zp)_p, \Qp \right)$  is generated by the bottom degree $H^{q_0}\left( Y_{1,1}, \Qp \right)_{\pi}$.

\end{itemize}

Thus the conjecture is true  if  $R_{\rho}^{\ordi}$ is smooth of dimension $r-l_0$.
\end{thm*}

We further remark that it is possible by using patching methods to prove that  $R_{\rho}^{\crys}$ classifying deformations crystalline at $p$ and unramified outside $S$ is isomorphic to $\Qp$ under some assumptions (see \cite{KT}, \cite{CGH}).

 Theorems \ref{dimensionality} and \ref{main} in fact show that the conjectures in Galois cohomology of Appendix B, higher dimensional generalizations of the Leopoldt conjecture,   are equivalent to properties of derived Hecke actions at $p$ on the cohomology of symmetric manifolds conjectured above.  This equivalence has a certain charm as the derived Hecke action is of a topological origin (as we explain in  greater detail in \S \ref{sectionderivedactions}).   In the polarized case which corresponds to the defect $l_0=0$, $H^2(G_{F,S}, \ad \rho_\pi)=0$  is proved in \cite{Allen}, and in this case $H^1_f(\Z[{ 1 \over S}],\ad(\rho_\pi))=H^1_f(\Z[{ 1 \over S}],\ad(\rho_\pi)(1))=0$. The case when $l_0>0$ is to our knowledge completely open.

We end the introduction with some remarks about our methods and possible refinements and generalizations of the work done in this paper.

\begin{itemize}

\item  Our use of Hida's control theorems for ordinary cohomology (cf. Proposition  \ref{complexordinarytower}) replaces the use of Taylor-Wiles primes in Venkatesh's work. This is not surprising as it was Hida's work which inspired Wiles in his choice of what are known as Taylor-Wiles primes and the patching arguments that use augmentation of levels at Taylor-Wiles places.

\item  In \cite{HT}  a degree raising action of the dual Selmer group we consider is constructed on $H^*(\Gamma,\Q_p)_\pi$ under smoothness assumptions on the deformation ring $R_\pi$. We discuss in Appendix \ref{HaTho} the relation between the two papers which are dual in a certain sense.

\item In forthcoming work with P. Allen \cite{AKR}  we will explore how using Taylor-Wiles patching  (done in the context of Hida theory in \cite{KT}), following a strategy used in Lemm 8.25 of \cite{akshay},  one can relax most of the assumptions we need here (except genericity of $\pi$ at $p$), to prove that the derived diamond action at $p$  on the cohomology of $\Gamma$ factors through as in Conjecture \ref{derivedactionthruSelmerconj} to a  (dual of ) dual Selmer action without  invoking  the dimension conjecture Conjecture \ref{dimensionconj}. This improves
Theorem \ref{dimensionality} but comes at the expense of the proof depending on patching and hence becoming more elaborate.

\item The  construction of the derived Hecke action at $p$ on cohomology of arithmetic groups   in this paper   does not depend on any global conjectures, which is a key reason that we will be able to show using patching  in subsequent work \cite{AKR}  that this action factors through dual of a dual Selmer group. This is in contrast to Theorem 4.9 (2) of  \cite{HT} where to construct a degree lowering action of a dual torus  the authors needs to assume a  dimension conjecture analogous to
Conjecture \ref{dimensionconj}.  

\item  One would  expect  heuristically  that the  derived Hecke action on $H^*(\Gamma,\Q_p)_\pi$ factors through the map $\Phi$ to global Galois cohomology as $H^*(\Gamma,\Q_p)_\pi$ is a ``global'' object. This is what we prove in  forthcoming  work  via patching, although we cannot rule out the possibility that there is a more direct proof of such a result that circumvents the use of patching.

\item We would expect that one can extend the results of this paper to the case where we no longer assume $\pi$ is ordinary at places above $p$, and include the more general cases of $\pi$ that have a non-critical stabilization as in \cite{HT}.

\end{itemize}

\subsection{Notation} \label{notationsection}
We set up most of the notation that will be used in the rest of the paper.

We denote by $\A$ the adeles of $\Q$ and for every finite set of places $P$, we let $\A^P$ be the subgroup of elements which are $1$ at every place in $P$.

Let $\mathrm G$ be a split, connected, reductive $\Q$-group.  (The reader might prefer to  focus on the case $G=GL_n$ as all features of our work in this paper is captured by that case.) As explained in \cite{Conrad1,Conrad2}, there exists a finite set of primes $T$ such that $\mathrm G$ admits a smooth, split, connected, reductive model over $\Z \left[ \frac{1}{T} \right]$, still denoted by $\mathrm G$.
There exists a finite set $\widehat T \supset T$ such that $\mathrm G(\Z_q)$ is a hyperspecial subgroup of $\mathrm G(\Q_q)$ for every $q \not\in \widehat T$: we let $\mathrm G(\A)$ be the restricted tensor product of the $\left\{ \mathrm G(\Q_l) \right\}_l$ with respect to the $\left\{ \mathrm G(\Z_q) \right\}_{q \not\in \widehat T}$. Similar notations hold for $\mathrm G(\A^P)$ where $P$ is any finite set of places.

Let $p \not\in \widehat T$ be an odd prime such that $p > \left| \mu \left( \mathrm Z_{\mathrm G} \right) \right|$, the torsion subgroup of the center $\mathrm Z_{\mathrm G}$ of $\mathrm G$, and $p > |W|$, where $W$ is the Weyl group of $\mathrm G$.

We fix once and for all an embedding $\iota: \overline \Q \hookrightarrow \overline \Qp$. We also denote by $E$ our coefficient field: a large enough, finite extension of $\Qp$.

Inside the smooth, split, connected reductive $\Zp$-group $\mathrm G$, we fix a maximal split $\Zp$-torus $\mathrm T$, a Borel subgroup $\mathrm B$ containing $\mathrm T$ and we let $\mathrm U = \mathcal R_u(\mathrm B)$ be the unipotent radical of $\mathrm B$. Let $ r = \rk \mathrm T$ be the rank of the torus $\mathrm T$.

We denote by $X^*(\mathrm T)$ and by $X_*(\mathrm T)$ respectively the character and cocharacter lattices of $\mathrm T$. They are in perfect duality, with the pairing $\langle \, , \, \rangle$ being Weyl-equivariant. We let $\Phi \left( \mathrm G, \mathrm T \right) \subset X^* (\mathrm T)$ be the set of roots.

The choice of the Borel subgroup $\mathrm B$ fixes a choice of positive roots $\Phi^+ = \Phi^+ \left( \mathrm G, \mathrm T \right)$ and thus a basis of simple roots $\Delta \subset \Phi^+$. Consequently, we obtain a choice of negative roots $\Phi^-$ such that $\Phi = \Phi^+ \coprod \Phi^-$ - the corresponding opposite unipotent subgroup is denoted by $\mathrm U^-$.

For each root $\alpha \in \Phi(\mathrm G, \mathrm T)$, we denote by $\mathrm U_{\alpha}$ the corresponding root subgroup, a split unipotent $\Zp$-group scheme having dimension $d_{\alpha}$.

We fix a pinning, which is to say the data of isomorphisms of $\Zp$-group schemes \[ \begin{gathered} \phi: \mathbf G_m^r \stackrel{\sim}{\lra}  \mathrm T  \\ \phi_{\alpha}: \mathbf G_a^{d_{\alpha}} \stackrel{\sim}{\lra}  \mathrm U_{\alpha}  \quad \forall \alpha \in \Phi \left( \mathrm G, \mathrm T \right) \end{gathered} \] satisfying in particular the following property \begin{equation} \label{pinningproperty} t \phi_{\alpha} (u_{\alpha}) t^{-1} = \phi_{\alpha} \left( \alpha(t) u_{\alpha} \right) \end{equation} for every $\Zp$-algebra $R$, and every $t \in \mathrm T(R)$, $u_{\alpha} \in \mathbf G_a^{d_{\alpha}} (R)$.

We will abuse notation by not mentioning the pinning from now on - so for example $\mathrm U_{\alpha}(R)$ denotes the subgroup $\phi_{\alpha} \left( \mathbf G_a^{d_{\alpha}} (R) \right)$ for any $\Zp$-algebra $R$, while $\mathrm T \left( 1 + p^k \Zp \right)$ is the subgroup of $\mathrm T(\Zp) \stackrel{\phi^{-1}}{\lra} \left( \Zp^* \right)^r$ corresponding to $\left( 1 + p^k \Zp \right)^r$.
We will also denote by $\mathrm T(\Zp)_p = \mathrm T(1 + p \Zp)$ the pro-$p$-radical of $\mathrm T(\Zp)$.

Inside the cocharacter lattice of $\mathrm T$, we denote the dominant cone by \[ X_*^+ \left( \mathrm T \right) = \left\{ \lambda \in X_*(\mathrm T) \, | \, \langle \lambda, \alpha \rangle > 0 \, \forall \alpha \in \Phi^+ \right\}. \]
We choose $\lambda_1, \ldots, \lambda_r \in X_* (\mathrm T)$ generating a finite-index subcone of $X_*^+ (\mathrm T)$. If $\mathrm G$ is simply connected we can ensure that $\lambda_1, \ldots, \lambda_r$ generate precisely $X_*^+ (\mathrm T)$.

The choice of the uniformizer $p$ yields an isomorphism \[ X_* (\mathrm T) \stackrel{\sim}{\lra} \mathrm T(\Qp) / \mathrm T(\Zp) \qquad \lambda \mapsto \lambda(p) \bmod \mathrm T(\Zp), \] and we denote by $X_*^+(\Zp)$ the preimage of $X_*^+ \left( \mathrm T \right) \subset X_*(\mathrm T)$ under the quotient map $\mathrm T(\Qp) \twoheadrightarrow \mathrm T(\Qp) / \mathrm T(\Zp)$.

We denote by $\check {\mathrm G}$ the algebraic group dual to $\mathrm G$. We again fix a smooth, split, connected, reductive $\Zp$-model, still denoted by $\check {\mathrm G}$. This contains a conjugacy class of maximal $\Zp$-split tori dual to $\mathrm T$ - we will later choose an element of this conjugacy class, denoted by $\check {\mathrm T}$.

The construction of the dual torus yields natural isomorphisms $X^* \left( \mathrm T \right) \cong X_* \left( \check {\mathrm T} \right)$ and $X_* \left( \mathrm T \right) \cong X^* \left( \check {\mathrm T} \right)$. We let $\check \Phi \left( \check {\mathrm G}, \check { \mathrm T} \right)$ be the roots of the dual root system.
Our choice of positive roots $\Phi^+(\mathrm G, \mathrm T)$ yields positive coroots for $\left( \mathrm G, \mathrm T \right)$, and thus a choice of positive roots $\check \Phi^+ \left( \check {\mathrm G}, \check { \mathrm T} \right)$ which correspond to the choice of a Borel subgroup $\check {\mathrm B}$ containing $\check{\mathrm T}$. Let $\check{\mathrm U}$ be the unipotent radical of $\check{\mathrm B}$.

We also denote the $\Qp$-Lie algebras of these groups by $\Lie \check{\mathrm G}$, $\Lie \check{\mathrm B}$, $\Lie \check{\mathrm T}$ and $\Lie \check{\mathrm U}$.

Given a ring $R$ and a topological group $G$ and a subsemigroup $H$ having the property that each double coset $HgH$ is a \emph{finite} union of left (equivalently, right) $H$-cosets, we denote by $\mathcal H_R \left( G, H \right)$ the Hecke algebra of bi-$H$-invariant functions $G \stackrel{f}{\lra} R$ supported on finitely many double cosets.% The multiplication operation is convolution.

For every field $F$ with a fixed algebraic closure $\overline F$, we denote by $\Gal_F = \Gal \left( \overline F / F \right)$ its absolute Galois group.
We fix embeddings $\overline {\Q} \hookrightarrow \overline {\Q_l}$ for every prime $l$, and therefore we get injections $\Gal_{\Q_l} \hookrightarrow \Gal_{\Q}$.  We denote by $I_{\Q_l}$ the inertia subgroup of $\Gal_{\Q_l}$, and also use $G_l$ and $I_l$ to denote the corresponding (conjugacy class of)
 subgroups of $\Gal_{\Q}$. Given a group homomorphism $\Gal_{\Q} \lra H$, we denote by $\sigma_l$ the restriction of $\sigma$ to $\Gal_{\Q_l}$.

Let $\rho: \Gal_{\Q} \lra \Gl{}(V)$ be a Galois representation. We denote by $C^{\bullet} \left( \Q, V \right)$ the complex of continuous inhomogenous cochains of $\Gal_{\Q}$ valued in $V$, by $Z^{\bullet} \left( \Q, V \right)$ the cocycles and by $H^* \left( \Q, V \right)$ the continuous Galois cohomology. Similar notations hold for the subgroup $\Gal_{\Q_l}$.

We let $\omega$ be the cyclotomic character, which by convention has Hodge-Tate weight $-1$. If $V$ is a ($p$-adic or $\bmod p$) Galois representations, we let $V^*$ be its contragredient, and $V(n)$ be its Tate twist by the $n$-th power of the cyclotomic character.

Given a local, complete, Noetherian $\Zp$-algebra $A$ with residue field $k$, we let $\mathcal C_A$ be the category of local, complete, Noetherian $A$-algebras with residue field $k$ and $\mathcal C_A^f$ be the subcategory of Artinian $A$-algebras. If $R_1, R_2 \in \mathcal C_A$, the homomorphisms of $A$-algebras respecting the augmentation map to $k$ are denoted by $\Hom_* (R_1, R_2)$.
We also denote by $\mathcal D(R)$ the derived category of $\Zp$-modules with a continuous actions of $R$.

Let $K_{\infty}$ be a maximal compact subgroup of $\mathrm G(\R)$ and let $Z = \mathrm Z_{\mathrm G}(\R)^{\circ}$ be the identity component of the real points of the center of $\mathrm G$.
We will consider different levels for our arithmetic manifolds, depending on open compact subgroups $K = \prod_{q < \infty} K_q \subset \mathrm G(\Af)$.
We fix $K_0 = \prod_{q < \infty} K_q$ to be our base level, and we assume that \begin{itemize}
\item $K_0$ is a good subgroup in the sense of \cite{KT}, section 6.1 (which is to say, $K_q \subset \mathrm G(\Z_q)$ for all $q \not\in \widehat T$ and $gK_0g^{-1} \cap \mathrm G(\Q)$ is neat for all $g \in \mathrm G(\Af)$).\footnote{Notice that this condition can always be arranged at the cost of possibly enlarging $\widehat T$ and shrinking $K_q$ for some $q \in \widehat T$.}
\item $K_p$ is a Iwahori subgroup of $\mathrm G(\Qp)$, which is to say that $K_p$ is the preimage of $\mathrm B(\Fp)$ under the reduction map $\mathrm G(\Zp) \lra \mathrm G(\Fp)$.
\end{itemize}
We also let $K_1 = \prod_{q < \infty} K'_q$ be the subgroup of $K_0$ having pro-$p$-Iwahori level at $p$, which is to say that $K'_p$ is the preimage of $\mathrm U(\Fp)$ under the reduction map $\mathrm G(\Zp) \lra \mathrm G(\Fp)$.

We notice that as explained in \cite{KT}, section 6.1, any open compact subgroup of $K_0$ is automatically good.

In this paper, we will consider good subgroups $K = \prod_q K_q$ of $K_0$. As before, if $P$ is a finite set of primes we denote $K^P = \prod_{q \not\in P} K_q$, an open compact subgroup of $\mathrm G \left( \A^{P, \infty} \right)$. We also let $K_P = \prod_{q \in P} K_q$.

For any good subgroup $K$, we let \[ Y(K) = \mathrm G(\Q) \backslash \mathrm G(\A) / \left( K_{\infty} \cdot Z^{\circ} \right) K = \mathrm G(\Q) \backslash \mathrm G(\Af) \times \mathrm G(\R) / \left( K_{\infty} \cdot Z^{\circ} \right) K \] be the associated arithmetic manifold.
We will denote $Y(K_0) = Y_0$. Let $d$ be the dimension of $Y(K)$ as a real manifold - it does not depend on the good subgroup $K$ under consideration.
The \emph{defect} of $\mathrm G$ is $l_0 = \rk \mathrm G(\R) - \rk K_{\infty}$ and we also denote $q_0 = \frac{d - l_0}{2}$.

We will study cohomological cuspidal automorphic representations $\pi$ of $\mathrm G(\A)$, and we denote by $E_\pi$ the  Hecke field  attached to such $\pi$.  The field $E_\pi$  is known to be  a number field, cf.   \cite{Clozel} and  \cite{BG} for more details. 
We fix an embedding $\iota:E_\pi \rightarrow \overline\Q_p$, and a finite extension $E$ of $\Q_p$ (which will be our field of coefficients)  which contains 
  $\iota(E_{\pi})$.

\subsection{Outline of the paper}

In section \ref{sectionordinarycohomology} we recall several results from Hida's theory of ordinary parts. We also describe a general Hida theory for a reductive split group $\mathrm G$ - this is certainly known to experts but we were unable to find a reference.

Section \ref{sectionderivedactions} is devoted to introducing the general setup of derived Hecke algebras, following \cite{akshay}. We also introduce the derived diamond operators and relate them to the general machinery of derived Hecke algebras. Finally, we explain how these operators acts on the cohomology of arithmetic manifolds and their interaction with Hida's Iwasawa algebra introduced in section \ref{sectionordinarycohomology}.

In section \ref{automorphicsection}   we collect results on the automorphic representations appearing in the cohomology of the arithmetic manifolds we are interested in. 

In section \ref{generationofcohomologysection} we prove  (cf. Theorem  \ref{cohomologygeneratedbybottomdegree}) that the derived diamond operators acting on $H^{q_0} \left( Y_{1,1}, \Qp \right)_{\pi}$ generate the whole tempered cohomology range $H^* \left( Y_{1,1}, \Qp \right)_{\pi}$.

We introduce the setup for deformations valued into the dual group $\check {\mathrm G}$ in section \ref{deformationsection}, following \cite{Patrikis}. We also recall the definitions of Selmer and dual Selmer group, as well as some results on their dimensions. We will apply these results both for deformations of a $\bmod p$ representation $\overline \rho$ and for deformations of a $\Qp$-valued representation $\rho$.

We then bring Galois representations into the picture in section \ref{Galoisrepforautrep}: letting $\rho_\pi$ be the $\Qp$-valued representation corresponding to $\pi$ as in \cite{HLTT} and \cite{Scholze}, we describe in Lemma \ref{keylemma} the map $\Phi$ from the space of derived diamond operators into the dual of a dual Selmer group for $\rho$.
Then, we prove in Theorem \ref{dimensionality} and Theorem \ref{main} that, assuming the dimension conjecture \ref{dimensionconjectureforring} for $R^{\ordi}_\rho$, the derived diamond action on $H^* \left( Y_{1,1}, \Qp \right)_{\pi}$ factors through the  map $\Phi$, and assuming the smoothness conjecture for $R^{\ordi}_\rho$, that $H^* \left( Y_{1,1}, \Qp \right)_{\pi}$ is generated by the bottom degree under the derived Hecke/diamond action.

\subsection{Acknowledgements}

 We would  like to thank
Patrick Allen and Najmuddin Fakhruddin for their help  with some of the work done in the paper.  We would like to thank Michel Harris and Akshay Venkatesh for their encouragement and interest.

\section{Recollections on Hida theory and ordinary cohomology} \label{sectionordinarycohomology}
In this section we recall the setup of Hida theory and in particular Hida's theory of ordinary parts. We follow closely \cite{KT}, in particular section 6.

We let $C_{\bullet}$ be the complex of singular chains with $\Z$-coefficients valued in $\mathrm G(\Af) \times \mathrm G(\R) / \left( K_{\infty} \cdot Z^{\circ} \right)$, and define \[ \begin{split} C_{\bullet} \left( Y(K), \Z \right) = C_{\bullet} \otimes_{\Z \left[ \mathrm G(\Q) \times K \right]} \Z, \\ C^{\bullet} \left( Y(K), \Z \right) = \Hom_{\mathrm G(\Q) \times K} \left( C_{\bullet} , \Z \right) \end{split} \] for every good subgroup $K$ of $K_0$.
We have then as in proposition 6.2 of \cite{KT} that $H^* (Y(K), \Z) \cong H^* \left( C^{\bullet}(Y(K), \Z) \right)$. %\footnote{Technically, for this isomorphism to hold even for non-neat subgroups $K$, we need to interpret $H^*(Y(K), \Z)$ as orbifold cohomology.}

We can define Hecke operators at the level of complexes, as in \cite{KT}, section 2. Let $U = \prod_q U_q$ and $V = \prod_q V_q$ two good subgroups such that $U_q = V_q$ for all $q$'s outside of a finite set of primes $S(U,V)$ depending on $U$ and $V$. Then for each $g \in \mathrm G \left( \A^{S(U,V)} \right)$ we have a map \[ \left[ UgV \right]: C^{\bullet} \left( Y(U), \Z \right) \lra C^{\bullet} \left( Y(V), \Z \right) \] defined as follows: we fix a finite decomposition $U g V = \coprod_i g_i V$, then \[ \left( \left[ UgV \right] \phi \right) (\sigma) = \sum_i g_i \phi \left( g_i^{-1} \sigma \right) \qquad \forall \sigma \in C_{\bullet}. \]
\begin{rem} Since we are taking trivial coefficients, the elements $g_i$ act trivially, but we still keep track of them in the formula for future work on the case of non-trivial coefficients.
\end{rem}
Taking $U=V$, we get an action of $\mathcal H_{\Z} \left( \mathrm G(\Af), U \right)$ on $C^{\bullet} \left( Y(U), \Z \right)$.
    
This global Hecke algebra is the restricted tensor product of the local Hecke algebras $\mathcal H_{\Z} \left( \mathrm G(\Q_q), U_q \right)$ if $U = \prod_q U_q$ is a product of local factors, and so we also have actions of these local Hecke algebras on $C^{\bullet} \left( Y(U), \Z \right)$.

We now to increase level at $p$.

We have a canonical choice of uniformizer in $p$, and then the Hecke operators $\{ T_p^i = \mathrm G(\Zp) \lambda_i(p) \mathrm G(\Zp) \}_{i=1, \ldots, r}$ generate the spherical Hecke algebra $\mathcal H_{\Z} \left( \mathrm G(\Qp), \mathrm G(\Zp) \right)$ by the classical Satake homomorphism. Since this is a commutative algebra, the operators $T_p^i$ commute as $i$ varies.

For each $1 \le b \le c$ we define $I(b,c) \subset \mathrm G(\Zp)$ to be the subgroup of elements $g \in \mathrm G(\Zp)$ such that \begin{itemize}
    \item under the projection $\bmod p^b$, $g$ lands into the central extension $\mathrm Z_{\mathrm G}(\Z / p^b \Z) \times \mathrm U(\Z / p^b \Z)$ of the unipotent subgroup $\mathrm U(\Z / p^b \Z)$.
    \item under the projection $\bmod p^c$, $g$ lands into the Borel subgroup $\mathrm B(\Z / p^c \Z)$.
\end{itemize}
%\begin{rem}In fact, in \cite{KT} the second condition is relaxed by allowing a central extension. We are assuming that $\mathrm G$ is semisimple, so that condition disappears.
%\end{rem}
The largest subgroup in this family is $I(1,1)$, which is the central extension of the pro-$p$-Iwahori subgroup corresponding to the Borel $\mathrm B$.

Let also denote by $K(b,c)$ the subgroup of $K_0$ where we change level only at $p$, replacing $K_p$ with $I(b,c)$. The corresponding arithmetic manifold is $Y_{b,c} = Y(K(b,c))$.

Notice that since $I(b,c)$ is normal in $I(1,c)$, every $\alpha \in I(1,c)$ gives rise to a \emph{diamond operator} $\left[ I(b,c) \alpha I(b,c) \right] \in \mathcal H_{\Z} \left( \mathrm G(\Qp), I(b,c) \right)$. We sometimes denote this by $\langle \alpha \rangle$.
On the other hand, we also have operators $U_p^i = \left[ I(b,c) \lambda_i(p) I(b,c) \right] \in \mathcal H_{\Z} \left( \mathrm G(\Qp), I(b,c) \right)$. Unlike the operators $T_p^i$ at spherical level, these depend on the choice of our uniformizer $p$.\footnote{We are making a slight abuse of notation here, since $\langle \alpha \rangle$ and $U_p^i$ do not mention the subgroup $I(b,c)$ that we are looking at - however, this abuse of notation is fairly minor because of the lemma \ref{Heckeincohomology}.}

We will denote by $\langle \alpha \rangle$ and $U_p^i$ both the elements of the Hecke algebra of $I(b,c)$ as well as the operators they induce on $C^{\bullet} \left( Y_{b,c}, \Z \right)$ and $H^* \left( Y_{b,c} , \Z \right)$.
\begin{lem}[Lemma 6.5 in \cite{KT}] \label{Heckeincohomology} Let $A$ be a commutative ring with a trivial $\mathrm G(\Qp)$-action. \begin{enumerate}
    \item The elements $U_p^i$ and $\langle \alpha \rangle$ commute inside $\mathcal H_A \left( \mathrm G(\Qp), I(b,c) \right)$, and thus the corresponding operators on $C^{\bullet} \left( Y_{b,c}, A \right)$ and $H^* \left( Y_{b,c} , A \right)$ commute as well.
    \item Suppose $b' \ge b$ and $c' \ge c$, so that $I(b',c') \subset I(b,c)$.
    Then the operators $U_p^i$ and $\langle \alpha \rangle$ commute with the canonical pullback maps \[ C^{\bullet} \left( Y_{b,c}, A \right) \lra C^{\bullet} \left( Y_{b',c'}, A \right) \textnormal{ and } H^* \left( Y_{b,c} , A \right) \lra H^* \left( Y_{b',c'} , A \right). \]
    \item Suppose $c -1 \ge b$, so that $I(b,c-1) \supset I(b,c)$. Then the operators $U_p^i$ on $C^{\bullet} \left( Y_{b,c}, A \right)$ take values inside the subcomplex $C^{\bullet} \left( Y_{b,c-1}, A \right)$, and thus the map on cohomology factors as \[ H^* \left( Y_{b,c} , A \right) \stackrel{\left[ I(b,c) \lambda_p^i I(b,c-1) \right]}{\lra} H^* \left( Y_{b,c-1} , A \right) \lra H^* \left( Y_{b,c}, A \right), \] where the last map is the canonical pullback.
\end{enumerate}
\end{lem}
\begin{proof} This is a standard result, whose proof relies on Hida's proposition 2.2 in \cite{Hida2}. That result is only proved in the setup of general (and special) linear groups, so we spell out the details of the case of a reductive group here.

We fix an ordering of the positive roots in nondecreasing order of height (see for instance Conrad \cite{Conrad1} proposition 1.4.11), so that the multiplication map \begin{equation} \label{algebraiciso} \prod_{\alpha \in \Phi^+} \mathrm U_{- \alpha} \times \mathrm T \times \prod_{\alpha \in \Phi^+} \mathrm U_{\alpha} \lra \mathrm G \end{equation} is an open embedding of $\Zp$-group schemes.

The preimage $X_*^+(\Zp)$ of the dominant cone under the quotient map $\mathrm T(\Qp) \twoheadrightarrow \mathrm T(\Qp) / \mathrm T(\Zp)$ coincides with $N_{\mathrm T(\Qp)} \mathrm B(\Zp)$, the elements of $\mathrm T(\Qp)$ normalizing $\mathrm B(\Zp)$.
This is a subsemigroup of $\mathrm T(\Qp)$ consisting of the torus elements that shrink (or leave invariant) $\mathrm U (\Zp)$. Consequently, these are the same elements that enlarge (or leave invariant) the opposite unipotent subgroup $\mathrm U^- (\Zp)$.
In particular, if $\xi \in X_*^+(\Zp)$ is such that $\xi \mathrm U^- \left( \Zp \right) \xi^{-1} \subset \mathrm U^- \left( \Zp \right)$, then in fact $\xi \in \mathrm T(\Zp)$ and $\xi \mathrm U^- \left( \Zp \right) \xi^{-1} = \mathrm U^- \left( \Zp \right)$.

Let now $C$ be an open compact subgroup of $\mathrm G(\Qp)$ with $I(b',c') \subset C \subset I(b,c)$.
We can consider $\Delta = C X_*^+(\Zp) C \subset \mathrm G(\Qp)$, and in fact we look at the functions of $\mathcal H_{\Z} \left( \mathrm G(\Qp), C \right)$ supported on $\Delta$: we aim to understand this submodule.
\begin{prop} $\Delta$ is a commutative subalgebra of $\mathcal H_{\Z} \left( \mathrm G(\Qp), C \right)$.
\end{prop}
\begin{proof}
Let $\xi = \tau \lambda(\varpi) \in X_*^+(\Zp)$ be the generic element of this integral dominant cone, with $\tau \in \mathrm T(\Zp)$ and $\lambda \in X_*^+$.
Since $C$ is open compact in $\mathrm G(\Qp)$, $C \cap \xi C \xi^{-1}$ has finite index in $C$, and thus we can choose a finite set of representatives $\eta \in X(\xi)$ to get \[ C = \coprod_{\eta \in X(\xi)} \eta \left( C \cap \xi C \xi^{-1} \right). \]
Right multiplying by $\xi C \xi^{-1}$ yields \[ C \xi C \xi^{-1} = \coprod_{\eta \in X(\xi)} \eta \xi C \xi^{-1} \Rightarrow C \xi C = \coprod_{\eta \in X(\xi)} \eta \xi C. \]
We now show we have a canonical, nice way to choose the set of coset representatives $X(\xi)$.
Let $c \in C$, then by the morphism in formula \ref{algebraiciso}, we have \[ c = \left( \prod_{\alpha \in \Phi^+} u_{- \alpha} \right) \tau \left( \prod_{\alpha \in \Phi^+} u_{\alpha} \right) \] where for each $\alpha \in \Phi^+$ \[ u_{\alpha} \in \mathrm U_{\alpha} (\Zp)  \textnormal{ and } u_{- \alpha} \in \mathrm U_{- \alpha} \left( p^c \Zp \right), \] and $t \in \mathrm T \left(1 + p^b \Zp \right)$.

Suppose $c$ also belongs to $ \xi C \xi^{-1}$, then we must have \[ C \ni \xi^{-1} c \xi = \left( \prod_{\alpha \in \Phi^+} \xi^{-1} u_{- \alpha} \xi \right) \tau \left( \prod_{\alpha \in \Phi^+} \xi^{-1} u_{\alpha} \xi \right). \]
The $\mathrm T(\Zp)$-factor in $\xi$, $\tau$, does not change the valuation of the $u_{\alpha}$'s and $u_{- \alpha}$'s upon conjugation, so it is irrelevant when discussing the condition that $\xi^{-1} c \xi \in C$, and we assume $\xi = \lambda(p)$.

Then, \[  \lambda(p)^{-1} c \lambda(p) = \left( \prod_{\alpha \in \Phi^+} \lambda(p)^{-1} u_{- \alpha} \lambda(p) \right) \tau \left( \prod_{\alpha \in \Phi^+} \lambda(p)^{-1} u_{\alpha} \lambda(p) \right) = \] \[ = \left( \prod_{\alpha \in \Phi^+} \left( p^{\langle - \lambda, - \alpha \rangle} \cdot u_{- \alpha} \right) \right) \tau \left( \prod_{\alpha \in \Phi^+} \left( p^{ \langle - \lambda, \alpha \rangle} \cdot u_{\alpha} \right) \right). \]
Since $\lambda \in X_*^+$, we have $\langle - \lambda, - \alpha \rangle = \langle \lambda, \alpha \rangle \ge 0$ for all $\alpha \in \Phi^+$, and $ \left( p^{ \langle - \lambda, \alpha \rangle} \cdot u_{- \alpha} \right) \in \mathrm U_{- \alpha} \left( p^c \Zp \right) $.
On the other hand, $\langle - \lambda, \alpha \rangle \le 0$, and thus to ensure that $\left( p^{ \langle - \lambda, \alpha \rangle} \cdot u_{\alpha} \right)$ is in $U_{\alpha} (\Zp)$, we have to require $u_{\alpha} \in U_{\alpha} \left( p^{\langle \lambda, \alpha \rangle} \Zp \right)$ for all $\alpha \in \Phi^+$.

We conclude that \[ C / \left( C \cap \xi C \xi^{-1} \right) \cong \prod_{\alpha \in \Phi^+} \left( U_{\alpha}(\Zp) / U_{\alpha} \left( p^{\langle \lambda, \alpha \rangle} \Zp \right) \right). \]
We choose then coset representatives for each $U_{\alpha} (\Zp) / U_{\alpha} \left( p^{\langle \lambda, \alpha \rangle} \Zp \right)$, and obtain as products the coset representatives for $X(\xi) = X(\lambda)$: \[ X(\lambda) = \prod_{\alpha \in \Phi^+} \left( U_{\alpha}(\Zp) / U_{\alpha} \left( p^{\langle \lambda, \alpha \rangle} \Zp \right) \right). \]
\begin{rem}
Notice in particular that these coset representatives \emph{do not depend on $C$, $b, b', c$ or $c'$}.
\end{rem}
Recall that $d_{\alpha} = \dim U_{\alpha}$, its dimension as an algebraic group, and define \[ \deg \xi = \deg \lambda = - \sum_{\alpha \in \Phi^+} \langle \lambda, \alpha \rangle d_{\alpha}. \]
The previous discussion shows that \[ \left| C \xi C / C \right| =  \left[ C : C \cap \xi C \xi^{-1} \right] = p^{\deg \xi}. \]
From our definition it is also clear that $\deg (\xi \xi') = \deg \xi + \deg \xi'$ and thus that \[ \left| C \xi \xi' C / C \right| = \left| C \xi C / C \right| \cdot \left| C \xi' C / C \right|. \]
It is obvious that $C \xi C \xi' C \supset C \xi \xi' C$, and to show equality it suffices then to show that $| C \xi C \xi' C / C| \le |C \xi C / C| \cdot |C \xi' C / C | $.
However, this last inequality is clear, since \[ C \xi C \xi' C = \coprod_{\nu \in X(\xi} \nu \xi C \xi' C = \coprod_{\nu \in X(\xi)} \coprod_{\nu' \in X(\xi')} \nu \xi \nu' \xi' C. \]
This shows that \[ C \xi C \xi' C = C \xi \xi' C = C \xi' \xi C = C \xi' C \xi C \qquad \forall \xi, \xi' \in X_*^+(\Zp), \] proving the proposition.
\end{proof}
The first claim of the lemma follows now immediately, since $U_p^i$ and $\langle \alpha \rangle$ are both elements of this subalgebra in the case where $C = I(b,c)$.

The second part of the lemma follows from the remark above that the coset representatives used to define the Hecke action do not depend on $b$ or $c$. We can thus use the same coset representatives to define the Hecke action on $C^{\bullet} \left( Y_{b,c}, A \right)$ and $ C^{\bullet} \left( Y_{b',c'}, A \right)$, which proves commutativity with the pullback map. This is the same argument as in \cite{Geraghty}, lemma 2.3.3.

The third claim of the lemma follows again by showing that \[ I(b,c) \lambda_i(p) I(b,c) = I(b,c-1) \lambda_i(p) I(b,c), \] which follows as before and as in the proof of lemma 2.5.2 of \cite{Geraghty} from the fact that the aforementioned coset representatives can be chosen independently of $b$ and $c$.
\end{proof}

\subsection{Ordinary cohomology}
We describe the ordinary summand of arithmetic cohomology and its properties.

The cover $Y_{b,c} \twoheadrightarrow Y_{1,c}$ is Galois, with Galois group \[ K(1,c) / K(b,c) \cong I(1,c) / I(b,c) \cong \mathrm T(\Zp)_p / \mathrm T(1 + p^b \Zp)_p \cong \mathrm T(\Z / p^b \Z)_p =:T_b. \]
Here we use crucially that $I(b,c)$ intersects trivially with $\mathrm Z_{\mathrm G}(\Q)$ (thanks to our assumption on the center not having any $p$-torsion), and thus $T_b$ acts freely on $Y_{b,c}$.

Let us denote $\Lambda_b = \Zp \left[ T_b \right]$, $\Lambda = \varprojlim_b \Lambda_b \cong \Zp \left[ \left[ \mathrm T(\Zp)_p \right] \right]$, $\Lambda_{\Qp} = \Lambda \otimes_{\Zp} \Qp = \Lambda \left[ \frac{1}{p} \right]$ and more generally $\Lambda_F = \Lambda \otimes_{\Zp} F$ for any finite extension $F$ of $\Qp$.

By lemma \ref{Heckeincohomology}, for each $c' \ge c, b' \ge b$ the natural pullback map \[ H^i \left( Y_{b,c}, \Z / p^n \Z \right) \lra H^i \left( Y_{b',c'}, \Z / p^n \Z \right) \] is equivariant for the action of each $U_p^i$.
Denote then $U_p = \prod_{i=1}^r U_p^i$.

\begin{defn} The ordinary part of $H^i \left( Y_{b,c}, \Z / p^n \Z \right)$ is the $\Z / p^n \Z$-submodule where $U_p$ acts invertibly. We denote it by $H^i \left( Y_{b,c}, \Z / p^n \Z \right)_{\ordi}$.
\end{defn}
See section 2.4 of \cite{KT} for a discussion on ordinary parts with respect to any endomorphism.
In particular, thanks to lemma \ref{Heckeincohomology}, we have that the pullback map preserves ordinary part: \[ H^i \left( Y_{b,c}, \Z / p^n \Z \right)_{\ordi} \lra H^i \left( Y_{b',c'}, \Z / p^n \Z \right)_{\ordi}. \]

We now prove a result which relates ordinary part in towers, and will be crucial in the next section.
\begin{lem} \label{ordinaryisom} For all $ c \ge b \ge 1$ and all $n \ge 1$, the pullback map \[ \pi: H^* \left( Y_{b,b}, \Z / p^n \Z \right)_{\ordi} \lra H^* \left( Y_{b,c}, \Z / p^n \Z \right)_{\ordi} \] is an isomorphism.
\end{lem}
\begin{proof}
By lemma \ref{Heckeincohomology}, the action of $U_p$ commutes with the pullback map, even at the level of complexes, so we have a commutative diagram \begin{displaymath} \xymatrix{ H^* \left( Y_{b,b}, \Z / p^n \Z \right) \ar[r]^{U_p} \ar[d]_{\pi} & H^* \left( Y_{b,b}, \Z / p^n \Z \right) \ar[d]^{\pi} \\  H^* \left( Y_{b,c}, \Z / p^n \Z \right) \ar[r]_{U_p} & H^* \left( Y_{b,c}, \Z / p^n \Z \right)  } \end{displaymath}
On the other hand, the last part of the lemma says that in fact the operator $U_p$ on $H^* \left( Y_{b,c}, \Z / p^n \Z \right)$ factors as $\pi \circ \left[ I(b,c) \prod_i \lambda_i(p) I(b,c-1) \right]$, so our diagram becomes \begin{displaymath} \xymatrix{ H^* \left( Y_{b,c-1}, \Z / p^n \Z \right) \ar[r]^{U_p} \ar[d]_{\pi} & H^* \left( Y_{b,c-1}, \Z / p^n \Z \right) \ar[d]^{\pi} \\  H^* \left( Y_{b,c}, \Z / p^n \Z \right) \ar[r]_{U_p} \ar[ur]_{r} & H^* \left( Y_{b,c}, \Z / p^n \Z \right)  } \end{displaymath} where we denote $r = \left[ I(b,c) \prod_i \lambda_i(p) I(b,c-1) \right]$.
The existence of the map $r$ implies then an isomorphism between the submodule of ordinary parts, as in lemma 6.10 of \cite{KT}.
\end{proof}

We now want to interpolate the ordinary subspaces in the tower $\left\{ H^* \left( Y_{c,c} , \Zp \right)_{\ordi} \right\}_{c \ge 1}$. It turns out it is more convenient to interpolate ordinary subspaces in homology, and then dualize if necessary. The following result is proposition 6.6 of \cite{KT}.
\begin{prop} \label{complexordinarytower} There exists a minimal complex $F_{\infty}^{\bullet}$ of $\Lambda$-modules, together with morphisms for every $c \ge 1$:
\[ \begin{split} g_c: F_{\infty}^{\bullet} \otimes_{\Lambda} \Lambda_c/p^c \lra C_{d - \bullet} \left( K(c,c), \Z / p^c \Z \right), \\ g_c': C_{d - \bullet} \left( K(c,c), \Z / p^c \Z \right) \lra F_{\infty}^{\bullet} \otimes_{\Lambda} \Lambda_c / p^c \end{split} \]
satisfying the following conditions: \begin{enumerate}
    \item $g_c' g_c = 1$, while $g_c g_c'$ is an idempotent in $\End_{\mathcal D \left( \Lambda_c/p^c \right)} \left( C_{d - \bullet} \left( K(c,c), \Z / p^c \right) \right)$.
    \item For each $c \ge 1$, the following diagram is commutative up to chain homotopy \begin{displaymath} \xymatrix{ F_{\infty}^{\bullet} \otimes_{\Lambda} \Lambda_{c+1} / p^{c+1} \ar[r]^{g_{c+1}} \ar[d] & C_{d - \bullet} \left( K(c+1,c+1), \Z / p^{c+1} \Z \right) \ar[d]^{\pi} \ar[r]^{g_{c+1}'} & F_{\infty}^{\bullet} \otimes_{\Lambda} \Lambda_{c+1} / p^{c+1} \ar[d] \\  F_{\infty}^{\bullet} \otimes_{\Lambda} \Lambda_c / p^c \ar[r]_{g_c} & C_{d - \bullet} \left( K(c,c), \Z / p^c \Z \right) \ar[r]_{g_c'} & F_{\infty}^{\bullet} \otimes_{\Lambda} \Lambda_c / p^c  } \end{displaymath} where the middle vertical map is the natural pushforward.
    \item For each $c \ge 1$, the map induced by $g_c g_c'$ on homology is the natural projection onto the ordinary subspace along the non-ordinary one: \[ H_* \left( g_c g_c' \right): H_{d - *} \left( Y_{c,c}, \Z / p^c \Z \right) \lra H_{d - *} \left( Y_{c,c} , \Z / p^c \Z \right)_{\ordi}. \]
\end{enumerate}
The minimal complex $F_{\infty}^{\bullet}$ is uniquely determined in $\mathcal D \left( \Lambda \right)$ by these properties, up to unique isomorphisms.

Finally, there is a unique homomorphism lifting the Hecke action away from $p$ to this complex: \[ \Phi: \mathcal H \left( \mathrm G \left( \A^{\infty, p} \right), K^p_0 \right) \lra \End_{\mathcal D \left( \Lambda \right)} \left( F_{\infty}^{\bullet} \right) \] such that for all $f \in \mathcal H \left( \mathrm G \left( \A^{\infty, p} \right), K^p_0 \right)$ one has \[ \Phi(f) \otimes_{\Lambda} \Lambda_c/p^c = g_c' f g_c \textnormal{ as endomorphisms of } F_{\infty}^{\bullet} \otimes_{\Lambda} \Lambda_c / p^c. \]
\end{prop}
\begin{proof}
This is proven exactly like in \cite{KT} - for sake of completeness we recall the main steps.

The crucial tool is proposition 2.15 in \cite{KT}, which is an abstract statement about glueing complexes. To apply this proposition, one needs to check that the natural pushforward maps in ordinary homology \[ H^* \left( C_{d - \bullet} \left( K(c+1, c+1), \Z / p^{c+1} \Z \right) \otimes_{\Lambda} \Lambda_c / p^c \right)_{\ordi} \lra H^* \left( C_{d - \bullet} \left( K(c, c), \Z / p^c \Z \right) \right)_{\ordi} \] are isomorphisms.

However, one has \[ C_{d - \bullet} \left( K(c+1, c+1), \Z / p^{c+1} \Z \right) \otimes_{\Lambda} \Lambda_c / p^c \cong C_{d - \bullet} \left( K(c+1, c+1), \Z / p^{c+1} \Z \right) \otimes_{\Zp \left[ I(c,c+1)/I(c+1,c+1) \right] } \Zp / p^c \Zp \] and now as in lemma 6.9 of \cite{KT} - which in our setup can be applied since $I(c,c+1)$ is pro-$p$ and thus has trivial intersection with $Z_{\mathrm G}(\Q)$ because the center has no $p$-torsion - the latter complex is isomorphic to \[ C_{d - \bullet} \left( K(c, c+1), \Z / p^c \Z \right). \]
It remains to check that the pushforward map in homology \[ H_* \left( K(c, c+1), \Z / p^c \Z \right)_{\ordi} \lra H_* \left( K(c, c), \Z / p^c \Z \right)_{\ordi} \] is an isomorphism, but this is the homological equivalent of our lemma \ref{ordinaryisom} and of proposition 6.10 in \cite{KT}, and is proven in the exact same way.
\end{proof}
\begin{cor} There is an Hecke-equivariant isomorphism (away from $p$): \[ H^* \left( F_{\infty}^{\bullet} \right) \cong \varprojlim_c H_{d - *} \left( Y_{c,c}, \Zp \right)_{\ordi}. \]
\end{cor}
\begin{proof} Just like in \cite{KT}, this is a consequence of lemma 6.13 of loc. cit.
Notice also that \[ \varprojlim_c H_{d - *} \left( Y_{c,c}, \Zp / p^c \Zp \right)_{\ordi} \cong \varprojlim_c H_{d - *} \left( Y_{c,c}, \Zp \right)_{\ordi}. \]
\end{proof}

\section{Derived Hecke actions on the cohomology of arithmetic manifolds} \label{sectionderivedactions}
In this section we describe the various derived Hecke actions on the cohomology of our arithmetic manifolds that will be studied in the rest of the paper, and the relations between these actions.

Let $R$ be a coefficient ring. We denote $G_p = \mathrm G(\Qp)$.
Suppose $K_p \subset \mathrm G(\Zp)$ is an open compact subgroup.
\begin{defn}[Derived Hecke algebras] The derived Hecke algebra for $K_p$ is the $\Z_{\ge 0}$-graded algebra \[ \mathcal H^*_R \left( G_p, K_p \right) := \Ext^*_{G_p} \left( \cInd{G_p}{K_p} R, \cInd{G_p}{K_p} R \right) \] under composition of extension.
Notice that this is exactly the cohomology of the complex $\Hom_{G_p} \left( I^{\bullet}, I^{\bullet} \right)$ where $\cInd{G_p}{K_p} R \lra I^{\bullet}$ is an injective resolution in the category of smooth $G_p$-modules. One can also think of this complex as $\Hom_{\mathcal D(G_p)} \left( \cInd{G_p}{K_p} R, \cInd{G_p}{K_p} R \right)$, the endomorphisms of $\cInd{G_p}{K_p} R$ in the derived category of smooth $G_p$-modules.

Equivalently (see \cite{akshay}, section 2 and \cite{OS}, section 3), it is the $R$-module of $G_p$-equivariant cohomology classes \[ f: G_p / K_p \times G_p / K_p \lra \bigoplus_{x, y \in G_p/ K_p} H^* \left( \Stab_{G_p} (x,y), R \right) \] supported on finitely many $G_p$-orbits.
We require that \begin{itemize}
    \item $f(x,y) \in H^* \left( \Stab_{G_p} (x,y), R \right)$ and that
    \item $c_g^* f(gx,gy) = f(x,y)$ for all $g \in G_p$, where conjugation by $g$ induces an isomorphism \[ H^* \left( \Stab_{G_p} (gx,gy), R \right) \lra H^* \left( \Stab_{G_p} (x,y), R \right). \]
\end{itemize}
In this second and equivalent model, the multiplication operation is given by convolution, as follows: \[ f_1 \circ f_2 (x,y) = \sum_{z \in G_p / K_p} f_1(x,z) \cup f_2(z,y) \] where in the sum we only take one representative $z$ for each $\Stab_{G_p} (x,y)$-orbit on $G_p / K_p$, and moreover we are dropping the (necessary) restriction and corestriction maps for ease of notation. See \cite{akshay}, section 2 or \cite{Ronchetti} for details.
\end{defn}
\begin{rem} The classical Hecke algebra of double $K_p$-cosets is then simply the degree zero subalgebra of $\mathcal H^*_R \left( G_p, K_p \right)$.
\end{rem}
Let now $C^{\bullet}$ be a complex of (smooth) $G_p$-modules with $R$-coefficients. Then the hypercohomology of $K_p$ with coefficients in $C^{\bullet}$ can be computed as the cohomology of the complex $\Hom_{G_p} \left( \cInd{G_p}{K_p} R, C^{\bullet} \right)$.
One can equivalently think of this complex as the homomorphisms $\Hom_{\mathcal D(G_p)} \left( \cInd{G_p}{K_p} R, C^{\bullet} \right)$ in the derived category of smooth $G_p$-modules with coefficients in $R$.
This complex obviously has an action (by pre-composition) of the endomorphisms of $\cInd{G_p}{K_p} R$ in the derived category and hence this latter description makes it obvious that (by taking cohomology of the respective complexes) we get an action of the derived Hecke algebra $\mathcal H^*_R \left( G_p, K_p \right)$ on $H^* \left( K_p, C^{\bullet} \right)$.
We will call a \emph{derived Hecke operator} any endomorphism of $H^* \left( K_p, C^{\bullet} \right)$ arising in this way.
\begin{rem} \label{derivedidentitycoset} A special case of derived Hecke operators is obtained by classes supported on the `identity coset' - which is to say, cohomology classes in $H^* \left( K_p, R \right) \subset \mathcal H^*_R \left( G_p, K_p \right)$. See for instance the derived diamond operators constructed in the next section \ref{deriveddiamond}.
\end{rem}
This very abstract setup can then be made concrete for plenty of interesting complexes $C^{\bullet}$.
We describe next how to interpret arithmetic cohomology in this context.
For each open compact subgroup $K_p \subset G_p$, we let $C^{\bullet} (K_p)$ be the cochain complex of $Y(K^p \times K_p)$ with $R$-coefficients.
Let \[ C^{\bullet} = \varinjlim_{K_p} C^{\bullet} (K_p) \] be the direct limit over open compact subgroups of $G_p$ - as explained in section 2.6 of \cite{akshay}, this complex knows the cohomology of the tower at $p$, in the sense that one has a quasi-isomorphism \[ \left( C^{\bullet} \right)^{K_p} \sim C^{\bullet} \left( K_p \right) \] and thus taking the hypercohomology of $K_p$ yields the cohomology of the corresponding arithmetic manifold: \[ H^* \left( K_p, C^{\bullet} \right) = H^* \left( Y \left( K^p \times K_p \right), R \right). \]
By the general setup outlined above, we obtain that $H^* \left( Y \left( K^p \times K_p \right), R \right)$ is a graded $\mathcal H^*_R \left( G_p, K_p \right)$-module. We will use this repeatedly.

\subsection{Derived diamond operators} \label{deriveddiamond}
We now introduce the derived diamond operators. These are operators on arithmetic cohomology which increase the cohomological degree, and come (just like the classical diamond operators) from the torus. However, while the classical diamond operators are obtained via group elements in the torus, the derived diamond operators come from classes in the cohomology of this torus.

As defined in section \ref{sectionordinarycohomology}, on the complex $C^{\bullet} \left( Y_{b,c}, \Z / p^n \right)$ and on its cohomology $H^* \left( Y_{b,c}, \Z / p^n \right)$ we have two actions: \begin{enumerate}
    \item There is the classical action of the $U_p$-operators, - which in our setup are the operators $U_p^i = I(b,c) \lambda_i(p) I(b,c)$.
    \item We also have the diamond action of the operators $\langle \alpha \rangle$ for each $\alpha \in I(1,c) / I(b,c)$.
\end{enumerate}
By lemma \ref{Heckeincohomology} these two actions commute, even at the level of complexes.

\begin{lem} The covering $Y_{c,c} \twoheadrightarrow Y_{b,c}$ is Galois, with Galois group $T_{b,c} = \mathrm T \left( \Z / p^b \right) / \mathrm T \left( \Z / p^c \right)$.
\end{lem}
\begin{proof} From lemma 6.1 of \cite{KT}, we have that since $K(c,c)$ is normal in $K(b,c)$, the covering $Y_{c,c} \twoheadrightarrow Y_{b,c}$ is a Galois covering, with Galois group $ K(b,c) / K(c,c) \cdot \left( \mathrm Z_{\mathrm G}(\Q) \cap K(b,c) \right)$.
Since $K(b,c)$ and $K(c,c)$ coincide away from $p$, we have that \[ K(c,c) \cdot \left( \mathrm Z_{\mathrm G}(\Q) \cap K(b,c) \right) = K(c,c) \cdot \left( \mathrm Z_{\mathrm G}(\Q) \cap I(b,c) \right). \]
Now $I(b,c)$ is a pro-$p$-group (being contained in the pro-$p$-Iwahori subgroup $I(1,1)$), and from our assumption that $p > \left| \mu \left( \mathrm Z_{\mathrm G} \right) \right|$ we have that the finite group $\mathrm Z_{\mathrm G}(\Q)$ has no $p$-torsion.
Hence, the intersection $\mathrm Z_{\mathrm G}(\Q) \cap I(b,c)$ is trivial, and we obtain that the covering is Galois with Galois group \[ K(b,c) / K(c,c) \cong I(b,c) / I(c,c) \cong \mathrm T \left( \Z / p^b \right) / \mathrm T \left( \Z / p^c \right). \]
\end{proof}
We have an action of the cohomology of $T_{b,c}$ on the cohomology of the base, as explained by Venkatesh: by the universal property of the classifying space, the cover $Y_{c,c} \twoheadrightarrow Y_{b,c}$ corresponds to a map $\pi_{b,c}: Y_{b,c} \lra B T_{b,c}$ - taking cohomology yields a map $\pi_{b,c}^*: H^* \left( T_{b,c} , R \right) \lra H^* \left( Y_{b,c}, R \right)$ and the action on $H^* \left( Y_{b,c} , R \right)$ of a cohomology class in $H^* \left( T_{b,c}, R \right)$ is obtained by applying $\pi_{b,c}^*$ and then taking the cup product.
We call this action a \emph{derived diamond action}.

We can rephrase the same action in a more algebraic way. Let $R$ be a general coefficient ring and $C_{\bullet} (Y_{c,c}, R)$ the complex of singular chains with $R$-coefficients, whose homology computes the homology of $Y_{c,c}$. This complex has a natural action of $T_{b,c}$ via deck transformations.

Then it is standard that \[ H^* \left( Y_{b,c}, R \right) \cong H^* \left( \Hom_{R[T_{b,c}]} \left( C_{\bullet} (Y_{c,c}, R), R \right) \right), \] where on the left we have the singular cohomology of $Y_{b,c}$ and on the right the cohomology of the complex of $T_{b,c}$-equivariant homomorphisms.

In particular, the right-hand side has a natural action of $H^* \left( \Hom_{\mathcal D(R[T_{b,c}])} \left( R, R \right) \right)$, the homology of the complex of endomorphisms of the trivial representation in the derived category of $R[T_{b,c}]$-modules.
But this homology yields (by definition) exactly the cohomology of $T_{b,c}$ with coefficients in the trivial representation, and we recover thus an action \[ H^* \left( Y_{b,c}, R \right) \times H^* \left( T_{b,c} , R \right) \lra H^* \left( Y_{b,c}, R \right). \]

\begin{prop} \label{topologicalderivedaction} This action of $H^* \left( T_{b,c}, R \right)$ coincide with the topological action defined above.
\end{prop}
\begin{proof} See appendix B4 in \cite{akshay}.
\end{proof}
We also want to connect this derived diamond action to our general abstract setup defined at the beginning of this section.
Fix $R = \Z / p^n \Z$, or more generally a finite, $p^{\infty}$-torsion ring.
The isomorphism $T_{b,c} \cong I(b,c) / I(c,c)$ realizes $T_{b,c}$ as the maximal abelian quotient of $I(b,c)$, thereby proving that $H^1 (I(b,c), \Z / p^n) \cong H^1(T_{b,c}, \Z / p^n)$.
Remark \ref{derivedidentitycoset} says that $H^* \left( I(b,c), \Z / p^n \right) \subset \mathcal H^*_{\Z / p^n \Z} \left( G_p, I(b,c) \right)$ consists of derived Hecke operators.
\begin{fact} The action of $H^1 \left( I(b,c), \Z / p^n \right) \subset \mathcal H^*_{\Z / p^n \Z} \left( G_p, I(b,c) \right)$ on $H^* \left( Y_{b,c}, \Z / p^n \right)$ obtained via the abstract derived Hecke setup coincide with the concrete derived diamond action of $H^1 \left( T_{b,c}, \Z / p^n \right)$ described above.
\end{fact}
\begin{proof} This follows by unraveling the definitions, since both actions eventually boils down to cupping with a homomorphism $I(b,c) \lra \Z / p^n \Z$, which is an element of $H^1 \left( I(b,c), \Z / p^n \Z \right)$.
\end{proof}
We now fix $b=1$.
\begin{prop} \label{deriveddiamondcommuteswithUp} The action of the degree $1$ derived diamond operators commutes with the $U_p^i$-operators.
\end{prop}
\begin{proof}As explained above, we can view the derived diamond operators as elements of degree $1$ of the derived Hecke algebra of $I(1,c)$, supported on the `identity coset' $I(1,c)$.

In the usual description of the derived Hecke algebra as `$G_p$-equivariant cohomology classes' as in Venkatesh, we have that $U_p^i$ is supported on the $G_p$-orbit of $\left( I(1,c), \lambda_i I(1,c) \right)$, and valued \[ 1 \in H^0 \left( I(1,c) \cap \lambda_i I(1,c) \lambda_i^{-1}, \Z / p^n \right) \] on this particular double $I(1,c)$-coset.

To check that they commute, we explicitly compute their convolution in both orders. Let $f \in H^1 \left( T_{1,c}, \Z / p^n \right)$ be a derived diamond operator and $T = U_p^i \in \mathcal H^0_{\Z / p^n \Z} \left( G_p, I(1,c) \right)$.
We have \[ f \circ T \left( x I(1,c), y I(1,c) \right) = \sum_{z \in G_p / I(1,c)} f \left( xI(1,c), z I(1,c) \right) \cup T \left( z I(1,c), y I(1,c) \right). \]
We can assume that $x = I(1,c)$ since $f \circ T$ is $G_p$-equivariant. Then, $f$ is supported on the $G_p$-orbit of $(I(1,c), I(1,c))$, so the only nonzero contribution to the convolution comes from $z = I(1,c)$ too, and we obtain \[f \circ T \left( I(1,c), y I(1,c) \right) =f \left( I(1,c),  I(1,c) \right) \cup T \left( I(1,c), y I(1,c) \right). \]
We can assume that $y$ is chosen among a list of coset respesentatives for $I(1,c) \backslash G_p / I(1,c)$, and then the only one which makes the contribution of $T$ nonzero is $y I(1,c) = \lambda_i I(1,c)$, so we get that $f \circ T$ is supported on the single $G_p$-orbit of $\left( I(1,c), \lambda_i I(1,c) \right) $ and its value here is \[  f \circ T \left( I(1,c), \lambda_i I(1,c) \right) =f \left( I(1,c),  I(1,c) \right) \cup T \left( I(1,c), \lambda_i I(1,c) \right) = \]
\[ = \cores{I(1,c) \cap \lambda_i I(1,c) \lambda_i^{-1}}{I(1,c) \cap \lambda_i I(1,c) \lambda_i^{-1}} \left( \res{I(1,c)}{I(1,c) \cap \lambda_i I(1,c) \lambda_i^{-1}} f \left( I(1,c),  I(1,c) \right) \cup \res{I(1,c) \cap \lambda_i I(1,c) \lambda_i^{-1}}{I(1,c) \cap \lambda_i I(1,c) \lambda_i^{-1}} T \left( I(1,c), \lambda_i I(1,c) \right) \right) = \]
\[ = \res{I(1,c)}{I(1,c) \cap \lambda_i I(1,c) \lambda_i^{-1}} f \left( I(1,c),  I(1,c) \right). \]
On the other hand, we compute \[ T \circ f \left( x I(1,c), y I(1,c) \right) = \sum_{z \in G_p / I(1,c)} T \left( xI(1,c), z I(1,c) \right) \cup f \left( z I(1,c), y I(1,c) \right). \]
We can assume that $y = I(1,c)$ and that $x$ is chosen among a set of coset representatives for $I(1,c) \backslash G_p / I(1,c)$. Then just like before, the only $z$ for which $f$ gives a nontrivial contribution to cohomology is $z = I(1,c)$, so the sum collapses to \[ T \circ f \left( x I(1,c), I(1,c) \right) = T \left( xI(1,c), I(1,c) \right) \cup f \left( I(1,c), I(1,c) \right). \]
In order for the contribution of $T$ to the sum being nonzero, we need then $x = \lambda_i^{-1} I(1,c)$ and then \[ T \circ f \left( \lambda_i^{-1} I(1,c), I(1,c) \right) = \res{I(1,c)}{\lambda_i^{-1} I(1,c) \lambda_i \cap I(1,c)} f \left( I(1,c),  I(1,c) \right). \]
To show commutativity, we denote $F = f \left( I(1,c),  I(1,c) \right)$ and it remains thus to show that \[ \mathrm{conj}_{\lambda_i}^* \left( F|_{\lambda_i^{-1} I(1,c) \lambda_i \cap I(1,c)} \right) = F|_{I(1,c) \cap \lambda_i I(1,c) \lambda_i^{-1} }, \] but this is clear since $F : I(1,c) \lra \Z / p^n \Z$ is a homomorphism inflated from the torus, and thus conjugation by the torus element $\lambda_i$ leaves it invariant.
\end{proof}
\begin{cor} The derived degree $1$ diamond operators act on $H^* \left( Y_{1,c}, \Z / p^n \right)_{\ordi}$, preserving all $U_p^i$-eigenspaces.
\end{cor}
As shown in lemma \ref{ordinaryisom}, the natural pullback map \begin{equation} \label{ordinaryisomorph} H^* \left( Y_{1,1}, \Z / p^n \right)_{\ordi} \lra H^* \left( Y_{1,c}, \Z / p^n \right)_{\ordi} \end{equation} is an isomorphism.
We obtain then that on $H^* \left( Y_{1,1}, \Z / p^n \right)_{\ordi}$ we have both the $U_p^i$-operators acting invertibly, as well as the derived diamond operators $H^1 \left( T_{1,c}, \Z / p^n \Z \right)$.
\begin{prop} \label{deriveddiamondintower} The actions of the cohomology groups $H^1 \left( T_{1,c}, \Z / p^n \right)$ on $H^* \left( Y_{1,1}, \Z / p^n \right)_{\ordi}$ are compatible as $c$ increases, and thus give rise to the action of $H^1 \left( \mathrm T(\Zp)_p, \Z / p^n \right)$ on $H^* \left( Y_{1,1}, \Z / p^n \right)_{\ordi}$.
\end{prop}
\begin{proof} Looking back at the construction of the action of the derived diamond operators, this boils down to showing that if $\tau \in H^1 \left( T_{1,c} , \Z / p^n \right)$ and $f \in H^* \left( I(1,c), \Z / p^n \right)$, then \[ \res{I(1,c)}{I(1,c+1)} \left( \pi_c^* \tau \cup f \right) = \pi_{c+1}^* \left( p_c^* \tau \right) \cup \res{I(1,c)}{I(1,c+1)} \] where $p_c: T_{1,c+1} \twoheadrightarrow T_{1,c}$.

The compatibility of restriction and cup product in group cohomology shows that it suffices to show that $\res{}{} \circ \pi_c^* = \pi_{c+1}^* \circ p_c^*$ on $\tau$, and this follows by the consideration that $\tau \in H^1 \left( T_{1,c}, \Z / p^n \right)$ is a group homomorphism.
\end{proof}
\begin{cor} We obtain a graded action of $H^* \left( \mathrm T(\Zp)_p, \Zp \right)$ on $H^* \left( Y_{1,1}, \Zp \right)_{\ordi}$ which commutes with the Hecke actions on $ H^* \left( F^{\bullet}_{\infty} \otimes_{\Lambda} \Lambda_1 \right)$.
\end{cor}
\begin{proof} By taking the inverse limit in $n$ we obtain from the previous proposition an action of $H^1 \left( \mathrm T(\Zp)_p, \Zp \right)$ on $H^* \left( Y_{1,1}, \Zp \right)_{\ordi}$.
Then, one knows that the full cohomology algebra of the pro-$p$ integral $p$-adic torus $\mathrm T(\Zp)_p$ is generated over its degree $1$: \[ H^* \left( \mathrm T(\Zp)_p, \Zp \right) = \bigwedge^* H^1 \left( \mathrm T(\Zp)_p, \Zp \right), \] which proves the claim about the graded action.

Compatibility with the Hecke operators away from $p$ is clear, while compatibility of this graded action with the $U_p$ operators follows from proposition \ref{deriveddiamondcommuteswithUp}.
\end{proof}

\section{Generalities on automorphic representations} \label{automorphicsection}
In this section we collect results on the automorphic representations appearing in the cohomology of the arithmetic manifolds we are interested in. We mainly follow section 4 of \cite{HT}.

Recall that $K_0 = \prod_{q < \infty} K_q \subset \mathrm G(\Af)$ is our base level, where $K_p$ is an Iwahori subgroup of $\mathrm G(\Zp)$, $K_q$ is a hyperspecial maximal compact subgroup of $\mathrm G(\Z_q)$ for all $q \not\in S$ and $K_0$ is a good subgroup.

We also considered $K_1 = \prod_{q < \infty} K'_q$, the subgroup of $K_0$ having pro-$p$-Iwahori level at $p$ and $K'_q = K_q$ for all $q \neq p$.

We fixed a maximal compact subgroup $K_{\infty}$ of $\mathrm G(\R)$, and for any good subgroup $K$ of $K_0$, we consider the arithmetic manifold \[ Y(K) = \mathrm G(\Q) \backslash \mathrm G(\A) / \left( K_{\infty} \cdot Z^{\circ} \right) K = \mathrm G(\Q) \backslash \mathrm G(\Af) \times \mathrm G(\R) / \left( K_{\infty} \cdot Z^{\circ} \right) K \] having dimension $d$ as a real manifold.

Recall that the defect of $\mathrm G$ is $l_0 = \rk \mathrm G(\R) - \rk K_{\infty}$, that we denoted $q_0 = \frac{d - l_0}{2}$, and that $E$ is a finite extension of $\Qp$ assumed to be large enough for all necessary purposes.

\begin{defn} \label{autrepcontributes} Let $\pi^{\infty} = \prod'_{q < \infty} \pi_q$ be an irreducible admissible representation of $\mathrm G(\Af)$ on an $E$-vector space.
Suppose that $\pi^{\infty}$ is unramified outside $S$, and let \[ \mathfrak M = \ker \left( \mathcal H_E \left( \mathrm G(\A^S), K^S \right) \lra \End_E \left( \prod'_{q \not\in S} \pi_q^{\mathrm G(\Z_q)} \right) \right), \]
a maximal ideal of $\mathcal H_E \left( \mathrm G(\A^S), K^S \right)$ with residue field a finite extension of $E$.

If $M^{\bullet}$ is a complex of $\mathcal H_E \left( \mathrm G(\A^S), K^S \right)$-modules, we say that $\pi^{\infty}$ \emph{contributes to $M^{\bullet}$} if the localization $M^{\bullet}_{\mathfrak M}$ is nonzero.
If $\pi$ is an automorphic representation with component $\pi^{\infty}$ away from $\infty$, we say that $\pi$ contributes to $M^{\bullet}$ if $\pi^{\infty}$ does.
\end{defn}

As explained in section \ref{sectionordinarycohomology}, we have a commutative subalgebra $E \left[ X_*^+ \left( \Zp \right) \right]$ of $\mathcal H_E \left( \mathrm G(\Qp), I_p' \right)$, where $I_p'$ is the pro-$p$-Iwahori subgroup of $\mathrm G(\Qp)$.

\begin{defn} \label{ordinaryHeckealgebraafterKT}
For every complex $M^{\bullet}$ of $\Lambda_E$-modules with an Hecke action \[ \mathcal H_E \left( G^{\infty, S}, K^S \right) \times E \left[ X_*^+ (\mathrm T) \right] \lra \End_{\mathcal D(\Lambda_E)} \left( M^{\bullet} \right) \] we denote by $\mathbb T^S(M^{\bullet})$ the image of the last action map inside $\End_{\mathcal D(\Lambda_E)} \left( M^{\bullet} \right)$.
\end{defn}
We also remark that since the action commutes with the differential of the complex $M^{\bullet}$, we obtain a homomorphism \[ \mathbb T^S (M^{\bullet}) \lra \mathbb T^S \left( H^* (M^{\bullet}) \right). \]

Following section 4.1 of \cite{Hida3}, suppose that the open compact subgroup $K$ contains $\mathrm U(\Zp)$. Then we notice that the operator $U_p$ constructed in section \ref{sectionordinarycohomology} acts on $H^* \left( Y(K), E \right)$. The idempotent $e = \lim_{n \to \infty} U_p^{n!}$ converges as an endomorphism of $H^* \left( Y(K), E \right)$.
\begin{defn} Let $\pi$ be a cuspidal automorphic representation of $\mathrm G(\A)$ of level $K$, with $K \supset \mathrm U(\Zp)$. $\pi$ is said to be \emph{ordinary} at $p$ if it appears in $e \left( H^* \left( Y(K), E \right) \right)$.
\end{defn}
Equivalently, we can fix a copy $V(\pi)$ of the representation space of $\pi$ inside $H^* \left( Y(K), E \right)$, and require that $e \left( V(\pi) \right) \neq 0$.
\begin{rem} Hida in \cite{Hida3} calls these representations \emph{nearly ordinary}.
\end{rem}
\begin{defn} \label{stronglygenericrep}
Suppose that $\pi$ is unramified at $p$, so that its component $\pi_p$ is an unramified principal series, induced by a unramified character $\chi: \mathrm T(\Qp) \lra E^*$. The character $\chi$ is only defined up to Weyl-conjugation, but this will not matter in the next argument. We call $\pi$ \emph{strongly generic} at $p$ if the stabilizer of $\chi$ under the Weyl group action is the trivial group.
\end{defn}

\begin{thm} \label{autrepincohomology} Let $\pi$ be a regular, algebraic, cuspidal automorphic representation of $\mathrm G(\A)$. Suppose that $\pi$ contributes to $H^* \left( Y(K_{1,1}), E \right)$ and that $\pi$ is strongly generic and ordinary at $p$.
Then \begin{enumerate}
\item The cohomology $H^i \left( Y(K_{1,1}), E \right)$ is only nonzero for $i \in [q_0, q_0 + l_0]$.
\item Let $\mathfrak M$ be the maximal ideal associated to $\pi^{\infty}$ by definition \ref{autrepcontributes}.
Then \[ H^* \left( Y(K_{1,1}), E \right)_{\mathfrak M} = H^* \left( Y(K_{1,1}), E \right) \left[ \mathfrak M \right], \] and therefore $\mathbb T^S \left( H^* \left( Y(K_{1,1}), E \right) \right)_{\mathfrak M}$ is $E$.
\item There exists a positive integer $m \left( \pi, K \right)$ such that \[ \dim_E H^{q_0+i} \left( Y(K_{1,1}), E \right)_{\mathfrak M} = m(\pi, K) \binom{l_0}{i} \qquad \forall i \in [0, l_0]. \]
\end{enumerate}
\end{thm}
\begin{proof} Theorems 5.1 in chapter III and 5.2 in chapter VII of \cite{BW} implies the first and the third part of the theorem, as explained in \cite{HT}.

Corollary 4.3 of \cite{Hida3} implies that $e \left( H^* \left( Y(K_{1,1}, E \right) \right)$ is a semisimple Hecke module.
In particular, for every ordinary maximal ideal $\mathfrak M$, we necessarily have $H^* \left( Y(K_{1,1}), E \right)_{\mathfrak M} = H^* \left( Y(K_{1,1}), E \right) \left[ \mathfrak M \right]$ as explained in the proof of proposition 4.2 of \cite{HT}.
\end{proof}
From now on, we will assume that our automorphic representation $\pi$ satisfies the conditions of theorem \ref{autrepincohomology}.

\subsection{Comparison of base levels at $p$}
In this section we discuss the relationships between spherical, Iwahori and pro-$p$-Iwahori level at $p$. We follow closely \cite{akshay}, section 6.
This is a very delicate subject when working with torsion coefficients, but we are instead considering cohomology with rational coefficients, and so our situation is much simpler and closely related to the classical, complex, case.
We denote by $G_p, K_p$, $I_p$ and $I_p'$ respectively $\mathrm G(\Qp)$, $\mathrm G(\Zp)$ (a hyperspecial maximal compact), the Iwahori subgroup at $p$ corresponding to the choice of a Borel subgroup we fixed in the introduction and its pro-$p$ Sylow subgroup.

First of all, we have by \cite{HKP}, sections 1.6 and 1.7 an injection \[ \overline {\Qp} \left[ X_* (\mathrm T) \right] \hookrightarrow \mathcal H_{\overline {\Qp}} \left( G_p, I_p \right) \] which is in fact defined over $\Qp (\sqrt p)$, and even over $\Qp$ if $\mathrm G$ is simply connected.

There is a natural action of the Weyl group $W$ on $X_*(\mathrm T)$, and in the above injection we have (see section 2 of \cite{HKP}) that \[ \overline {\Qp} \left[ X_* (\mathrm T) \right]^W \stackrel{\sim}{\lra} Z \left( \mathcal H_{\overline {\Qp}} \left( G_p, I_p \right) \right), \] which is to say that the Weyl-invariant subalgebra of $\overline {\Qp} \left[ X_*(\mathrm T) \right]$ maps isomorphically onto the center of the Iwahori-Hecke algebra (this is the Bernstein isomorphism).
We will identify implicitly identify $\overline {\Qp} \left[ X_* (\mathrm T) \right]^W$ and  $Z \left( \mathcal H_{\overline {\Qp}} \left( G_p, I_p \right) \right)$ via this isomorphism, from now on.

Next, the Satake transform (as in \cite{HKP}, section 4) is an injective morphism \[ \mathcal H_{\overline {\Qp}} \left( G_p, K_p \right)  \lra \overline {\Qp} \left[ X_* (\mathrm T) \right] \] which turns out to be an isomorphism onto $\overline {\Qp} \left[ X_* (\mathrm T) \right]^W$. Again, this is defined over $\Qp(\sqrt p)$, and over $\Qp$ if $\mathrm G$ is simply connected.

Let then suppose that our coefficient field $E$ contains $\Qp ( \sqrt p)$ if $\mathrm G$ is not simply connected.
We obtain \[ \mathcal H_E \left( G_p, K_p \right) \stackrel{\sim}{\lra} Z \left( \mathcal H_E \left( G_p, I_p \right) \right). \]

Consider the $\left( \mathcal H_E \left( G_p, K_p \right), \mathcal H_E \left( G_p, I_p \right) \right)$-bimodule $H_{KI} = \Hom_{E[G_p]} \left( \cInd{G_p}{K_p} E, \cInd{G_p}{I_p} E \right)$ as well as the $\left( \mathcal H_E \left( G_p, I_p \right), \mathcal H_E \left( G_p, K_p \right) \right)$-bimodule $H_{IK} = \Hom_{E[G_p]} \left( \cInd{G_p}{I_p} E, \cInd{G_p}{K_p} E \right)$.

We obtain then \emph{change of level} morphisms \begin{equation} \label{changeoflevel} \begin{gathered} H^* \left( Y(\mathrm G(\Zp)), E \right) \otimes_{\mathcal H_E \left( G_p, K_p \right)} H_{KI} \lra H^* \left( Y_0, E \right) \\ H^* \left( Y_0, E \right) \otimes_{\mathcal H_E \left( G_p, I_p \right)} H_{IK} \lra H^* \left( Y \left( \mathrm G(\Zp) \right),E \right). \end{gathered} \end{equation}
Let $\pi^{\infty} = \prod'_{q < \infty} \pi_q$ be a cuspidal irreducible automorphic representation of $\mathrm G(\Af)$ on an $E$-vector space.
Suppose that $\pi$ is unramified outside $T$, and let \[ \mathfrak N = \ker \left( \mathcal H_E \left( \mathrm G(\A^T), K^T \right) \lra \End_E \left( \prod'_{q \not\in T} \pi_q^{\mathrm G(\Z_q)} \right) \right), \]
a maximal ideal of $\mathcal H_E \left( \mathrm G(\A^T), K^T \right)$ with residue field a finite extension of $E$.
There is an obvious injection $\mathcal H_E \left( \mathrm G(\A^S), K^S \right) \hookrightarrow \mathcal H_E \left( \mathrm G(\A^T), K^T \right)$, and with a small abuse of notation we also denote by $\mathfrak N$ the maximal ideal of $\mathcal H_E \left( \mathrm G(\A^S), K^S \right)$ obtained by pulling back.

Notice that $\mathcal H_E \left( \mathrm G(\A^T), K^T \right)$ acts on $H^* \left( Y(\mathrm G(\Zp)), E \right)$ and $\mathcal H_E \left( \mathrm G(\A^S), K^S \right)$ acts on $H^* \left( Y_0, E \right)$.
We can then localize these cohomology groups and the maps in formula \ref{changeoflevel} at $\mathfrak N$, and we obtain \begin{equation} \label{localizedchangeoflevel} \begin{gathered} H^* \left( Y(\mathrm G(\Zp)), E \right)_{\mathfrak N} \otimes_{\mathcal H_E \left( G_p, K_p \right)} H_{KI} \lra H^* \left( Y_0, E \right)_{\mathfrak N} \\ H^* \left( Y_0, E \right)_{\mathfrak N} \otimes_{\mathcal H_E \left( G_p, I_p \right)} H_{IK} \lra H^* \left( Y \left( \mathrm G(\Zp) \right), E \right)_{\mathfrak N}. \end{gathered} \end{equation}
Let now \[ f = \prod_{\alpha \in \Phi(\mathrm G, \mathrm T)} (1- \alpha^*) \in E \left[ X_*(\mathrm T) \right]^W \cong \mathcal H_E \left( G_p, K_p \right) \] be the \emph{discriminant}, where $\alpha^*=(\alpha^{\vee})^{m_{\alpha}} \in X_*(\mathrm T)$  is a power of the coroot $\alpha^{\vee}$, and $m_{\alpha}$ is the largest positive integer so that $\frac{\alpha}{m_{\alpha}}$ is still in the integral character group $ X^*(\mathrm T)$.

Recall that since we assume that $\pi$ is unramified at $p$, its component $\pi_p$ is an unramified principal series, induced by a unramified character $\chi: \mathrm T(\Qp) \lra E^*$ (defined only up to Weyl-conjugation).

Since $\chi$ is unramified, it corresponds to a homomorphism $\chi: \mathrm T(\Qp) / \mathrm T(\Zp) = X_* \left( \mathrm T \right) \lra E^*$, which we extend linearly to an algebra homomorphism $\chi: E \left[ X_*(\mathrm T) \right] \lra E$.
\begin{thm}[See lemma 6.5 of \cite{akshay}] \label{changeoflevelthm}
\begin{enumerate}
\item The maps in formula \ref{changeoflevel} are isomorphisms upon localization away from $f$.
\item Suppose that the $\chi$ is strongly regular, in the sense that its stabilizer under the Weyl group action on $X_*(\mathrm T)$ consists of the identity element only. In other words, suppose that $\pi$ is strongly generic at $p$ as in Definition \ref{stronglygenericrep}.
Then the maps in formula \ref{localizedchangeoflevel} are isomorphisms, where $\mathfrak N$ is as usual the ideal corresponding to $\pi$.
\end{enumerate}
\end{thm}
\begin{proof} 
An easy computation shows that $\mathcal H_E(G_p, K_p)$ acts on $\pi_p^{\mathrm G(\Zp)}$ precisely as the character $\chi$, via the identification $\mathcal H_E(G_p,K_p) \cong E \left[ X_*(\mathrm T) \right]^W$ provided by the Satake isomorphism.

Therefore, the image of the discriminant in the map \[ \chi_{\pi}: \mathcal H_E \left( \mathrm G(\A^T), K^T \right) \lra \End_E \left( \prod'_{q \not\in T} \pi_q^{\mathrm G(\Z_q)} \right) \] is \[ \chi_{\pi} \left( \prod_{\alpha \in \Phi(\mathrm G, \mathrm T)} (1- \alpha^*) \right) = \prod_{\alpha \in \Phi(\mathrm G, \mathrm T)} \left( 1- \chi(\alpha^*) \right). \]
Notice that $W$-invariance of the discriminant guarantees that this image is independent of the representative chosen in the Weyl-conjugacy class of $\chi$. 

The definition of $\alpha^*$ is so that $\psi(\alpha^*)=1$ exactly when $\psi \in \Hom \left( X_*(\mathrm T), \Z \right) \cong X^*(\mathrm T)$ is fixed by the reflection $\alpha$. Since we are assuming that $\chi$ is strongly regular, we have that $\chi(\alpha^*) \neq 1$ for all roots $\alpha$, and therefore $f \not\in \mathfrak M$.
The first part of the theorem implies thus the second one, so it remains to prove that the maps in formula \ref{changeoflevel} are isomorphisms upon localization away from $f$.

To prove the first part of the theorem, we again follow \cite{akshay}: the same argument as in the proof of lemma 6.5 of loc. cit. shows that it suffices to prove that \[ \begin{gathered} F: \Mod \left( \mathcal H_E \left( G_p, K_p \right)_f \right) \lra \Mod \left( \mathcal H_E \left( G_p, I_p \right)_f \right) \qquad F(M)= M \otimes_{\mathcal H_E \left( G_p, K_p \right)_f} (H_{KI})_f \\  G: \Mod \left( \mathcal H_E \left( G_p, I_p \right)_f \right) \lra \Mod \left( \mathcal H_E \left( G_p, K_p \right)_f \right) \qquad G(N)= N \otimes_{\mathcal H_E \left( G_p, I_p \right)_f} (H_{IK})_f \end{gathered} \] are functors defining an equivalence of categories.
Notice that here we are using respectively the Satake isomorphism and the Bernstein isomorphism to localize $\mathcal H_E \left( G_p, K_p \right)$ and $\mathcal H_E \left( G_p, I_p \right)$ at $f$.

Lemma 4.5 of \cite{akshay} shows that $F$ and $G$ define an equivalence of categories in the more subtle positive characteristic setup. Since that proof works verbatim in our setting, we only recall here the main steps.
It suffices to show that the multiplication maps \[ \begin{gathered} (H_{IK})_f \otimes_{\mathcal H_E \left( G_p, K_p \right)_f} (H_{KI})_f \lra \mathcal H_E \left( G_p, I_p \right)_f  \\ (H_{KI})_f \otimes_{\mathcal H_E \left( G_p, I_p \right)_f} (H_{IK})_f \lra \mathcal H_E \left( G_p, K_p \right)_f \end{gathered} \] are isomorphisms.
We show that the first one is, the second one will follow similarly (in fact, it will be easier as the dimensions are collapsing).

The standard presentation of $H_E \left( G_p, I_p \right)$ in terms of the affine Weyl group (see for instance section 1 of \cite{HKP}) shows that $H_E \left( G_p, I_p \right)$ is free of rank $|W|$ as a (right) $E \left[ X_*(\mathrm T) \right]$-module.

The Chevalley-Shepard-Todd theorem implies that $E \left[ X_*(\mathrm T) \right]$ is locally free of rank $|W|$ as a (right) $Z = E \left[ X_*(\mathrm T) \right]^W$-module.

We obtain thus that $\left( H_E \left( G_p, I_p \right) \right)_f$ is locally free of rank $|W|^2$ as a $Z_f$-module, and similarly $(H_{IK})_f$ and $(H_{KI})_f$ are locally free of rank $|W|$ as a $Z_f$-module, while clearly (by the Satake isomorphism) $\left( H_E \left( G_p, I_p \right) \right)_f$ is free of rank $1$ as a $Z_f$-module.

Since the ranks match up, it suffices to show that the multiplication map above is onto after quotienting modulo each maximal ideal $\mathfrak p$ of $Z_f$ - in fact, since the failure of the (unlocalized) multiplication map $H_{IK} \otimes_{\mathcal H_E(G_p,K_p)} K_{KI} \lra \mathcal H_E(G_p, I_p)$ to be an isomorphism of $Z$-modules is given by the vanishing of a certain polynomial (again, since the ranks match up), it suffices to consider maximal ideals $\mathfrak p$ whose complement has codimension at least $2$.

The quotient map $Z_f / \mathfrak p \lra \overline E$ can be extended to $\phi: E \left[ X_*(\mathrm T) \right]_f \lra \overline E$, which is a point in the character variety $\mathrm{Spec} \left( E \left[ X_* (\mathrm T) \right] \right)$.

We can assume that $\phi$ is strongly regular, in the sense that its stabilizer under the Weyl-action consists only of the identity element.
Denote also by $\phi$ the corresponding unramified character $\mathrm T(\Qp) \lra \overline E^*$, and let $V_{\phi}$ be the principal series representation of $G_p$ induced from $\phi$.

Strong regularity ensures that $\dim_{\overline E} V_{\phi}^{I_p} = |W|$, with a basis being given by the functions $\left\{ v_w \right\}_{w \in W}$ whose restriction to $K_p$ is the characteristic function of the Bruhat cell indexed to $w$. $Z$ acts on $V_{\phi}^{I_p}$ via the character $\phi$.

This description and an explicit computation show that \[ \mathcal H_E \left( G_p, I_p \right) \otimes_Z \overline E \lra \Hom_{\overline E} \left( V_{\phi}^{I_p}, V_{\phi}^{I_p} \right) \] is an isomorphism, and similarly \[ \begin{gathered} H_{IK} \otimes_Z \overline E \cong \Hom_{\overline E} \left( V_{\phi}^{I_p}, V_{\phi}^{K_p} \right) \\ H_{KI} \otimes_Z \overline E \cong \Hom_{\overline E} \left( V_{\phi}^{K_p}, V_{\phi}^{I_p} \right) \\ \mathcal H_E \left( G_p, K_p \right) \otimes_Z \overline E \cong \Hom_{\overline E} \left( V_{\phi}^{K_p}, V_{\phi}^{K_p} \right) \end{gathered} \]
This shows that upon quotienting modulo $\mathfrak p$, the multiplication map is a surjection among $\overline E$-vector spaces of dimension $|W|^2$, proving the claim.
\end{proof}
Our next goal is to understand the decomposition of $H^* \left( Y_0, E \right)_{\mathfrak N}$ in eigenspaces under the Iwahori-Hecke action, which is not present at the level of the maximal compact subgroup $\mathrm G(\Zp)$.
In fact, to understand these eigenspaces it suffices to study the action of the commutative subalgebra $E \left[ X_*^+ (\mathrm T)^+ \right] \hookrightarrow \mathcal H_E(G_p, I_p)$.

Recall that $\chi: X_*(\mathrm T) \lra E$ is the unramified character inducing the principal series representation $\pi_p$. We are assuming that $\chi$ is strongly regular, and therefore its conjugacy class $\left\{ \chi^w \right\}_{w \in W}$ consists of precisely $|W|$ distinct characters, each inducing a principal series representation isomorphic to $\pi_p$.

\begin{cor} Let $\pi$ be as in theorem \ref{changeoflevelthm}.
The generalized eigenvalues of $E \left[ X_*^+ (\mathrm T) \right]$ on $H^* \left( Y_0, E \right)_{\mathfrak N}$ are all those of the form $\chi^w$, as $w$ varies in the Weyl group.
\end{cor}
\begin{proof} We compute with the isomorphism \[ H^* \left( Y_0, E \right)_{\mathfrak N} \cong H^* \left( Y(\mathrm G(\Zp)), E \right)_{\mathfrak N} \otimes_{\mathcal H_E \left( G_p, K_p \right)} H_{KI} \] from theorem \ref{changeoflevelthm}.

Let $\phi: E \left[ X_*^+ (\mathrm T)^+ \right] \lra E$: then $H^* \left( Y(\mathrm G(\Zp)), E \right)_{\mathfrak N} \otimes_{\mathcal H_E \left( G_p, K_p \right)} H_{KI}$ contains an eigenvector with eigencharacter $\phi$ if and only if $H^* \left( Y(\mathrm G(\Zp)), E \right)_{\mathfrak N}$ contains an eigenvector with eigencharacter $\phi^w$ for some $w \in W$, which (since we know that $\mathcal H_E(G_p, K_p)$ acts on $H^*(Y(\mathrm G(\Zp)), E)$ as $\chi$) boils down to $\chi = \phi^w$ for some $w \in W$.
\end{proof}
We denote by $H^* (Y_0, E)_{\mathfrak N, \chi^w}$ the generalized eigenspace of $\chi^w$. An equivalent formulation of the above result in terms of maximal ideals is given by the following corollary.
\begin{cor} There are exactly $|W|$ maximal ideals $\mathfrak N'$ of $\mathcal H_E \left( \mathrm G(\A^S), K^S \right) \times E \left[ X_*^+ \left( \Zp \right) \right]$ above $\mathfrak N$. Each $w \in W$ yields one of these maximal ideals as \[ \mathfrak N' = \mathfrak N_w = \langle \mathfrak N, U_p^1 - \chi^w \left( U_p^1 \right), \ldots, U_p^r - \chi^w \left( U_p^r \right) \rangle. \]
\end{cor}
We also record the following result.
\begin{cor} For every $w \in W$ we have isomorphisms (Hecke-equivariant away from $S$) \[ H^* (Y_0, E)_{\mathfrak N, \chi^w} \cong H^* \left( Y(\mathrm G(\Zp)), E \right)_{\mathfrak N} \] where the maps are respectively pushforward in one direction, and pullback followed by projection onto the $\chi^w$-eigenspace in the other.
\end{cor}
\begin{proof} This is an explicit computation as in the discussion following corollary 6.6 of \cite{akshay}.
\end{proof}
We end this section by describing the change from Iwahori level to pro-$p$-Iwahori level. This is simpler than the discussion above (because the pro-$p$-Iwahori is normal in the Iwahori, with finite, abelian quotient) - we again follow section 6 of \cite{akshay}.

Recall that $I_p'$ is a the pro-$p$-Iwahori subgroup.
Inside the Hecke algebra $\mathcal H_E \left( G_p, I_p' \right)$, we have the commutative subalgebra $E \left[ \widetilde X_*^+ (\mathrm T) \right]$, where $\widetilde X_*(\mathrm T) \cong \mathrm T(\Qp) / \mathrm T(\Zp)_p$ fits into the exact sequence \[ 1 \lra \mathrm T(\Zp)_p \lra \mathrm T(\Qp) \lra \widetilde X_*(\mathrm T) \lra 1 \] and $\widetilde X_*^+ (\mathrm T)$ is the preimage of the dominant cone under the quotient map $\widetilde X_* (\mathrm T) \lra X_*(\mathrm T)$.
\begin{prop} The eigenspaces of $E \left[ \widetilde X_*^+ (\mathrm T) \right]$ acting on $H^* \left( Y_1, E \right)_{\mathfrak N}$ coincide with the eigenspaces of $E \left[  X_*^+ (\mathrm T) \right]$ acting on $H^* \left( Y_0, E \right)_{\mathfrak N}$ via the quotient map $E \left[ \widetilde X_*^+ (\mathrm T) \right] \lra E \left[ X_*^+ (\mathrm T) \right]$.
\end{prop}
\begin{proof} This is proved exactly like in the last lemma of section 6.7 of \cite{akshay}.
We therefore recall the main steps of the argument.
It suffices to prove the claim before localization at $\mathfrak N$.
Recall that \[ \begin{gathered} H^*(Y_1, E) \cong \Ext_{E[G_p]} \left( E \left[ G_p / I_p' \right], M^{\bullet} \right) \\ H^*(Y_0, E) \cong \Ext_{E[G_p]} \left( E \left[ G_p / I_p \right], M^{\bullet} \right) \end{gathered} \] where $M^{\bullet}$ is the direct limit of cochain complexes of finite coverings.
\begin{claim} There exists a filtration of $E \left[ G_p / I_p' \right]$ by $G_p$-submodules: \[ 0=F^0 \subset F^1 \subset \ldots \subset F^s=E \left[ G_p / I_p' \right] \] such that \begin{enumerate}
\item Each successive quotient $F^{j+1}/F^j$ is isomorphic as a $G_p$-module to a direct sum of copies of $E \left[ G_p / I_p \right]$.
\item The filtration is stable for the action of $E \left[ \widetilde X_*^+ (\mathrm T) \right]$, and on each successive quotient $E \left[ \widetilde X_*^+ (\mathrm T) \right]$ acts via its quotient map to $E \left[ X_*^+ (\mathrm T) \right]$.
\end{enumerate}
\end{claim}
Notice that the first requirement of the claim implies in particular that the filtration has finitely many steps, since the index of $I_p'$ in $I_p$ is finite.
\begin{proof}[Proof of claim] Denote $\Delta = I_p / I_p' \cong \mathrm T(\Fp)$ in this proof.
We have $E \left[ G_p / I_p' \right] \cong E [G_p] \otimes_{E[I_p]} E [ \Delta]$. The action of $\delta \in \Delta$ on $E[G_p]$ coincide with the action of $I_p' \delta I_p'$ and therefore commutes with the action of $E \left[ \widetilde X_*^+ (\mathrm T) \right]$.

Let $\mathfrak m \subset E[\Delta]$ be the augmentation ideal (which is the maximal ideal).
One filters $E \left[ \widetilde X_*^+ (\mathrm T) \right]$ by the kernels of successive powers of $\mathfrak m$: \[ F^j = E[G_p] \otimes_{E[I_p]} E[\Delta] \langle \mathfrak m^j \rangle, \] obtaining a $E \left[ \widetilde X_*^+ (\mathrm T) \right]$-stable filtration since the action of this monoid algebra commutes with the action of $\Delta$.

The successive quotients of the filtration are then \[ F^{j+1}/ F^j \cong E[G_p] \otimes_{E[I_p]} \langle \mathfrak m^{j+1} \rangle / \langle \mathfrak m^j \rangle, \] which is indeed a finite direct sum of copies of $E \left[ G_p / I_p \right]$.

To check the second requirement of the claim, one notices that a finite sum of multiplication maps by various $x \in \mathfrak m^j$ yields an isomorphism $F^{j+1} / F^j \lra F^1 / F^0 = F^1$ which is $E \left[ \widetilde X_*^+ (\mathrm T) \right]$-equivariant, as $\Delta$ commutes with $E \left[ \widetilde X_*^+ (\mathrm T) \right]$.
It suffices thus to check the second requirement for $F^1$, but here it is obvious as we have shown that $F^1$ is a direct sum of copies of $E[G_p / I_p]$.
\end{proof}
The proposition now follows by induction on the index $j$ of the filtration, using the long exact sequence in cohomology (which is $E \left[ \widetilde X_*^+ (\mathrm T) \right]$-equivariant by the second requirement of the claim).
\end{proof}
We again rephrase the last result in terms of maximal ideals.
\begin{cor} \label{maxidealsatproplevel} Let $\mathfrak N_w \subset \mathcal H_E \left( \mathrm G(\A^S), K^S \right) \times E \left[ X_*^+ \left( \Zp \right) \right]$ be a maximal ideal associated to $\pi$ appearing in the Hecke decomposition of $H^*(Y_0, E)$.
There is exactly one maximal ideal of $\mathcal H_E \left( \mathrm G(\A^S), K^S \right) \times E \left[ \widetilde X_*^+ \left( \Zp \right) \right]$ above it, defined by \[ \widetilde {\mathfrak N_w} = \langle \mathfrak N, U_p^1 - \chi^w \left( U_p^1 \right), \ldots, U_p^r - \chi^w \left( U_p^r \right) \rangle, \] where $U_p^i = I_p' \lambda_i(p) I_p'$.
\end{cor}
\subsection{Inverting $p$} \label{invertingpsubsec}
In this subsection we recall the results on ordinary cohomology and the derived diamond operators from sections \ref{sectionordinarycohomology} and \ref{sectionderivedactions} and apply them to our setup with rational coefficients by inverting $p$.

Recall that $F_{\infty}^{\bullet}$, obtained as in proposition \ref{complexordinarytower}, is the complex interpolating ordinary cohomology in the Hida tower.
We denote $F_{\infty,E}^{\bullet} = F_{\infty}^{\bullet} \otimes_{\Zp} E$, this has an action of $\mathcal H_E \left( G^{\infty, S}, K^S \right) \times E \left[ X_*^+ (\mathrm T) \right]$, and let us denote \[\mathbb T^S_{\ordi}(K, E) = \mathbb T^S \left( F_{\infty,E}^{\bullet} \right), \] as in definition \ref{ordinaryHeckealgebraafterKT}.

Recall that our automorphic representation $\pi$ is as in theorem \ref{autrepincohomology}.
Let $\mathfrak m = \widetilde {\mathfrak N_w}$ be one of the maximal ideals of $\mathbb T^S_{\ordi}(K, E)$ obtained as in corollary \ref{maxidealsatproplevel}. Then, as in lemma 6.17 of \cite{KT}, $\mathfrak m$ appears in the support of $H^* \left( F_{\infty,E}^{\bullet} \right)$.

As in section 2.4 of \cite{KT}, the maximal ideal $\mathfrak m \subset \mathbb T^S_{\ordi} (K,E)$ determines an idempotent $e_{\mathfrak m} \in \End_{\mathcal D(\Lambda_E)} \left( F_{\infty,E}^{\bullet} \right)$.
We denote by $F_{\mathfrak m}^{\bullet}$ the direct factor representing the subcomplex $e_{\mathfrak m} F_{\infty,E}^{\bullet}$, and we get canonical isomorphisms \[ \mathbb T^S_{\ordi}(K,E)_{\mathfrak m} \cong e_{\mathfrak m} \mathbb T^S_{\ordi}(K,E) \cong \mathbb T^S \left( F_{\mathfrak m}^{\bullet} \right). \]

Following section 6.5 of \cite{KT} (and using the fact that $\pi$ is ordinary at $p$), we obtain: \begin{equation} \label{ordinarycohomologycomplex} \begin{split} H^* \left( F_{\mathfrak m}^{\bullet} \otimes_{\Lambda_E} \Lambda_{c,E} \right) \cong H_* \left( Y_{c,c}, E \right)_{ \mathfrak m} \\
H^* \left( F_{\mathfrak m}^{\bullet} \otimes_{\Lambda,E } \Lambda_{c,E}  \otimes_{\Lambda_c,E } E \right) \cong H_* \left( Y_{1,c}, E \right)_{\mathfrak m} \\
H^* \left( \Hom_{\Lambda_{c,E} } \left( F_{\mathfrak m}^{\bullet} \otimes_{\Lambda_E} \Lambda_{c,E} , E \right) \right) \cong H^* \left( Y_{1,c}, E \right)_{\mathfrak m} \\
H^* \left( \Hom_{\Lambda_E } \left( F_{\mathfrak m}^{\bullet} , E \right) \right) \cong H^* \left( Y_{1,1}, E \right)_{ \mathfrak m} \end{split} \end{equation}
Tensoring with $E$ in proposition \ref{deriveddiamondintower} implies then that we obtain an action of $H^* \left( \mathrm T(\Zp)_p, E \right)$ on each eigenspace $H^* \left( Y_{1,1}, E \right)_{\mathfrak m}$.
As in the discussion surrounding proposition \ref{topologicalderivedaction}, this derived action can equivalently be interpreted as the natural action of $\Ext^*_{\Lambda_E} \left( E, E \right)$ on $\Ext^*_{\Lambda_E} \left( F_{\mathfrak m}^{\bullet}, E \right) \cong H^* \left( \Hom_{\Lambda_E} \left( F_{\mathfrak m}^{\bullet} ,E  \right) \right)$.

\section{Generation of cohomology over its bottom degree} \label{generationofcohomologysection}
Recall the definitions of the defect of $\mathrm G$: $l_0 = \rk \mathrm G(\R) - \rk K_{\infty}$ and $q_0 = \frac{d - l_0}{2}$, where $d$ is the dimension of $Y(K)$ as a real manifold.

Let $\pi$ be as in theorem \ref{autrepincohomology}, and let $w$ be any element of the Weyl group.
Then for any level $K \subset K_{1,1}$, with $K^p = I(1,1)$, the pair $(\pi, w)$ determines a maximal ideal $\mathfrak m = \widetilde {\mathfrak N_w}$ of $\mathbb T^S_{\ordi} (K, E)$ as in corollary \ref{maxidealsatproplevel}.
\begin{thm}[Franke, Borel, Matsushima et al.] \label{rationalcohomologyFranke} Let $\mathfrak m \subset \mathbb T^S_{\ordi} (K_{c,c}, E)$ be the maximal ideal corresponding to the pair $(\pi, w)$. 
Then for every $c \ge 1$ we have that $H^i \left( Y_{c,c}, E \right)_{\mathfrak m} \neq 0$ only if $i \in [q_0, q_0 + l_0]$, and moreover $\dim_E H^{q_0+i} \left( Y_{c,c}, E \right)_{ \mathfrak m} = m(\pi, K) \binom{l_0}{i}$ for some positive integer $m=m(\pi, K)$.
\end{thm}
Recall that $F^{\bullet}_{\mathfrak m}$ is the complex of of $\Lambda_E$-modules obtained as in the last subsection \ref{invertingpsubsec}, interpolating the $\mathfrak m$-eigenspaces in the ordinary tower.

Let us now assume the following dimension conjecture (related, but not equivalent, to conjecture \ref{dimensionconjectureforring} which takes a Galois-theoretic perspective), also known as a non-abelian Leopoldt conjecture (after Hida \cite{Hida}).
Recall that $\Lambda_E$ is a regular local ring of dimension $r = \rk \mathrm T$:
\begin{conj}[Dimension conjecture for ordinary complex] \label{dimensionconj}
We have \[ \dim_{\Lambda_E} H^* \left( F_{\mathfrak m}^{\bullet}  \right) = r - l_0. \]
\end{conj}
This conjecture allows us to use the following known lemma of Calegari and Geraghty (see \cite{CG} and corollary 3.2 in \cite{KT}):
\begin{lem} \label{CGlemma} Let $S$ be a regular local ring of dimension $d$ and let $0 \le l \le d$. Let $C^{\bullet}$ be a complex of finite-free $S$-modules concentrated in degrees $[q, q+l]$.
Then $\dim_S H^* \left( C^{\bullet} \right) \le d - l$. If equality holds, then there exists a unique nonzero cohomology group $H^i \left( C^{\bullet} \right)$ in degree $i = q +l$, and moreover we have $\projdim_S H^{q + l} \left( C^{\bullet} \right) = l$ and $\dpth_S H^{q+l} \left( C^{\bullet} \right) = d - l$.
\end{lem}
We want to apply this lemma to $F_{\mathfrak m}^{\bullet}$.
We have \[ H^* \left( \Hom_{\Lambda_E} \left( F^{\bullet}_{\mathfrak m} \otimes_{\Lambda_E} \Lambda_{c, E}, E \right) \right) \cong H^i \left( Y_{c,c}, E \right)_{\mathfrak m}, \]
and by theorem \ref{rationalcohomologyFranke} for any $c \ge 1$ the right side is only nonzero for $i \in [q_0, l_0 + q_0]$.

As in sections 2 and 6 of \cite{KT}, we have that $F_{\mathfrak m}^{\bullet}$ is a complex of finite, projective $\Lambda_E$-modules whose cohomology is only nonzero in degrees $[q_0, q_0+l_0]$ - by minimality, we conclude that $F_{\mathfrak m}^{\bullet}$ is concentrated in the same range of degrees.

We can thus apply lemma \ref{CGlemma}, and since the dimension conjecture \ref{dimensionconj} gives that $\dim_{\Lambda_E} H^* \left(  F^{\bullet}_{\mathfrak m}  \right) = r - l_0$, we obtain that $H^i \left( F^{\bullet}_{\mathfrak m}  \right)$ is only nonzero for $i = l_0 + q_0$, and that $H^{q_0+l_0} \left( F_{\mathfrak m}^{\bullet}  \right)$ has projective dimension equal to $l_0$ and depth equal to $r- l_0$ as a $\Lambda_E$-module.

\begin{thm}\label{complexes} Assume Conjecture  \ref{dimensionconj}, and that the multiplicity of $\pi$ in cohomology is $m(\pi,K) = 1$. The complex $F_{ \mathfrak m}^{\bullet} $ is quasi-isomorphic to the quotient of $\Lambda_E$ by a regular sequence of length $l_0$: \[ F_{ \mathfrak m}^{\bullet} \sim \Lambda_E / (f_1, \ldots, f_{l_0}). \]
\end{thm}
\begin{rem} The multiplicity one assumption is known to hold for $\mathrm G = \Gl{n}$.
\end{rem}
\begin{proof} The construction of the complex $F_{\mathfrak m}^{\bullet}$ is such that \[ H^* \left( \Hom_{\Lambda_E } \left( F_{\mathfrak m}^{\bullet}  , E \right) \right) \cong H^* \left( Y_{1,1}, E \right)_{\mathfrak m}, \] and by Franke's theorem \ref{rationalcohomologyFranke}, the right side is only nonzero in degrees $[q_0, q_0+l_0]$.

On the other hand, we have already seen that $F_{\mathfrak m}^{\bullet} $ is concentrated in degree $q_0+l_0$ and since taking $\Hom_{\Lambda_E}$ reverses the indexing on the complex, we obtain that \[ H^i \left( Y_{1,1}, E \right)_{\mathfrak m}  \cong \Ext_{\Lambda_E}^i \left( F_{\mathfrak m}^{\bullet} , E \right) \] is only nonzero for $i \in [0,l_0]$ and has dimension (over $E$) equal to $\binom{l_0}{i}$.

In particular, $\dim_E \Hom_{\Lambda_E} \left( F_{\mathfrak m}^{\bullet} , E \right) = 1$, which means that $F_{\mathfrak m}^{\bullet} $ is a cyclic $\Lambda_E$-module: there exists some ideal $I$ so that $F_{\mathfrak m}^{\bullet} \cong \Lambda_{\Qp} / I$.

Then, we obtain that \[ \Ext_{\Lambda_E}^1 \left( F_{\mathfrak m}^{\bullet} , E \right) \cong \Hom_{\Lambda_E} \left( I, E \right) \] has dimension $l_0$. This means that $I$ can be generated by $l_0$ elements $f_1, \ldots, f_{l_0}$.
Since \[ \dim \Lambda_E - l_0 = r - l_0 = \dim_{\Lambda_E}  F_{\mathfrak m}^{\bullet}   = \dim_{\Lambda_E} \Lambda_E / I = \dim_{\Lambda_E} \Lambda_E / \left( f_1, \ldots, f_{l_0} \right), \] the elements $f_1, \ldots, f_{l_0}$ are necessarily a regular sequence.
\end{proof}

We owe the proof of the  following result to N. Fakhruddin.

\begin{lem}\label{naf} Let $\Lambda$ be a regular local ring with maximal ideal $\mathfrak p$ and residue field $k$, $I = \left( x_1, \ldots, x_n \right)$ be an ideal generated by a regular sequence and let $R = \Lambda / I$.
The  map $\Ext_{\Lambda}^0 \left( R, k \right) \times \Ext^i_{\Lambda} (k,k) \lra \Ext_{\Lambda}^i \left( R, k \right)$ is  identified with  the map  of $k$-vector spaces $\bigwedge^i \Hom (\mathfrak p/\mathfrak p^2,k) \lra \bigwedge^i \Hom (I/\mathfrak pI,k)$.  Thus   the map is surjective for any $i>0$, if and only if   the regular sequence $\left( x_1, \ldots, x_n \right)$ is part of a system of parameters.
\end{lem}

\begin{proof} First, notice that $\Ext_{\Lambda}^0 \left( R, k \right) = \Hom_{\Lambda} (R, k) \cong k$, with a generator being the quotient map $R \twoheadrightarrow k$. Since $\Lambda$ is free over itself and $I$ is generated by a regular sequence, by a standard computation of Koszul complexes we have  isomorphisms \[ \begin{gathered} \Ext^i_{\Lambda} (k,k) \cong  \bigwedge^i \Hom_{\Lambda} \left( \mathfrak p / \mathfrak p^2 , k \right) \\ \Ext^1_{\Lambda} (R,k) \cong  \bigwedge^i \Hom_{\Lambda} \left( I / \mathfrak p I , k \right) \end{gathered}, \] and the  map $ \Ext_{\Lambda}^0 \left( R, k \right)  \times \Ext^i_{\Lambda} (k,k) \cong \Ext^i_{\Lambda}(k,k) \lra \Ext_{\Lambda}^1 \left( R, k \right)$ is  induced by the map $I / \mathfrak p I \lra \mathfrak p / \mathfrak p^2$.
The assumption about the action being surjective implies  the injectivity of $I / \mathfrak p I \lra \mathfrak p / \mathfrak p^2$, which is to say that any element in $I \cap \mathfrak p^2$ has to belong to $\mathfrak p I$.
In particular, no element $x_i$ of the regular sequence generating $I$ can belong to $\mathfrak p^2$ as that would imply $x_i \in \mathfrak p I$, contradicting regularity of the sequence. Regularity of the ring $\Lambda$ implies therefore that $\left( x_1, \ldots, x_n \right)$ can be extended to a regular system of parameters.
\end{proof}

The following result follows easily.

\begin{cor} \label{cohomologygeneratedbybottomdegree}  Assume Conjecture \ref{dimensionconj} and the multiplicity one conjecture $m(\pi, K)=1$.  We can choose $f_1, \ldots, f_{l_0}$ to be part of a system of parameters in $\Lambda_E$ if and only if  action of the derived diamond operators generates $H^* \left( Y_{1,1}, E \right)_{\mathfrak m} $ over the bottom degree $H^{q_0} \left( Y_{1,1}, E \right)_{\mathfrak m}$.
\end{cor}
\begin{proof}
We have  identified (in proposition \ref{topologicalderivedaction} and at the end of section \ref{sectionderivedactions}) the action of $\Ext^*_{\Lambda_E} \left( E, E \right)$ on $ \Ext^1_{\Lambda_E} \left( F_{\mathfrak m}^{\bullet} , E \right)$ with the derived Hecke action of $H^* \left( \mathrm T(\Zp)_p, E \right)$ on $H^{q_0} \left( Y_{1,1}, E \right)_{\mathfrak m}$.
Thus the action we are considering is the action $\Ext_{\Lambda}^i \left( R, k \right) \times \Ext^j_{\Lambda} (k,k) \lra \Ext_{\Lambda}^{i+j}\left( R, k \right)$.  We conclude the proof using Theorem \ref{complexes} and  Lemma \ref{naf}.
\end{proof}

\section{Galois cohomology and Selmer groups} \label{deformationsection}

In this section we recall (following Patrikis \cite{Patrikis}) several results on $\check {\mathrm G}$-valued deformations and the relevant Selmer groups, which will be used in slightly different setups in the rest of the paper.

For this section, we fix a Galois representation $\sigma: \Gal_{\Q} \lra \check {\mathrm G}(k)$, where $k$ could be a finite extension of either $\Fp$ or $\Qp$. We assume that $\sigma$ is unramified outside a finite set of primes $S$.
We let $A = \left\{ \begin{array}{cc} W(k) & \textnormal{ if } k \supset \Fp \\ k & \textnormal{ if } k \supset \Qp \end{array} \right.$, and we will study deformations of $\sigma$ into local, complete, Noetherian $A$-algebras with residue field $k$.

\subsection{Local deformations}
We record results on the types of local conditions we will use. Let $q$ be any prime.

We denote by $\Lift_{\sigma_q}: \mathcal C_A^f \lra \Sets$ the functor defined by letting $\Lift_{\sigma_q}(R)$ be the set of lifts $\widetilde {\sigma_q}: \Gal_{\Q_q} \lra \check {\mathrm G}(R)$ for every Artinian $A$-algebra $R$.  (We denote by $G_q=\Gal_{\Q_q}$ a decomposition group at $q$.) The tangent space $\mathfrak t_q^{\Box}$ to $\Lift_{\sigma_q}$ is the vector space of cocycles $Z^1 \left( \Q_q, \Lie \check {\mathrm G} \right)$.
\begin{defn} Two lifts $\widetilde {\sigma_q}, \widetilde {\sigma_q}': \Gal_{\Q_q} \lra \check {\mathrm G}(R)$ are said to be \emph{strictly equivalent} if they are conjugate under an element of $\ker \left( \check{\mathrm G}(R) \lra \check{\mathrm G}(k) \right)$.
\end{defn}
The notion of strict equivalence defines an equivalence relation on $\Lift_{\sigma_q}(R)$.
\begin{defn}[Definitions 3.1, 3.2 and 3.5 in \cite{Patrikis}] A \emph{local deformation condition} is a representable subfunctor $\mathcal L_q^{\star}$ of $\Lift_{\sigma_q}$ which is closed under strict equivalence.

A local deformation condition gives rise to a \emph{local deformation functor} $\mathcal D_q^{\star}: \mathcal C^A_f \lra \Sets$, which associated to every Artinian $A$-algebra $R$ the set of strict equivalence classes of elements of $\mathcal L_q^{\star}(R)$.

A local deformation condition $\mathcal L_q^{\star}$ is called \emph{liftable} if its (framed) representation ring $R_q^{\Box, \star}$ is isomorphic to a power series ring over $A$ in $\dim_k \mathfrak t_q^{\Box, \star}$ variables, where $\mathfrak t_q^{\Box, \star} = \mathcal L_q^{\star} \left( k [\varepsilon] \right) = \Hom_{\mathcal C_A} \left( R_q^{\Box, \star}, k [ \varepsilon] \right)$ is the tangent space of $L_q^{\star}$.\footnote{As usual, $k[\varepsilon]$ is the ring of dual numbers.}
\end{defn}

\subsubsection{Deformations at $p$} \label{localatpdeformationsection}
\begin{defn} \label{defnordicrys} We say that $\sigma$ is \emph{ordinary at $p$} if there exists a Borel subgroup $\check {\mathrm B}$ of $\check {\mathrm G}$ defined over $k$ and such that $\im \left( \sigma_p \right) \subset \check{\mathrm B}(k)$.

We say that $\sigma_p$ is \emph{crystalline} if the associated adjoint representation \[ \Ad \sigma_p = \Ad \circ \sigma_p: \Gal_{\Qp} \stackrel{\sigma_p}{\lra} \check {\mathrm G}(k) \stackrel{\Ad}{\lra} \mathrm{GL} \left( \Lie \check{\mathrm G} \right) \] is crystalline in the usual sense.
\end{defn}
Recall that for every representation $V$ of $\Gal_{\Qp}$, the cohomology group $H^1_f \left( \Qp, V \right)$ classifies crystalline extensions.
Explicitly, this \emph{crystalline subgroup} is defined as \[ H^1_f \left( \Qp, V \right) = \ker \left( H^1 \left( \Qp, V \right) \lra H^1 \left( \Qp, V \otimes_{\Qp} B_{\crys} \right) \right).\footnote{$B_{\crys}$ is a period ring from $p$-adic Hodge theory.} \]
If $V$ is a $p$-adic crystalline representation of $\Gal_{\Qp}$, then an extension \[ 0 \lra V \lra W \lra \Qp \lra 0 \] is crystalline if and only if (the class of) $W$ belongs to the subgroup $H^1_f(\Qp, V)$.
\begin{rem} We recall that the subgroup $H^1_f$ is the correct replacement for the classical unramified subgroup in the setting of $p$-adic representation, in the sense that $H^1_f$ has the following important property: under local Tate duality between $V$ and $V^*(1)$, the respective $H^1_f$-subgroups are annihilators of each other.
Sometimes, the property of being in $H^1_f$ is called `crystalline Selmer condition'.
\end{rem}

We assume for the rest of this section that $\sigma$ is ordinary at $p$, and we fix a Borel subgroup $\check {\mathrm B}$ whose $k$-points contain the image of $\sigma_p$. As explained in \cite{Conrad1,Conrad2}, we can fix a smooth $A$-model $\check {\mathrm B}$, now a Borel subgroup of the $A$-group $\check {\mathrm G}$.

The ordinary assumption implies in particular that $\Lie \check {\mathrm B}$ is a $\Gal_{\Qp}$-subrepresentation of $\Lie \check {\mathrm G}$.
\begin{defn}[Ordinary deformation] Let $\Lift_{\sigma_p}^{\ordi}: \mathcal C_A^f \lra \Sets$ be the functor defined as follows: for any Artinian $A$-algebra $R$ we let $\Lift_{\sigma_p}^{\ordi} (R) $ be the set of representations $\widetilde {\sigma_p}: \Gal_{\Qp} \lra \check {\mathrm G}(R)$ lifting $\sigma_p$ and such that there exists $g \in \ker \left( \check {\mathrm G} (R) \lra \check {\mathrm G}(k) \right)$ with $g \widetilde \sigma_p g^{-1} \subset \check {\mathrm B}(R)$.
\end{defn}
\begin{rem} This is a subfunctor of $\Lift_{\sigma_p}$. Notice that our definition of ordinary deformation is different than definition 4.1 in \cite{Patrikis}, because we do not impose (for now) any condition on the action of inertia on the torus quotient. However, most of the proofs of \cite{Patrikis} section 4.1 hold verbatim for our condition as well.
\end{rem}
From now on, we assume that \begin{equation}\tag{REG} \label{regcondition} H^0 \left( \Qp, \frac{\Lie \check {\mathrm G}}{ \Lie \check {\mathrm B}} \right) = 0. \end{equation}
This condition is denoted (REG) also in \cite{Patrikis}, and it is some sort of `big image' condition, because it is saying that $\sigma_p \left( \Gal_{\Qp} \right)$ is a subspace of $\check {\mathrm B}(k)$ large enough to not have any invariants on $\Lie \check {\mathrm G} / \Lie \check {\mathrm B}$ (under the adjoint action).
\begin{prop}[Lemma 4.2 in \cite{Patrikis}] The functor $\Lift_{\sigma_p}^{\ordi}$ is closed under strict equivalence and it is representable.
\end{prop}
This result says that $\Lift_{\sigma_p}^{\ordi}$ is a local deformation condition.
Therefore we can consider the associated deformation functor $\Def_{\sigma_p}^{\ordi}: \mathcal C_A^f \lra \Sets$: for any Artinian $A$-algebra $R$,  $\Def_{\sigma_p}^{\ordi} (R) $ is the set of strict equivalence classes of elements in $\Lift_{\sigma_p}^{\ordi}(R)$.
\begin{prop}[Lemma 4.3 in \cite{Patrikis}] The tangent space $\mathfrak t_p^{\ordi}$ to $\Def_{\sigma_p}^{\ordi}$ is $H^1 \left( \Qp, \Lie \check {\mathrm B} \right)$.
\end{prop}
To proceed, we need to assume the dual of condition \ref{regcondition}: \begin{equation}\tag{REG*} \label{regdualcondition} H^0 \left( \Qp, \left( \frac{ \Lie \check {\mathrm G} }{ \Lie \check {\mathrm B}  } \right) (1) \right) = 0. \end{equation}
\begin{prop}[Proposition 4.4 in \cite{Patrikis}] \label{dimensionlocalordinaryring} The tangent space $\mathfrak t_p^{\ordi}$ has dimension ${\dim_k \left( \Lie \check {\mathrm B} \right) + \dim_k H^0 \left( \Qp, \Lie \check {\mathrm G} \right)}$, and the (framed) deformation ring $R_p^{\Box, \ordi}$ representing $\Lift_{\sigma_p}^{\ordi}$ is a power series ring over $A$ in $\dim_k \left( \Lie \check {\mathrm G} \right) + \dim_k \left( \Lie \check {\mathrm B} \right)$ variables.

In particular, $\Lift_{\sigma_p}^{\ordi}$ is a liftable local deformation condition.
\end{prop}
\begin{defn}[Crystalline deformation] Let $\Lift_{\sigma_p}^{\crys}: \mathcal C_A^f \lra \Sets$ be the functor defined as follows: for any Artinian $A$-algebra $R$ we let $\Lift_{\sigma_p}^{\crys} (R) $ be the set of representations $\widetilde {\sigma_p}: \Gal_{\Qp} \lra \check {\mathrm G}(R)$ lifting $\sigma_p$ and such that they are crystalline (in the sense of definition \ref{defnordicrys}).
\end{defn}
It is well-known that $\Lift_{\sigma_p}^{\crys}$ is a local deformation condition, and that the tangent space of the associated deformation functor $\Def_{\sigma_p}^{\crys}$ is $\mathfrak t_p^{\crys} = H^1_f \left( \Qp, \Lie \check {\mathrm G} \right)$.

Finally, we want to recall a result relating ordinary representations of a certain weight to crystalline representations. This is very similar to lemma 4.8 of \cite{Patrikis}, and follows from the results of \cite{PerrinRiou} (see in particular proposition 3.1).

Notice that since every deformation $\widetilde \sigma_p \in \Def_{\sigma_p}^{\ordi} (R)$ can be conjugated into $\check {\mathrm B}(R)$, we can quotient by the unipotent radical $\check {\mathrm U}(R)$ and by local class field theory we obtain a map $\widetilde \sigma_p^{\ab} : \Qp^* \lra \check {\mathrm T}(R)$. This map does not depend on the choice of the normalizing element, since $\check {\mathrm B}$ is its own normalizer inside $\check {\mathrm G}$.
\begin{lem} \label{ordithencrys} Suppose that $\alpha \circ \widetilde \sigma_p^{\ab} \neq \omega$  with $\omega$ the
(mod $p$ or $p$-adic)  cyclotomic character  for every simple root $\alpha \in \check \Delta \subset \check \Phi^+ \left( \check{\mathrm G}, \check {\mathrm T} \right)$. Then the deformation $\widetilde \sigma_p$ is crystalline.
\end{lem}
\begin{rem} In our applications we will consider representations that satisfy  $\alpha \circ \widetilde \sigma_p^{\ab} \neq \omega$, but $(\alpha \circ \widetilde \sigma_p^{\ab})|_{I_p}= \omega$   for every simple root $\alpha \in \check \Delta \subset \check \Phi^+ \left( \check{\mathrm G}, \check {\mathrm T} \right)$.

\end{rem}

\subsubsection{Deformations at $q \neq p$}
At finite primes different from $p$, we often impose no condition at all - in this case $\mathcal L_q = \Lift_{\sigma_q}$.

If $\sigma_q$ is unramified, we may consider unramified lifts: $\Lift_{\sigma_q}^{\unr}$ is the subfunctor of $\Lift_{\sigma_q}$ such that $\Lift_{\sigma_q}^{\unr}(R)$ is the set of unramified lifts $\widetilde {\sigma_q}: \Gal_{\Q_q} \lra \check {\mathrm G}(R)$ for every Artinian $A$-algebra $R$.

We record the well known fact that $\Lift_{\sigma_q}^{\unr}$ is a local deformation condition, and the associated deformation functor $\Def_{\sigma_q}^{\unr}$ has tangent space being the kernel of the restriction map \[\mathfrak t_q^{\unr} = \ker \left( H^1(\Q_q, \Lie  \check G) \lra H^1 \left( I_{\Q_q}, \Lie  \check  G\right) \right)\] ($I_{\Q_q}$  is the inertia subgroup of $\Gal_{\Q_q}$) denoted $H^1_{\unr} \left( \Q_q, \Lie \check G \right)$.

We also recall that the unramified condition at $q \neq p$ is self-dual with respect to local Tate duality - which is to say, the annihilator of $H^1_{\unr} \left( \Q_q, V \right)$ inside $H^1 \left( \Q_q, V^*(1) \right)$ is $H^1_{\unr} \left( \Q_q, V^*(1) \right)$ for any $\Gal_{\Q_q}$-module $V$.
\begin{defn} We say that $\sigma_q$ is generic if $H^2(\Q_q, \Lie \check G)=0$.
\end{defn}
\begin{defn}[Inertial level] A tame  group homomorphism $f: \Gal_{\Q_q} \lra H$ has \emph{inertial level $n$} if the subgroup of $p^n$-powers of the inertia group $I_{\Q_q}$ is contained in $\ker f$.
\end{defn}
The Artin reciprocity  isomorphism implies that if $f: \Gal_{\Q_q} \lra H$ has inertial level $n$, then the homomorphism induced between the abelianizations $\Q_q^* \cong \Gal_{\Q_q}^{\ab} \stackrel{f^{\ab}}{\lra} H^{\ab}$ contains $\Z_q^* / p^n \cong \mathbb F_q^* / p^n$ in its kernel.

We can then consider another type of local deformation condition: suppose that $\sigma_q$ has inertial level $m$ and that $n \ge m$.
We consider the subfunctor $\Lift_q^{\le n}$ of $\Lift_{\sigma_q}$ defined on an Artinian $A$-algebra $R$ as follows: $\Lift_{\sigma_q}^{\le n}(R)$ is the set of lifts $\widetilde {\sigma_q}: \Gal_{\Q_q} \lra \check {\mathrm G} (R)$ having inertial level $n$.

\subsubsection{Deformations at $\infty$}
The Galois representations we will consider are odd in the following sense.
\begin{defn}[definition 2.1 in \cite{Calegari}] \label{oddgaloisrep} Let $\sigma: \Gal_{\Q} \lra \check {\mathrm G}(R)$ be a continuous Galois representation.
$\sigma$ is said to be \emph{odd} or \emph{$\check {\mathrm G}$-odd} if for any complex conjugation $c$, the adjoint action of the involution $\sigma(c) \in \check {\mathrm G}(R)$ on the $R$-Lie algebra $\Lie \check {\mathrm G}$ has minimal trace among all involutions.
\end{defn}
As explained in section 2 of \cite{Calegari}, if $R=k$ this condition is equivalent to \[ \dim_k H^0 \left( \R, \Lie \check{\mathrm G} \right) = \dim_k \Lie \check{\mathrm U} + l_0, \] where $l_0$ is the defect of $\mathrm G$.
\begin{rem} Notice that since $p \neq 2$, any deformation of a $\check {\mathrm G}$-odd representation is automatically $\check {\mathrm G}$-odd.
\end{rem}

\subsection{Global deformations} \label{globaldeformationsubsection}
We consider a finite set of primes $\Sigma$, containing all archimedean primes as well as all primes where $\sigma$ is ramified, and the prime $p$.\footnote{In future work, we will choose $\Sigma = S \cup Q$, where $Q$ is a Taylor-Wiles datum.}  We assume that $\sigma$ is {\it generic}  at all the places in $\Sigma$ away from $p$ and $\infty$, besides the conditions of  $(REG)$ and $(REG^*)$ at $p$, and oddness at infinity.

We can consider then $\sigma: \Gal_{\Q,\Sigma} \lra \check {\mathrm G}(k)$ to be a representation of the Galois group $\Gal_{\Q, \Sigma} = \Gal (\Q^{\Sigma} / \Q)$ of the maximal extension of $\Q$ unramified outside $\Sigma$.
As usual, the advantage of using this smaller Galois group is that it satisfies Schlessinger's representability criterion.

Let $\Lift_{\sigma}: \mathcal C_A^f \lra \Sets$ be the functor defined as follows: for each Artinian $A$-algebra $R$, $\Lift_{\sigma}(R)$ is the set of all group homomorphisms $\widetilde \rho: \Gal_{\Q, \Sigma} \lra \check {\mathrm G}(R)$ lifting $\sigma$.

Fix now liftable local deformation conditions $\mathcal L_q^{\mathcal P_q} \subset \Lift_{\sigma_q}$ for each $q \in \Sigma$, with $\mathcal L_p^{\mathcal P_p} = \Lift_{\sigma_p}^{\ordi}$.

Let $\mathcal L_{\sigma}^{\mathcal P}$ be the subfunctor of $\Lift_{\sigma}$ such that for any Artinian $A$-algebra $R$, $\mathcal L_{\sigma}^{\mathcal P} (R)$ is the set of lifts $\widetilde \sigma: \Gal_{\Q, \Sigma} \lra \check {\mathrm G}(R)$ such that ${\widetilde \sigma}_q \in \mathcal L_q^{\mathcal P_q}(R)$ for all $q \in \Sigma$.
Similarly, we denote by $\mathcal D_{\sigma}^{\mathcal P}$ the deformation functor associated to $\mathcal L_{\sigma}^{\mathcal P}$, which is to say that $\mathcal D_{\sigma}^{\mathcal P}(R)$ is the set of strict equivalence classes of lifts in $\mathcal L_{\sigma}^{\mathcal P}(R)$ for each Artinian $A$-algebra $R$.
\begin{prop} \label{globaldeformationrepresentable} Suppose that $H^0 \left( \Gal_{\Q, \Sigma}, \Lie \check {\mathrm G} \right)=0$. Then $\mathcal D_{\sigma}^{\mathcal P}$ is representable by a ring $R_{\sigma, \Sigma}^{\mathcal P}$.
\end{prop}
\begin{proof} This is the first part of propositions 3.6 of \cite{Patrikis} - see also the discussion at the beginning of section 3.3 of loc. cit.
\end{proof}
Let now $\mathfrak t_q^{\Box, \mathcal P_q} \subset Z^1 \left( \Q_q, \Lie \check {\mathrm G} \right)$ be the tangent space of the subfunctor $\mathcal L_q^{\mathcal P_q}$, and let us denote by $\mathfrak t_q^{\mathcal P_q}$ its image in $H^1 \left( \Q_q, \Lie \check {\mathrm G} \right)$.

We define a Selmer complex using a cone construction, as in \cite{Patrikis}, section 3.2 or \cite{GV}, appendix B: we first set \[ L_q^{\Box, i} = \left\{ \begin{array}{ccc} C^0 \left( \Q_q, \Lie \check {\mathrm G} \right) & \textnormal{ if } & i=0 \\ \mathfrak t_q^{\Box, \mathcal P_q} & \textnormal{ if } & i=1 \\ 0 & \textnormal{ if } & i \ge 2 \end{array} \right. \] then the cone on the restriction map \[ \res{}{}: C^{\bullet} \left( \Gal_{\Q, \Sigma} , \Lie \check {\mathrm G} \right) \lra \bigoplus_{q \in \Sigma} C^{\bullet} \left( \Q_q , \Lie \check {\mathrm G} \right) / L_q^{\Box, \bullet} \] yields a complex which we denote $C^{\bullet}_{\mathcal P} \left( \Gal_{\Q, \Sigma}, \Lie \check {\mathrm G} \right)$.

Taking cohomology of this cone complex yields a long exact sequence \begin{equation} \label{LESfromcone} 0 \lra H^1_{\mathcal P} \left( \Gal_{\Q, \Sigma}, \Lie \check {\mathrm G} \right) \lra H^1 \left( \Gal_{\Q, \Sigma}, \Lie \check {\mathrm G} \right) \lra \bigoplus_{q \in \Sigma} H^1 \left( \Q_q, \Lie \check {\mathrm G} \right) / \mathfrak t_q^{\mathcal P_q} \lra H^2_{\mathcal P} \left( \Gal_{\Q, \Sigma}, \Lie \check {\mathrm G} \right) \lra \ldots \end{equation} which instrinsically defines the Selmer group $H^1_{\mathcal P} \left( \Gal_{\Q, \Sigma}, \Lie \check {\mathrm G} \right)$.

On the other hand, local Tate duality gives us dual conditions $\mathfrak t_q^{\mathcal P_q^{\perp}} \subset H^1 \left( \Q_q, \left( \Lie \check {\mathrm G} \right)^* (1) \right)$, and a dual Selmer group: \begin{equation} \label{dualSelmerdef} H^1_{\mathcal P^{\perp}} \left( \Gal_{\Q, \Sigma}, \left( \Lie \check {\mathrm G} \right)^* (1) \right) = \ker \left( \res{}{}: H^1 \left( \Gal_{\Q, \Sigma}, \left( \Lie \check {\mathrm G} \right)^* (1) \right) \lra \bigoplus_{q \in \Sigma} H^1 \left( \Q_q, \left( \Lie \check {\mathrm G} \right)^* (1) \right) / \mathfrak t_q^{\mathcal P_q^{\perp}} \right). \end{equation}
The cohomology groups involved are then related by the following classical results:
\begin{prop} \label{deformationringsdimensions}  $ $  \begin{enumerate}
    \item \[ \dim_k H^2_{\mathcal P} \left( \Gal_{\Q, \Sigma}, \Lie \check {\mathrm G} \right) = \dim_k H^1_{\mathcal P^{\perp}} \left( \Gal_{\Q, \Sigma}, \left( \Lie \check {\mathrm G} \right)^* (1) \right). \]
    \item \[ \begin{gathered} \dim_k H^1_{\mathcal P} \left( \Gal_{\Q, \Sigma}, \Lie \check {\mathrm G} \right) - \dim_k H^1_{\mathcal P^{\perp}} \left( \Gal_{\Q, \Sigma}, \left( \Lie \check {\mathrm G} \right)^* (1) \right) = \\ = \dim_k H^0 \left( \Q, \Lie \check {\mathrm G} \right) - \dim_k H^0 \left( \Q, \left( \Lie \check {\mathrm G} \right)^* (1) \right) + \sum_{q \in \Sigma} \left( \dim_k \mathfrak t_q^{\mathcal P} - \dim_k H^0 \left( \Q_q, \Lie \check {\mathrm G} \right) \right). \end{gathered} \]
    \item Suppose $H^1_{\mathcal P^{\perp}} \left( \Gal_{\Q, \Sigma}, \left( \Lie \check {\mathrm G} \right)^* (1) \right) = 0$. Then $R_{\sigma, \Sigma}^{\mathcal P}$ is a power series ring over $A$ in $\dim_k H^1_{\mathcal P} \left( \Gal_{\Q, \Sigma}, \Lie \check {\mathrm G} \right)$ variables. We refer to this by saying that the deformation problem $\mathcal P$ is smooth.
\end{enumerate}
\end{prop}
\begin{proof} See propositions 3.9 and 3.10 and corollary 3.11 of \cite{Patrikis}.
\end{proof}

\section{Galois cohomology associated to  $\rho$ that arise from cuspidal  automorphic representations $\pi$} \label{Galoisrepforautrep}

\subsection{The representation $\rho_\pi$}
Recall that $\pi$ is a tempered, cuspidal, automorphic representation of $\mathrm G$ on an $E$-vector space which is  unramified, strongly generic and ordinary at $p$ as in Theorem \ref{autrepincohomology},  $\pi_q$ is generic at all  finite places  $q \neq p$  at which $\pi_q$ is ramified, and $\pi$ appears in the cohomology $H^* \left( Y_{1,1}, E \right)_{\ordi}$.   Recall from Definition \ref{stronglygenericrep} that strong genericity of $\pi$ at $p$ corresponds to the fact that $\pi_p$ is unramified, and the corresponding character $\chi: \mathrm T(\Qp) / \mathrm T(\Zp) \lra E^*$ is strongly regular. (Conjecturally $\pi_q$ being generic   is equivalent to its (conjectural)  Langlands parameter, a   Weil-Deligne representation  $(r_q,N_q)$,  being  generic, i.e.,  there is no nontrivial morphism  $(r_q,N_q) \longrightarrow (r_q(1),N_q)$. This is known in the special case when $G=GL_n$, see \S 1.1 of \cite{Allen}.)

Given a Weyl group element $w$, we obtain a maximal ideal $\mathfrak m = \widetilde{\mathfrak N_w}$ of the ordinary Hecke algebra $\mathbb T^S_{\ordi}(K_{1,1}, E)$ away from $S$, associated to $\pi$ and $w$ as in corollary \ref{maxidealsatproplevel}.
Recall that $\check{\mathrm G} $ is a smooth $\Zp$-model of the dual group of $\mathrm G$.

We assume  from here on the following conjecture.  This is known in many cases when $G=GL_n$, see \cite{tenauthors} and its references.

\begin{conj} \label{Galoisrepexists1} [see conjecture 6.1 of \cite{GV}] There exists a Galois representation \[ \rho =\rho_{\pi}: \Gal_{\Q} \lra \check {\mathrm G} (E) \] with the following properties: \begin{enumerate}
    \item For all primes $q \neq p$ where $K_q$ is unramified, $\rho_{\pi}|_{G_q}$ is unramified and we have the usual Hecke-Frobenius compatibility: for each algebraic representation $\tau$ of $\mathrm{ \check G}$, the trace of $\left( \tau \circ \rho_{\pi} \right)(\Frob_q)$ coincide with the image of $T_{\tau,q}$ under the surjection $\mathbb T^S_{\ordi}(K(1,1), E) \twoheadrightarrow \mathbb T^S_{\ordi}(K(1,1), E) / \mathfrak m$, and $T_{\tau,q}$ corresponds to $\tau \in \Rep_E(\check G)$ via the Satake isomorphism.
    \item $\rho_{\pi}$ is $\check {\mathrm G}$-odd, as in definition \ref{oddgaloisrep}.
    \item The restriction ${\rho_\pi}|_p={\rho_{\pi}}|_{G_p}$ of $\rho_{\pi}$ to a  decomposition group $G_p$ at $p$ is ordinary and crystalline. Further
    $(\alpha \circ ({\rho_\pi}|_p)^{\ab})  \neq \omega$ and $(\alpha \circ ({\rho_\pi}|_p)^{\ab})|_{I_p}=\omega$ for for every simple root $\alpha \in \check \Delta \subset \check \Phi^+ \left( \check{\mathrm G}, \check {\mathrm T} \right)$.
\item  for all primes $q \neq p$ where $K_q$ is ramified, $\rho_{\pi}|_{G_q}$ is generic,  namely  $H^2(\Q_q, \Lie \mathrm {\check G} )=0$,
where $ \Lie \mathrm {\check G}$ is as before the adjoint representation associated to  $\rho_{\pi}|_{G_q}$.
\item $\rho_{\pi}$ is absolutely irreducible, in the sense that its image is not contained in $\check {\mathrm P} (\overline E)$ for any proper parabolic subgroup $\check {\mathrm P}$ of $\check{\mathrm G}_{\overline E}$.
\end{enumerate}
\end{conj}

We note the following useful lemma.

\begin{lem}\label{useful}
 The representation ${\rho_\pi}|_{G_p}$ satisfies the condition $(REG^*)$.
\end{lem}

\begin{proof}
 This is an immediate consequence of   the condition $(\alpha \circ ({\rho_\pi}|_p)^{\ab})  \neq \omega$ and $(\alpha \circ ({\rho_\pi}|_p)^{\ab})|_{I_p}=\omega$ for for every simple root $\alpha \in \check \Delta \subset \check \Phi^+ \left( \check{\mathrm G}, \check {\mathrm T} \right)$, which implies that \[H^0 \left( \Qp, \left( \frac{ \Lie \check {\mathrm G} }{ \Lie \check {\mathrm B}  } \right) (1) \right) = 0,\] and hence that ${\rho_\pi}|_{G_p}$ satisfies the condition $(REG^*)$.
\end{proof}

\subsection{Local theory of ordinary representations at $p$}
The standard local-globcal compatibility assumptions in conjecture \ref{Galoisrepexists1} predict that $\rho$ is unramified away from $S$, and ordinary at $p$ (because of the corresponding assumptions on $\mathfrak m$ and thus $\pi$). We thus fix once and for all a Borel subgroup $\check {\mathrm B}$ so that $\rho_p$ maps into $\check {\mathrm B}(E)$.

Consider the adjoint representation $\Ad \rho$ with underlying $E$-vector space $\Lie \check{\mathrm G}$. If we restrict this to $\Gal_{\Qp}$, thanks to the ordinary assumption we obtain a subrepresentation $\Lie \check{\mathrm B}$. Notice that as $\check {\mathrm U}$ is normal in $\check {\mathrm B}$, we have another subrepresentation $\Lie \check {\mathrm U}$.

Standard computations similar to those in section 8.14 of \cite{akshay} show that we have isomorphisms \[ \begin{gathered} \Hom \left( \Zp^*, \Lie \check {\mathrm T} \right) \cong \Hom \left( \Zp^*, X_* \left( \check {\mathrm T} \right) \otimes_{\Z} E \right)  \cong \Hom \left( \Zp^*, E \right) \otimes_{\Z} X_* \left( \check {\mathrm T} \right) \cong \\ \cong  \Hom \left( \Zp^* \otimes_{\Z} X_*(\mathrm T) , E \right) \cong \Hom \left( \mathrm T(\Zp), E \right). \end{gathered} \]

We also have a map (denoted $\varepsilon$ in section 8.3 of  \cite{CM}) \[ \varepsilon: H^1 \left( \Qp, \Lie \check {\mathrm B} \right) \lra \Hom \left( \Qp^*, \Lie \check {\mathrm T} \right) \] defined as follows: we have a surjection of Lie algebras \[ \Lie \check {\mathrm B} \twoheadrightarrow \frac{\Lie \check {\mathrm B}}{ \Lie \check {\mathrm U}} \cong \Lie \check {\mathrm T} \] which is Galois-equivariant and thus induces a map in cohomology \[ H^1 \left( \Qp, \Lie \check {\mathrm B} \right) \lra H^1 \left( \Qp, \Lie \check {\mathrm T} \right) \]
Since the coadjoint action of $\check {\mathrm B}(E)$ on $\Lie \check {\mathrm T}$ is trivial, the target cohomology group identifies with $H^1 \left( \Qp, \Lie \check {\mathrm T} \right) \cong \Hom \left( \Qp^*, \Lie \check {\mathrm T} \right)$.

The subgroup $H^1_f \left( \Qp, \Lie \check {\mathrm B} \right)$ of crystalline extensions maps to $H^1_f \left( \Qp, \Lie \check {\mathrm T} \right)$.
A routine computation (see for instance \cite{feng}) shows that $H^1_f \left( \Qp, \Lie \check{\mathrm T} \right)$ picks out the unramified representations, so by fixing a splitting $\Qp^* \cong \Zp^* \times \Z$ we can identify \[ \frac{H^1 \left( \Qp, \Lie \check {\mathrm T}  \right) }{ H^1_f \left( \Qp, \Lie \check {\mathrm T} \right)} \cong \frac{\Hom \left( \Qp^*, \Lie \check {\mathrm T}  \right) }{ \Hom \left( \Z, \Lie \check {\mathrm T}  \right)} \cong \Hom \left( \Zp^*, \Lie \check {\mathrm T}   \right) \] and we obtain a map \[ \frac{H^1 \left( \Qp, \Lie \check {\mathrm B} \right)}{H^1_f \left( \Qp, \Lie \check {\mathrm B} \right)} \lra \Hom \left( \Zp^*, \Lie \check{\mathrm T} \right) \cong \Hom \left( \mathrm T(\Zp), E \right). \]

\begin{prop} \label{ordivscrys} Suppose that condition  (\ref{regdualcondition}) holds. Then the  map   \[ \frac{H^1 \left( \Qp, \Lie \check {\mathrm B} \right)}{H^1_f \left( \Qp, \Lie \check {\mathrm B} \right)} \lra  \Hom \left( \mathrm T(\Zp), E \right) \]  defined above  is an isomorphism.
\end{prop}
\begin{proof}
The map is obviously $E$-linear, because it is induced by Galois cohomology maps on $E$-linear representations.

We start by proving surjectivity of the map \begin{equation} \label{crystallinequotientmap} \frac{H^1 \left( \Qp, \Lie \check {\mathrm B} \right)}{H^1_f \left( \Qp, \Lie \check {\mathrm B} \right)} \lra \frac{H^1 \left( \Qp, \Lie \check {\mathrm T} \right) }{ H^1_f \left( \Qp, \Lie \check {\mathrm T} \right)}, \end{equation} and then we will show that the domain and target are vector spaces of the same dimension.

We have an exact sequence of $p$-adic representations of $\Gal_{\Qp}$: \[ 1 \lra \Lie \check{\mathrm U}  \lra \Lie \check {\mathrm B} \lra \Lie \check{\mathrm T}  \lra 1. \]
The induced long exact sequence in Galois cohomology reads \[ \ldots \lra H^1 \left( \Qp, \Lie \check {\mathrm B} \right) \lra H^1 \left( \Qp, \Lie \check {\mathrm T}  \right) \lra H^2 \left( \Qp, \Lie \check{\mathrm U} \right) \lra \ldots \]
As explained in \cite{Patrikis}, proof of proposition 4.4, the Killing form on $\mathrm G$ gives a Galois-equivariant non-degenerate pairing $\Lie \check{\mathrm U} \times \frac{\Lie \check{\mathrm G}}{\Lie \check {\mathrm B}} \lra E$.
Therefore, $\left( \Lie \check{\mathrm U} \right)^* \cong \frac{\Lie \check{\mathrm G}}{\Lie \check {\mathrm B}}$ and local Tate duality yields that \[ H^2 \left( \Qp, \Lie \check {\mathrm U}  \right) \cong H^0 \left( \Qp, \left( \Lie \check{\mathrm U} \right)^* (1) \right)^{\vee} = H^0 \left( \Qp, \left( \frac{\Lie \check{\mathrm G}}{\Lie \check{\mathrm B}} \right) (1) \right)^{\vee} = 0, \] where the last equality follows from condition (REG*).

We thus have a surjection \[ H^1 \left( \Qp, \Lie \check {\mathrm B} \right) \twoheadrightarrow H^1 \left( \Qp, \Lie \check {\mathrm T} \right) \]
and therefore the map in equation \ref{crystallinequotientmap} is surjective as well, since it preserves the $f$-subgroups.

It remains to show that domain and target of equation \ref{crystallinequotientmap} have the same dimension over $E$.
Consider first the target space $\frac{H^1 \left( \Qp, \Lie \check {\mathrm T} \right) }{ H^1_f \left( \Qp, \Lie \check {\mathrm T} \right)}$. We proved above that it is isomorphic to $\Hom \left( \mathrm T(\Zp), E \right)$ and hence its dimension over $E$ is equal to $\rk \mathrm T$, the rank of the $\Zp$-torus $\mathrm T$.

Recall from the Euler characteristic formula (see e.g. \cite{feng}) that for any $p$-adic field $F$ and any $\Gal_F$-representation $V$ on a $\Qp$-vector space one has \[ \dim_{\Qp} H^1 \left( \Gal_F, V \right) = \dim_{\Qp} H^2 \left( \Gal_F, V \right) + [F: \Qp] \cdot \dim_{\Qp} V + \dim_{\Qp} H^0 \left( \Gal_F, V \right) \] while (e.g. proof of theorem 4.15 in \cite{feng}) \[ \dim_{\Qp} H^1_f \left( \Gal_F, V \right) = \dim_{\Qp} H^0 \left( \Gal_F, V \right) + [F: \Qp] \cdot \left| \left\{ \textnormal{negative Hodge-Tate weights in $V$} \right\} \right|, \] where we count the negative Hodge-Tate weights on $V$ with multiplicity.

In our case, we obtain \[ \dim_{\Qp} \frac{H^1 \left( \Qp, \Lie \check {\mathrm B} \right)}{H^1_f \left( \Qp, \Lie \check {\mathrm B} \right)} = \] \[ = \dim_{\Qp} H^2 \left( \Qp, \Lie \check {\mathrm B} \right)  + \dim_{\Qp} \left( \Lie \check {\mathrm B} \right) - \left| \left\{ \textnormal{negative Hodge-Tate weights in $\Lie \check {\mathrm B}$} \right\} \right|. \]
We start by showing that the first summand $ \dim_{\Qp} H^2 \left( \Qp, \Lie \check{\mathrm B} \right) $ is zero.
The short exact sequence of $\Gal_{\Qp}$-representations $ 1 \lra \Lie \check{\mathrm U}  \lra \Lie \check{\mathrm B} \lra \Lie \check{\mathrm T}  \lra 1$ can be dualized and Tate-twisted to give \[ 1  \lra \left( \Lie \check{\mathrm T} \right)^* (1)  \lra \left( \Lie \check{\mathrm B} \right)^* (1) \lra \left( \Lie \check{\mathrm U} \right)^* (1)  \lra 1. \]
By what we have shown above, the Galois-invariants on this sequence yield an isomorphism \[ H^0 \left( \Qp, \left( \Lie \check{\mathrm T} \right)^* (1) \right) \cong H^0 \left( \Qp, \left( \Lie \check{\mathrm B} \right)^* (1) \right). \]
Since the adjoint action on $\Lie \check {\mathrm T}$ is trivial, $\left( \Lie \check {\mathrm T} \right)^* (1)$ is a sum of copies of $\Qp(1)$, and hence it has no invariants.
We conclude that \[ 0 = H^0 \left( \Qp, \left( \Lie \check{\mathrm B} \right)^* (1) \right) \cong H^2 \left( \Qp, \Lie \check{\mathrm B} \right)^{\vee} \] where the last isomorphism is local Tate duality.
In particular, $H^2 \left( \Qp, \Lie \check{\mathrm B} \right)=0$.

Next, notice that the adjoint representation $\Lie \check {\mathrm B}$ has trivial Hodge-Tate weights on the Cartan subalgebra $\Lie \check{\mathrm T}$, and the height filtration on $\Lie \check {\mathrm U}$ shows that the Hodge-Tate weight on each root space $\Lie \check {\mathrm U}_{\alpha}$ is the same as $\langle \alpha, \chi_{\rho_p} \rangle$, where $\chi_{\rho_p} \in X_*(\check {\mathrm T})$ is the cocharacter describing the restriction to the inertia subgroup at $p$ of the quotient representation $\Gal_{\Qp} \lra \check {\mathrm B}(E) \twoheadrightarrow \check {\mathrm T}(E)$.

The local-global compatibility assumed in Conjecture \ref{Galoisrepexists1} (known in some situations, see for instance \cite{HLTT}, \cite{tenauthors}, \cite{BlGGT}) ensures that since $\pi$ appears in the cohomology with constant coefficients, the cocharacter $\chi_{\rho_p}$ is in the strongly negative cone, and thus $\langle \alpha, \chi_{\rho_p} \rangle<0$ for every positive root $\alpha$.

We obtain \[ \begin{gathered} \dim_{\Qp} \left( \Lie \check {\mathrm B} \right) - \left| \left\{ \textnormal{negative Hodge-Tate weights in $\Lie \check {\mathrm B}$} \right\} \right| = \\ \dim_{\Qp} \left( \Lie \check{\mathrm B} \right) - \dim_{\Qp} \left( \Lie \check{\mathrm U} \right) = \dim_{\Qp} \Lie \check {\mathrm T} = \rk \mathrm T \cdot [E : \Qp], \end{gathered} \] which shows that $\dim_E \frac{H^1 \left( \Qp, \Lie \check {\mathrm B} \right)}{H^1_f \left( \Qp, \Lie \check {\mathrm B} \right)} = \rk \mathrm T$ and completes the proof.
\end{proof}
The Teichmuller lift gives a canonical splitting of the integral points of the torus: $\mathrm T(\Zp) \cong \mathrm T(\Zp)_p \times \mathrm T(\Fp)$. Since $\mathrm T(\Fp)$ is a finite group and $E$ is torsion-free, we have that $\Hom \left( \mathrm T(\Fp), E \right) = 0$ and so we immediately get the following result.
\begin{cor} \label{ordivscryscor} The  map  \[ \frac{H^1 \left( \Qp, \Lie \check {\mathrm B} \right)}{H^1_f \left( \Qp, \Lie \check {\mathrm B} \right)} \lra  \Hom \left( \mathrm T(\Zp)_p, E \right) \] obtained by composing the morphism in proposition \ref{ordivscrys} with the restriction map $ \Hom \left( \mathrm T(\Zp), E \right) \lra  \Hom \left( \mathrm T(\Zp)_p, E \right) $ is an isomorphism.
\end{cor}

\begin{proof}

This follows from  Proposition \ref{ordivscrys} and Lemma \ref{useful}.

\end{proof}

\subsection{Map from derived Hecke algebra to Selmer group}

Let now $\Sigma = S \cup \{ p \}$, and consider the following deformation conditions \[ \mathcal L_q = \left\{ \begin{array}{cc} \Lift_{\rho_q} & \textnormal{ if } q \neq p \\ \Lift_{\rho_p}^{\ordi} & \textnormal{ if } q=p \end{array} \right. \textnormal{ so that } \mathfrak t_q = \left\{ \begin{array}{cc} H^1 \left( \Q_q, \Lie \check {\mathrm G} \right) & \textnormal{ if } q \neq p \\ H^1 \left( \Qp, \Lie \check {\mathrm B} \right) & \textnormal{ if } q=p \end{array} \right. \]
We denote by $\mathcal P_{\ordi}$ the corresponding deformation problem as in subsection \ref{globaldeformationsubsection} and by $R^{\ordi}_{\rho}$ the ring representing the associated deformation functor $\Def_{\rho}^{\ordi}$, obtained as in proposition \ref{globaldeformationrepresentable}.
$R^{\ordi}_{\rho}$ is a local, complete, Noetherian $E$-algebra, and we denote its tangent space by $H^1_{\ordi} \left( \Z \left[ \frac{1}{S} \right] , \Lie \check {\mathrm G} \right) = H^1_{\mathcal P_{\ordi}} \left( \Gal_{\Q, \Sigma} , \Lie \check {\mathrm G} \right)$.

Fix a representative for the universal deformation ${ \widetilde \rho_{\ordi}: \Gal_{\Q, \Sigma} \lra \check {\mathrm G} \left( R_{\rho}^{\ordi} \right) }$.
Then as explained in section \ref{localatpdeformationsection} we obtain a group homomorphism \[ \left( \widetilde \rho_{\ordi} \right)^{\ab} |_{\Zp^*} : \Zp^* \lra \check {\mathrm T} \left( R^{\ordi}_{\rho} \right). \]
Pairing with any element of $X^* \left( \check {\mathrm T} \right) = X_* \left( \mathrm T \right)$ yields a homomorphism $\Zp^* \lra \left( R_{\rho}^{\ordi} \right)^*$ hence by duality between $\mathrm T$ and $\check {\mathrm T}$, the map $\left( \widetilde \rho_{\ordi} \right)^{\ab} |_{\Zp^*}$ corresponds to a continuous group homomorphism ${ \mathrm T(\Zp) \lra \left( R^{\ordi}_{\rho} \right)^* }$.
We can restrict to the pro-$p$-Sylow on the domain and then extend $E$-linearly to obtain a morphism in $\mathcal C_E$: \begin{equation} \label{Hidaalgtodeformationring} \psi: \Lambda_E = \Zp \left[ \left[ \mathrm T(\Zp)_p \right] \right] \otimes_{\Zp} E \lra R_{\rho}^{\ordi}. \end{equation}

Still keeping $\Sigma = S \cup \{ p \}$, we fix the following deformation conditions \[ \mathcal L_q = \left\{ \begin{array}{cc} \Lift_{\rho_q} & \textnormal{ if } q \neq p \\ \Lift_{\rho_p}^{\crys} & \textnormal{ if } q=p \end{array} \right. \textnormal{ so that } \mathfrak t_q = \left\{ \begin{array}{cc} H^1 \left( \Q_q, \Lie \check {\mathrm G} \right) & \textnormal{ if } q \neq p \\ H^1_f \left( \Qp, \Lie \check {\mathrm G} \right) & \textnormal{ if } q=p \end{array} \right. \] and we let $\mathcal P_{\crys}$ be the corresponding global deformation problem as in subsection \ref{globaldeformationsubsection}, and $R_{\rho}^{\crys}$ be the local complete Noetherian $E$-algebra representing the associated deformation functor $\Def_{\rho}^{\crys}$ obtained as in proposition \ref{globaldeformationrepresentable}.

We obtain the Selmer group by specializing equation \ref{LESfromcone} to this situation: \[ H^1_f \left( \Z \left[ \frac{1}{S} \right], \Lie \check {\mathrm G} \right) := H^1_{\mathcal P_{\crys}} \left( \Gal_{\Q, \Sigma}, \Lie \check {\mathrm G} \right) = \ker \left( H^1 \left( \Gal_{\Q, \Sigma}, \Lie \check {\mathrm G} \right) \lra \frac{ H^1 \left( \Qp, \Lie \check {\mathrm G} \right)} {H^1_f \left( \Qp, \Lie \check {\mathrm G} \right)} \right), \] and the dual Selmer group by specializing equation \ref{dualSelmerdef} is \begin{equation} \label{dualselmercrystalline} \begin{gathered} H^1_f \left( \Z \left[ \frac{1}{S} \right], \left( \Lie \check {\mathrm G} \right)^* (1) \right) := H^1_{\mathcal P_{\crys}^{\perp}} \left( \Gal_{\Q, \Sigma}, \left( \Lie \check {\mathrm G} \right)^* (1) \right) = \\ = \ker \left( H^1 \left( \Gal_{\Q, \Sigma}, \left( \Lie \check {\mathrm G} \right)^* (1) \right) \lra \bigoplus_{q \in S} H^1 \left( \Q_q, \left( \Lie \check {\mathrm G} \right)^* (1) \right) \oplus \frac{H^1 \left( \Qp, \left( \Lie \check {\mathrm G} \right)^* (1) \right)}{ H^1_f \left( \Qp, \left( \Lie \check {\mathrm G} \right)^* (1) \right) } \right). \end{gathered} \end{equation} corresponding to the local conditions \[ \mathfrak t_q^{\perp} = \left\{ \begin{array}{cc} 0 & \textnormal{ if } q \neq p \\ H^1_f \left( \Qp, \left( \Lie \check {\mathrm G} \right)^* (1) \right) & \textnormal{ if } q=p \end{array} \right. \]
The restriction map at $p$: $H^1_f \left( \Z \left[ \frac{1}{S} \right], \left( \Lie \check{\mathrm G} \right)^* (1) \right) \lra H^1_f \left( \Qp, \left( \Lie \check{\mathrm G} \right)^* (1) \right)$ induces a map between $E$-duals: \[ H^1_f \left( \Qp, \left( \Lie \check{\mathrm G} \right)^* (1) \right)^{\vee} \lra H^1_f \left( \Z \left[ \frac{1}{S} \right], \left( \Lie \check{\mathrm G} \right)^* (1) \right)^{\vee}. \]

Now the embedding $\Lie \check{\mathrm B} \hookrightarrow \Lie \check{\mathrm G}$ gives an injection of cohomology groups \[ \frac{H^1 \left( \Qp, \Lie \check{\mathrm B} \right)}{H^1_f \left( \Qp, \Lie \check{\mathrm B} \right)} \hookrightarrow \frac{H^1 \left( \Qp, \Lie \check{\mathrm G} \right)}{H^1_f \left( \Qp, \Lie \check{\mathrm G} \right)}. \]
Local Tate duality yields that the target of this map is the $E$-dual of the crystalline subgroup for the Tate twist of the dual (coadjoint) representation: \[ \frac{H^1 \left( \Qp, \Lie \check{\mathrm G} \right)} {H^1_f \left( \Qp, \Lie \check{\mathrm G} \right)} \cong H^1_f \left( \Qp, \left( \Lie \check{\mathrm G} \right)^* (1) \right)^{\vee}. \]

\begin{lem}\label{keylemma}
Under the assumptions of Conjecture \ref{Galoisrepexists1}  we get  a map \[ \Phi: \Hom \left( \mathrm T(\Zp)_p, E \right) \lra H^1_f \left( \Z \left[ \frac{1}{S} \right], \left( \Lie \check{\mathrm G} \right)^* (1) \right)^{\vee}\] obtained by composing the isomorphism  \[   \Hom \left( T(\Zp)_p, E \right)   \stackrel{\sim}{\lra} \frac{H^1 \left( \Qp, \Lie \check {\mathrm B} \right)}{H^1_f \left( \Qp, \Lie \check {\mathrm B} \right)} \]  of corollary \ref{ordivscryscor} with

 \[ \frac{H^1 \left( \Qp, \Lie \check{\mathrm B} \right)}{H^1_f \left( \Qp, \Lie \check{\mathrm B} \right)} \hookrightarrow  H^1_f \left( \Qp, \left( \Lie \check{\mathrm G} \right)^* (1) \right)^{\vee} \lra  H^1_f \left( \Z \left[ \frac{1}{S} \right], \left( \Lie \check{\mathrm G} \right)^* (1) \right)^{\vee}. \]

\end{lem}

\begin{proof}
This follows from the above discussion and Corollary \ref{ordivscryscor}.
\end{proof}
The map $\psi$ in equation \ref{Hidaalgtodeformationring} gives rise to a map of tangent spaces in the opposite direction: \[ \psi^{\vee}: \mathfrak t_{R_{\rho}^{\ordi}} = \Hom_* \left(R_{\rho}^{\ordi}, E[\varepsilon] / \varepsilon^2 \right) \lra \Hom_* \left(\Lambda_E, E[\varepsilon] / \varepsilon^2 \right) = \mathfrak t_{\Lambda_E}. \]

The tangent space to the deformation ring is \[ \mathfrak t_{R_{\rho}^{\ordi}} = H^1_{\ordi} \left( \Z \left[ \frac{1}{S} \right], \Lie \check{\mathrm G} \right) = \ker \left( H^1 \left( \Gal_{\Q, \Sigma}, \Lie \check {\mathrm G} \right) \lra \frac{ H^1 \left( \Qp, \Lie \check {\mathrm G} \right)} {H^1 \left( \Qp, \Lie \check {\mathrm B} \right)} \right) \] 
and on the other hand we have \[ \mathfrak t_{\Lambda_E} = \Hom( I/I^2, E) =\Hom_* \left(\Lambda_E, E[\varepsilon] / \varepsilon^2 \right) \cong \Hom \left( \mathrm T(\Zp)_p, E \right) \cong \frac{H^1 \left( \Qp, \Lie \check {\mathrm B} \right)} {H^1_f \left( \Qp, \Lie \check {\mathrm B} \right)} \] where the last isomorphism follows from corollary \ref{ordivscryscor}.

The construction of $\left( \widetilde \rho_p \right)^{\ab}$ and therefore of $\psi$ shows then that under the identifications above, $\psi^{\vee}$ is the local restriction map in Galois cohomology: \[ \psi^{\vee}: \mathfrak t_{R_{\rho}^{\ordi}} = H^1_{\ordi} \left( \Z \left[ \frac{1}{S} \right], \Lie \check{\mathrm G} \right) \lra \frac{H^1 \left( \Qp, \Lie \check {\mathrm B} \right)} {H^1_f \left( \Qp, \Lie \check {\mathrm B} \right)} = \mathfrak t_{\Lambda_E}. \]

\subsection{Dimension formulas}

Let us compute the dimension of $H^1_f \left( \Z \left[ \frac{1}{S} \right], \left( \Lie \check {\mathrm G} \right)^* (1) \right)$. By proposition \ref{deformationringsdimensions}, we have \[ \begin{gathered} \dim_E H^1_f \left( \Z \left[ \frac{1}{S} \right], \Lie \check {\mathrm G} \right) - \dim_E H^1_f \left( \Z \left[ \frac{1}{S} \right], \left( \Lie \check {\mathrm G} \right)^* (1) \right) = \\ = \dim_E H^0 \left( \Q, \Lie \check {\mathrm G} \right) - \dim_E H^0 \left( \Q, \left( \Lie \check {\mathrm G} \right)^* (1) \right) + \sum_{q \in S \cup \{ p, \infty \} } \left( \dim_E \mathfrak t_q^{\mathcal P_f} - \dim_E H^0 \left( \Q_q, \Lie \check {\mathrm G} \right) \right). = \\ = \left( \dim_E H^1_f \left( \Qp, \Lie \check {\mathrm G} \right) - \dim_E H^0 \left( \Qp, \Lie \check {\mathrm G}  \right)  \right) + \left( - \dim_E \Lie \check {\mathrm U} - l_0 \right),  \end{gathered} \] since the only difference with the $\mathcal P_{\ordi}$-condition used to compute the dimension of $\mathcal R_{\rho}^{\ordi}$ is at the prime $p$.

The dimension formula for $p$-adic representations (see for example \cite{bellaiche}, section 2.2) yields that \[ \dim_E H^1_f \left( \Qp, \Lie \check {\mathrm G} \right) - \dim_E H^0 \left( \Qp, \Lie \check {\mathrm G}  \right) = \# \textnormal{negative Hodge-Tate weights of } \Lie \check {\mathrm G}. \]
Recall that our ordinary assumption says that the image of $\rho_p$ is contained on $\check {\mathrm B}(E)$, and thus $\Lie \check {\mathrm U}$ and $\Lie \check {\mathrm B}$ are subrepresentations of $\Lie \check {\mathrm G}$.
The hypothesis of proposition \ref{ordivscrys} - which is a standing assumption since it is necessary to construct the map $\psi$ - says that the Hodge-Tate weights on $\Lie \check {\mathrm U}$ are all negative.

On the other hand, the quotient $\frac{\Lie \check {\mathrm B}}{\Lie \check {\mathrm U}}$ is isomorphic to $\Lie \check {\mathrm T}$, with the adjoint action of $\check {\mathrm B}(E)$ factoring through $\check {\mathrm B}(E) / \check {\mathrm U}(E) \cong \check {\mathrm T}(E)$.
In particular, this adjoint action is trivial and hence the Hodge-Tate weights on $\frac{\Lie \check {\mathrm B}}{\Lie \check {\mathrm U}}$ are $0$.

Finally, since the adjoint action is algebraic, the Hodge-Tate weights on $\frac{\Lie \check {\mathrm G}}{\Lie \check {\mathrm B}} \cong \Lie \check {\mathrm U}^-$ are the opposite of those on $\Lie \check {\mathrm U}$, in particular are all positive.

We conclude that \[ \# \textnormal{negative Hodge-Tate weights of } \Lie \check {\mathrm G} = \dim_E \Lie \check {\mathrm U}, \] and therefore \[ \dim_E H^1_f \left( \Z \left[ \frac{1}{S} \right], \Lie \check {\mathrm G} \right) - \dim_E H^1_f \left( \Z \left[ \frac{1}{S} \right], \left( \Lie \check {\mathrm G} \right)^* (1) \right) = - l_0 \] so that \[ \dim_E H^1_f \left( \Z \left[ \frac{1}{S} \right], \left( \Lie \check {\mathrm G} \right)^* (1) \right) = l_0 + \dim_E H^1_f \left( \Z \left[ \frac{1}{S} \right], \Lie \check {\mathrm G} \right). \]

We also record that \[ \begin{gathered} \dim_E \coker \left( \psi^{\vee} \right) = \dim_E \mathfrak t_{\Lambda_E} - \dim_E \im \left( \psi^{\vee} \right) = \\ =  \dim_E \mathfrak t_{\Lambda_E} - \left( \dim_E H^1_{\ordi} \left( \Z \left[ \frac{1}{S} \right], \Lie \check{\mathrm G} \right) - \dim_E \ker \left( \psi^{\vee} \right) \right) \end{gathered} \]

\subsection{Poitou-Tate duality and  Selmer groups}

The following key proposition follows from Proposition 3.10 of \cite{HT} specialized to the ordinary case.

\begin{prop}\label{PT}
 We have an exact sequence
 \[ 0 \lra  H^1_{\ordi^\perp} \left( \Z \left[ \frac{1}{S} \right], \left( \Lie \check{\mathrm G} \right)^* (1) \right)  
 \lra  H^1_f \left( \Z \left[ \frac{1}{S} \right], \left( \Lie \check{\mathrm G} \right)^* (1) \right) 
 \lra \left( \frac{H^1 \left( \Qp, \Lie \check {\mathrm B} \right)} {H^1_f \left( \Qp, \Lie \check {\mathrm B} \right) } \right)^\vee \]
 \[ \lra H^1_{\ordi} \left( \Z \left[ \frac{1}{S} \right], \Lie \check{\mathrm G} \right)^\vee  \lra    H^1_f \left( \Z \left[ \frac{1}{S} \right], \Lie \check {\mathrm G} \right)^\vee  \lra 0 \]
\end{prop}

\begin{proof}
  This is an application of  the Poitou-Tate exact sequence as justified in  Proposition 3.10 of \cite{HT}.
\end{proof}

\subsection{A duality pairing}

\begin{thm} \label{cokerneldualSelmer}
Global duality in Galois cohomology gives a natural pairing \[ H^1_f \left( \Z \left[ \frac{1}{S} \right], \left( \Lie \check {\mathrm G} \right)^* (1) \right) \times \coker \left( \psi^{\vee} \right) \lra E \]

This induces a perfect pairing   \[ \im (\psi^{PT}) \times \coker \left( \psi^{\vee} \right)  \lra E \] where $\psi^{PT}$ is the map 

\[H^1_f \left( \Z \left[ \frac{1}{S} \right], \left( \Lie \check{\mathrm G} \right)^* (1) \right)  \lra \left( \frac{H^1 \left( \Qp, \Lie \check {\mathrm B} \right)} {H^1_f \left( \Qp, \Lie \check {\mathrm B} \right) } \right)^\vee \] from   Proposition \ref{PT}.

Moreover, suppose that $R_{\rho}^{\crys} \cong E$. Then $\psi^{\vee}$ is injective, and   if further $R^{\ordi}_\rho$ is smooth of dimension $r - l_0$, the pairing induces  an isomorphism \[ \coker \left( \psi^{\vee} \right) \cong H^1_f \left( \Z \left[ \frac{1}{S} \right], \left( \Lie \check {\mathrm G} \right)^* (1) \right)^{\vee}. \]
\end{thm}

\begin{proof} We will describe a perfect pairing \[ \langle \cdot , \cdot \rangle: H^1_f \left( \Z \left[ \frac{1}{S} \right], \left( \Lie \check {\mathrm G} \right)^* (1) \right) \times \coker \left( \psi^{\vee} \right) \lra E. \]
Indeed, an element of $\coker \left( \psi^{\vee} \right)$ is a class $\beta_p \in \frac{H^1 \left( \Qp, \Lie \check {\mathrm B} \right)} {H^1_f \left( \Qp, \Lie \check {\mathrm B} \right)}$ modulo the image of $\psi^{\vee}$, which are the global classes.
We take $\alpha \in H^1_f \left( \Z \left[ \frac{1}{S} \right], \left( \Lie \check {\mathrm G} \right)^* (1) \right)$ and we pair \[ \langle \alpha, [\beta_p] \rangle = \left( \alpha_p, \beta_p \right)_p \] where $\left( \cdot , \cdot \right)_p$ is the local Tate duality pairing at $p$.

Global duality ensures that the pairing does not change if we replace $\beta_p$ by $\beta_p + \psi^{\vee} \left( \widetilde \beta \right)$ for $\widetilde \beta \in H^1_{\ordi} \left( \Z \left[ \frac{1}{S} \right], \Lie \check {\mathrm G} \right)$.

For the pairing to be well-defined, we also need to show $\alpha_p \in \mathrm{Ann} \left( H^1_f \left( \Qp, \Lie \check {\mathrm B} \right) \right)$, the annihilator of $H^1_f \left( \Qp, \Lie \check {\mathrm B} \right)$ under local Tate duality.
By equation \ref{dualselmercrystalline}, we have that $\alpha_p \in H^1_f \left( \Qp, \left( \Lie \check {\mathrm G} \right)^* (1) \right)$, and on the other hand the injection $\Lie \check {\mathrm B} \hookrightarrow \Lie \check {\mathrm G}$ yields an injection in cohomology $H^1_f \left( \Qp, \Lie \check {\mathrm B} \right) \hookrightarrow H^1_f \left( \Qp, \Lie \check {\mathrm G} \right)$, so that $\alpha_p$ is indeed in the annihilator of $H^1_f \left( \Qp, \Lie \check {\mathrm B} \right)$.

To prove the first part of the theorem, it remains to check that $\coker \left( \psi^{\vee} \right)$ and $H^1_f \left( \Z \left[ \frac{1}{S} \right], \left( \Lie \check {\mathrm G} \right)^* (1) \right)$ have the same dimension. This follows from the Poitou-Tate exact sequence in Proposition \ref{PT}.

Let us now assume that $R_{\rho}^{\crys} \cong E$ - the tangent space $H^1_f \left( \Z \left[ \frac{1}{S} \right], \Lie \check {\mathrm G} \right)$ is therefore zero and so $\ker \left( \psi^{\vee} \right) = 0$.

If we assume further that $R^{\ordi}_{\rho_\pi}$ is smooth of dimension $r-l_0$, then by Proposition \ref{PT} we get that
\[ H^1_{\ordi^\perp} \left( \Z \left[ \frac{1}{S} \right], \left( \Lie \check{\mathrm G} \right)^* (1) \right)=0,\] the map  $\psi^{PT}$ between  \[  H^1_f \left( \Z \left[ \frac{1}{S} \right], \left( \Lie \check{\mathrm G} \right)^* (1) \right)  \lra \left( \frac{H^1 \left( \Qp, \Lie \check {\mathrm B} \right)} {H^1_f \left( \Qp, \Lie \check {\mathrm B} \right) } \right)^\vee \] is injective and we deduce that  the  natural pairing \[ \coker \left( \psi^{\vee} \right) \times H^1_f \left( \Z \left[ \frac{1}{S} \right], \left( \Lie \check {\mathrm G} \right)^* (1) \right) \lra E \] is a perfect pairing between vector spaces of dimension $l_0$.

\end{proof}

\section{Derived Hecke action and dual Selmer groups} \label{dualSelmersection}

We prove in this section Theorem \ref{dimensionality}  that shows  the action of the degree $1$ derived Hecke operators $H^1 \left( \mathrm T(\Zp)_p, E \right)$ on the arithmetic cohomology $H^* \left( Y_{1,1}, E \right)_{\mathfrak m}$ factors through the dual of  a  dual Selmer group assuming the dimension Conjecture  \ref{dimensionconjectureforring} below.  We then  prove Theorem \ref{main}  that shows under  Conjecture \ref{smoothnessconj}  that the derived Hecke action at $p$ generates 
$H^* \left( Y_{1,1}, E \right)_{\mathfrak m}$  over the bottom degree.

For the proof of both theorems the key ingredients are  our work in  \S  \ref{generationofcohomologysection}, the map constructed in Lemma \ref{keylemma}  from $H^1 \left( \mathrm T(\Zp)_p, E \right)$  to the dual of the dual  dual Selmer group, and the duality result formulated in  Proposition \ref{PT}.

\subsection{ An isomorphism of deformation rings and Hecke rings}

We assume the following hypothesis  which is known under some assumptions (see \cite{KT} and \cite{CGH}).

\begin{hypothesis}\label{R=T}
 We have an isomorphism of $\Lambda_E$-algebras $R_\rho^\ordi = \mathbb T^S_\ordi (K(1,1),E)_{\mathfrak m}$.
\end{hypothesis}

\begin{lem}\label{surjectivity}
 Assume Hypothesis \ref{R=T}. Then 
 the map  $\mathfrak m_{\Lambda_E}/\mathfrak m_{\Lambda_E}^2 \lra \mathfrak m_{R_\rho^\ordi}/\mathfrak m_{R_\rho^\ordi}^2$ is surjective,  $R_\rho^\crys=E$, and
  $H^1_f \left( \Z \left[ \frac{1}{S} \right], \Lie \check {\mathrm G} \right)=0$.
\end{lem}

\begin{proof}
  We use  control of ordinary cohomology  in  Proposition  \ref{complexordinarytower} and multiplicity one theorems, as  in the   proof of Theorem \ref{complexes},  which gives  that $T^S_\ordi (K(1,1),E)_{\mathfrak m}/\mathfrak m_{\Lambda_E}=E$.  By Nakayama's lemma  this implies 
  that the map $\Lambda_E \lra \mathbb T^S_\ordi (K(1,1),E)_{\mathfrak m}$ is surjective, and thus by  Hypothesis \ref{R=T} we deduce 
 the map  $\mathfrak m_{\Lambda_E}/\mathfrak m_{\Lambda_E}^2 \lra \mathfrak m_{R_\rho^\ordi}/\mathfrak m_{R_\rho^\ordi}^2$ is surjective,  $R_\rho^\crys=E$, and
  $H^1_f \left( \Z \left[ \frac{1}{S} \right], \Lie \check {\mathrm G} \right)=0$.
\end{proof}

We also record the following conjecture, which is closely related to conjecture \ref{dimensionconj} but is phrased in terms of a Galois deformation ring, rather than the complex interpolating ordinary cohomology of arithmetic groups.
\begin{conj}[Dimension conjecture for ordinary deformation ring] \label{dimensionconjectureforring}
The ring $R_{\rho}^{\ordi}$ has dimension $\dim \Lambda_E - l_0 = r- l_0$.
\end{conj}

\subsection{Dimension conjecture implies that derived Hecke action factors through dual of dual Selmer}

We now study the following conjecture. Recall our standing assumption of genericity of $\pi$ at $p$, which gives that $\rho_\pi$ 
of Conjecture \ref{Galoisrepexists1} satisfies our conditions $(REG), (REG^*)$ at $p$.

\begin{conj} \label{derivedactionthruSelmerconj} Assume that $\rho_{\pi}$ is as in Conjecture \ref{Galoisrepexists1}.
Then the derived diamond action of $\Hom \left( \mathrm T(\Zp)_p, E \right)$ on $H^* \left( Y_{1,1}, E \right)_{\mathfrak m}$ factors through the map $\Phi$ of lemma \ref{keylemma} to an action of $H^1_f \left( \Z \left[ \frac{1}{S} \right], \left( \Lie \check{\mathrm G} \right)^* (1) \right)^{\vee}$.
\end{conj}

\begin{thm}\label{dimensionality}
The Conjecture  \ref{derivedactionthruSelmerconj} is true if we assume  Hypothesis \ref{R=T} and  Conjecture \ref{dimensionconjectureforring}.
\end{thm}

\begin{proof}

The proof uses  our work in \S \ref{generationofcohomologysection} and arguments of Hansen-Thorne \cite{HT}. 
The main ingredient is Poitou-Tate duality (see Proposition \ref{PT} of this paper) as  used in the proof of  Theorem 4.13 of \cite{HT}.

We have the following    exact sequence:
$$ 0  \lra I \cap \mathfrak m_{\Lambda_E}^2/\mathfrak m_{\Lambda_E}I \lra I/\mathfrak m_{\Lambda_E}I \lra \mathfrak m_{\Lambda_E}/\mathfrak m_{\Lambda_E}^2 \lra \mathfrak m_{R_\rho^\ordi}/\mathfrak m_{R_\rho^\ordi}^2 \lra 0.$$ The exactness on the right comes from Lemma \ref{surjectivity}.  Using the notation of the earlier section we write $R_\rho^\ordi=\Lambda_E/I$ where $I$ is generated by a regular sequence of length $l_0$.

The argument in the proof of  Theorem 4.13 of \cite{HT}  identifies   the map
$I/\mathfrak m_{\Lambda_E}I \lra \mathfrak m_{\Lambda_E}/\mathfrak m_{\Lambda_E}^2$ with the map \[H^1_f \left( \Z \left[ \frac{1}{S} \right], \left( \Lie \check{\mathrm G} \right)^* (1) \right)  \lra  \left( \frac{H^1 \left( \Qp, \Lie \check {\mathrm B} \right)} {H^1_f \left( \Qp, \Lie \check {\mathrm B} \right) } \right)^\vee \]  which is dual to  the map $\Phi$ given by composition of the maps 
 \[   \Hom \left( T(\Zp)_p, E \right)   \stackrel{\sim}{\lra} \frac{H^1 \left( \Qp, \Lie \check {\mathrm B} \right)}{H^1_f \left( \Qp, \Lie \check {\mathrm B} \right)}  \hookrightarrow  H^1_f \left( \Qp, \left( \Lie \check{\mathrm G} \right)^* (1) \right)^{\vee} \lra  H^1_f \left( \Z \left[ \frac{1}{S} \right], \left( \Lie \check{\mathrm G} \right)^* (1) \right)^{\vee} \] considered in Lemma \ref{keylemma}. Also as noted earlier \[ \Hom \left( T(\Zp)_p, E \right)   \stackrel{\sim}{\lra} \frac{H^1 \left( \Qp, \Lie \check {\mathrm B} \right)}{H^1_f \left( \Qp, \Lie \check {\mathrm B} \right)}  \stackrel{\sim}{\lra} \Hom( \mathfrak m_{\Lambda_E}/\mathfrak m_{\Lambda_E}^2, E).\]
This proves our theorem when combined with the arguments in   Lemma \ref{naf} which show that  the action of  $\bigwedge^i \Hom \left( \mathrm T(\Zp)_p, E \right)$ on $H^* \left( Y_{1,1}, E \right)_{\mathfrak m}$, using the isomorphisms above, is the  same as the action of $\bigwedge^* \Hom (\mathfrak m_{\Lambda_E}/\mathfrak m_{\Lambda_E}^2, E)$ on $\bigwedge^* \Hom (I/\mathfrak m_{\Lambda_E}I, E)$ induced by the  composition of the map \[ \bigwedge^* \Hom (\mathfrak m_{\Lambda_E}/\mathfrak m_{\Lambda_E}^2, E) \lra \bigwedge^* \Hom (I/\mathfrak m_{\Lambda_E}I, E) \] (in turn induced by the map  $I/\mathfrak m_{\Lambda_E} I \lra \mathfrak m_{\Lambda_E} /\mathfrak m_{\Lambda_E}^2$) and the maps \[ \bigwedge^i \Hom (I/\mathfrak m_{\Lambda_E}I, E) \times \bigwedge^j \Hom (I/\mathfrak m_{\Lambda_E}I, E) \lra \bigwedge^{i+j} \Hom (I/\mathfrak m_{\Lambda_E}I, E). \]

\end{proof}

\subsection{Smoothness of ordinary deformation rings $R_\pi^\ordi$ and bigness of derived Hecke action}

We state the expected smoothness conjecture for the deformation rings we consider which is a higher analog of the classical Leopoldt conjecture.

\begin{conj}[Smoothness conjecture] \label{smoothnessconj}  Suppose that the cohomology group $H^2_{\mathcal P_{\ordi}} \left( \Gal_{\Q, \Sigma}, \Lie \check {\mathrm G} \right)$ is $0$. Thus  (as we justify below)  $R_{\rho}^{\ordi}$ is smooth of dimension $r-l_0$.
\end{conj}

By proposition \ref{deformationringsdimensions}, $R_{\rho}^{\ordi}$ is thus a power series ring over $E$ in $\dim_E H^1_{\ordi} \left( \Z \left[ \frac{1}{S} \right] , \Lie \check {\mathrm G} \right)$ variables.
Still by proposition \ref{deformationringsdimensions}, we have \[ \begin{gathered} \dim_E H^1_{\ordi} \left( \Z \left[ \frac{1}{S} \right] , \Lie \check {\mathrm G} \right) = \dim_E H^0 \left( \Q, \Lie \check {\mathrm G} \right) - \dim_E H^0 \left( \Q, \left( \Lie \check {\mathrm G} \right)^* (1) \right) + \\ + \sum_{q \in \Sigma} \left( \dim_E \mathfrak t_q^{\mathcal P_{\ordi}} - \dim_E H^0 \left( \Q_q, \Lie \check {\mathrm G} \right) \right). \end{gathered} \]
We dualize and Tate-twist the short exact sequence $0 \lra \Lie \check{ \mathrm B} \lra \Lie \check {\mathrm G} \lra \frac{\Lie \check {\mathrm G}}{\Lie \check {\mathrm B}} \lra 0$, and then taking $\Gal_{\Qp}$-cohomology yields the long exact sequence \[ 0 \lra H^0 \left( \Qp, \left( \frac{\Lie \check {\mathrm G}}{\Lie \check {\mathrm B}} \right)^* (1) \right) \lra H^0 \left( \Qp, \left( \Lie \check {\mathrm G} \right)^* (1) \right) \lra H^0 \left( \Qp, \left( \Lie \check {\mathrm B} \right)^* (1) \right) \lra \ldots. \]
In the proof of proposition \ref{ordivscrys}, we showed that $H^0 \left( \Qp, \left( \Lie \check {\mathrm B} \right)^* (1) \right)=0$.
On the other hand, following \cite{Patrikis}, proposition 4.4, we have that the Killing form gives a Galois-equivariant isomorphism $ \left( \frac{\Lie \check {\mathrm G}}{\Lie \check {\mathrm B}} \right)^* \cong \Lie \check{\mathrm U}$.
Because of our assumption on $\rho_{\pi}|_p$ (see Conjecture \ref{Galoisrepexists1} (3) and Lemma \ref{useful})  we have  $H^0 \left( \Qp, \Lie \check {\mathrm U} (1) \right)=0$, and thus  \[ H^0 \left( \Q, \left( \Lie \check {\mathrm G} \right)^* (1) \right) = 0. \]
The assumption from conjecture \ref{Galoisrepexists1} that $\rho_{\pi}$ is absolutely irreducible allows us to apply lemma A.2 of \cite{FKP} to conclude that \[ H^0 \left( \Q, \Lie \check {\mathrm G} \right) = 0. \]
Notice that lemma A.2 in loc. cit. is dealing with representations in positive characteristic and thus relies on more subtle results than what we need for our $p$-adic representation $\rho$. Either way, the same proof as in loc. cit. holds verbatim (our standing assumption that $p>|W|$ implies in particular that $p \nmid (n+1)$ if $\mathrm G$ contains a factor of type $A_n$).

We obtain then \[ \dim_E H^1_{\ordi} \left( \Z \left[ \frac{1}{S} \right] , \Lie \check {\mathrm G} \right) = \sum_{q \in \Sigma} \left( \dim_E \mathfrak t_q^{\mathcal P_{\ordi}} - \dim_E H^0 \left( \Q_q, \Lie \check {\mathrm G} \right) \right). \]
Consider first $q = p \in \Sigma$: by proposition \ref{dimensionlocalordinaryring}, we have that \[ \dim_E \mathfrak t_p^{\mathcal P_{\ordi}} - \dim_E H^0 \left( \Qp, \Lie \check {\mathrm G} \right) = \dim_E \left( \Lie \check{\mathrm B} \right). \]
Next, let $q \in S - \{ p, \infty \}$ be a bad prime. By smoothness of the deformation ring $ R_{\rho}^{\ordi}$ we have that $H^2 \left( \Q_q, \Lie \check{\mathrm G} \right) =0$, and then the local Euler characteristic formula yields that \[ \dim_E \mathfrak t_q^{\mathcal P_{\ordi}} - \dim_E H^0 \left( \Q_q, \Lie \check {\mathrm G} \right) = \dim_E H^1 \left( \Q_q, \Lie \check {\mathrm G} \right) - \dim_E H^0 \left( \Q_q, \Lie \check {\mathrm G} \right) = 0. \]
Finally, we recall that $\rho$ is $\check {\mathrm G}$-odd as in definition \ref{oddgaloisrep} and thus \[ \dim_E \mathfrak t_{\infty}^{\mathcal P_{\ordi}} - \dim_E H^0 \left( \R, \Lie \check {\mathrm G} \right) = 0 - \left( \dim_E \Lie \check {\mathrm U} + l_0 \right) = - \dim_E \Lie \check {\mathrm U} - l_0. \]
Putting all this together, we obtain that $R_{\rho}^{\ordi}$ is a power series ring over $E$ in \[ \dim_E H^1_{\ordi} \left( \Z \left[ \frac{1}{S} \right], \Lie \check {\mathrm G} \right) = \dim_E \Lie \check{\mathrm B} - \dim_E \Lie \check {\mathrm U} - l_0 = \dim_E \Lie \check {\mathrm T} - l_0 = r - l_0 \] variables.

If we assume the Conjecture \ref{smoothnessconj} then we can verify Hypothesis \ref{R=T}. 

\begin{lem}\label{obv}
 Under the smoothness conjecture we get  an isomorphism of $\Lambda_E$-algebras $R_\pi^\ordi = \mathbb T^S_\ordi (K(1,1),E)_{\mathfrak m}=E$.
\end{lem}

\begin{proof}

This follows using the concentration of the cohomology groups $H^* (Y(K_{1,1}),E)_{\mathfrak m}$ in a range of degrees of size $l_0$ which implies by the Calegari-Geraghty lemma that $\dim \mathbb T^S_\ordi (K(1,1),E)_{\mathfrak m} \geq r -l_0$. (A more general result is proved in Hansen's thesis.)
\end{proof}

%The arguments in \S \ref{generationofcohomologysection} show  that if we assume the smoothness conjecture for the ring $R_\rho^\ordi$ then we can go further than Theorem \ref{dimensionality}  and show that the action of the derived Hecke algebra at $p$, which under the assumption of the dimension conjecture \ref{dimensionconjectureforring} factors by loc. cit.  through the  map $\Phi$ to dual of the dual Selmer, is big enough to provide the extra symmetries that ``explain"  the dimensions of $H^*(\Gamma,\Q)_\pi$.  

We prove the following  relation between the bigness of the derived Hecke action and smoothness of $R_\rho^\ordi$.

\begin{thm}\label{main} $ $

\begin{enumerate}

\item Assume Hypothesis \ref{R=T} and Conjecture \ref{dimensionconjectureforring}.  If  the derived Hecke action of $\Hom (T(\Zp)_p, E)$  generates the cohomology $H^*( Y, E)_{\mathfrak m}$ over the bottom degree, and  $R_\rho^\ordi$ is of the expected dimension $r-l_0$,  then  $R_\rho^\ordi$ is  smooth of the expected dimension $r-l_0$.

\item  Assume Conjecture \ref{smoothnessconj}. Then  the derived Hecke action generates the cohomology $H^*( Y ,E)_{\mathfrak m}$ over the bottom degree.

\end{enumerate}

\end{thm}

\begin{proof}  The proof is close to the proof of Corollary \ref{cohomologygeneratedbybottomdegree}.

 We first prove (1). Thus assume Hypothesis \ref{R=T} and Conjecture \ref{dimensionconjectureforring}. Then we know by Lemma \ref{surjectivity} and  Lemma \ref{naf}  that the map  $\Ext_{\Lambda_E}^0 \left( R_{\rho}^{\ordi}, E \right) \times \Ext^i_{\Lambda_E} (E, E) \lra \Ext_{\Lambda}^i \left( R_{\rho}^{\ordi}, E \right)$ is   surjective for all $i \geq 0$ (equivalently for $i=1$)  if and only if the map $I/\mathfrak m_{\Lambda_E}I \lra \mathfrak m_{\Lambda_E}/\mathfrak m_{\Lambda_E}^2$ considered earlier  is injective.  As in proof of Theorem \ref{dimensionality}, by the Poitou-Tate sequence in Proposition \ref{PT} this   is equivalent to the vanishing of  $H^2_{\mathcal P_{\ordi}} \left( \Gal_{\Q, \Sigma}, \Lie \check {\mathrm G} \right)$ which in turn  is equivalent to the ordinary deformation ring $R_\rho^\ordi$  being smooth of dimension $r-\ell_0$.   Furthermore the surjectivity of the map $ \Ext_{\Lambda_E}^0 \left( R_{\rho}^{\ordi}, E \right) \times \Ext^i_{\Lambda_E} (E, E) \lra \Ext_{\Lambda}^i \left( R_{\rho}^{\ordi}, E \right)$  for all $i \geq 0$  is equivalent to   the  derived Hecke action of $\Hom (T(\Zp)_p, E)$  generating the cohomology $H^*( Y, E)_{\mathfrak m}$ over the bottom degree.  This finishes the proof of (1).

   To prove part (2) we  observe
that under  Conjecture \ref{smoothnessconj} we know Hypothesis \ref{R=T} using Lemma \ref{obv}.

\end{proof}

\begin{rem}
One would like to ideally prove   that  if  the derived Hecke action of $\Hom (T(\Zp)_p,E)$  generates the cohomology $H^*(\Gamma, E)_\pi$ over the bottom degree,   then  $R_\rho^\ordi$ is  smooth of the expected dimension $r-l_0$ without additionally  assuming the dimension conjecture \ref{dimensionconjectureforring}.
We would know this if we knew {\em a priori}  (without assuming the dimension conjecture \ref{dimensionconjectureforring})  that the derived diamond action of $H^1 \left( \mathrm T(\Zp)_p , E \right) = \Hom \left( \mathrm T(\Zp)_p, E \right)$ on $H^* \left( Y_{1,1}, E \right)_{\mathfrak m}$ factors through the map $\Phi$ to an action of $H^1_f \left( \Z \left[ \frac{1}{S} \right], \left( \Lie \check{\mathrm G} \right)^* (1) \right)^{\vee}$. We will return to this in the forthcoming work \cite{AKR}.

\end{rem}

\subsection{Comparison to \cite{akshay}}

We sketch an alternative proof of Theorem \ref{dimensionality} (2) along the lines of the arguments in \cite{akshay} of the following proposition (which is slightly weaker than the results above) which we believe gives a useful comparison of the various arguments used here and in \cite{akshay}.

\begin{prop}

Assume that $R_\rho^\ordi$ is  smooth of   dimension $r-l_0$.  Then  the derived Hecke action $H^1 \left( \mathrm T(\Zp)_p , E \right) = \Hom \left( \mathrm T(\Zp)_p, E \right)$ on $H^* \left( Y_{1,1}, E \right)_{\mathfrak m}$ factors through the map $\Phi$ to an action of $H^1_f \left( \Z \left[ \frac{1}{S} \right], \left( \Lie \check{\mathrm G} \right)^* (1) \right)^{\vee}$, and  generates the cohomology $H^*( Y_{1,1} , E )_{\mathfrak m}$ over the bottom degree.

\end{prop}

\begin{proof}
  We have already proved  in  \S \ref{generationofcohomologysection} the last part of the proposition and need only prove that the derived Hecke action factors through the map $\Phi$ of Lemma \ref{keylemma}.
  
Theorem \ref{cokerneldualSelmer} gives us a surjection \[ \Hom \left( \mathrm T(\Zp)_p, E \right) \twoheadrightarrow H^1_f \left( \Z \left[ \frac{1}{S} \right], \left( \Lie \check {\mathrm G} \right)^* (1) \right)^{\vee} \] with kernel $H^1_{\ordi} \left( \Z \left[ \frac{1}{S} \right], \Lie \check {\mathrm G} \right)$, so we have to prove that the image under $\psi^{\vee}$ of the tangent space $\mathfrak t_{R^{\ordi}_{\rho}}$ acts trivially on $H^* \left( Y_{1,1}, E \right)_{\mathfrak m}$.

Recall from section \ref{generationofcohomologysection} that the cohomology $H^* \left( Y_{1,1}, E \right)_{\mathfrak m}$ is computed by $H^* \left( \Hom_{\Lambda_E} \left( F_{ \mathfrak m}^{\bullet}, E \right) \right)$.
On the other hand, we have shown that \[ F_{ \mathfrak m}^{\bullet}  \sim \Lambda_E / I \] where $I$ is generated by a system of parameters of length equal to the defect $l_0$.

An application of the Calegari-Geraghty lemma \ref{CGlemma} as in section \ref{generationofcohomologysection} gives then that $H^* \left( F^{\bullet}_{\mathfrak m} \right)$ is a finite free $R_{\rho}^{\ordi}$-module - the rank $m$ is expected to be $1$ if classical multiplicity one results hold, but the rest of the proof holds independently of $m$.
In particular, \[ H^* \left( Y_{1,1},  E \right)_{\mathfrak m} \cong H^* \left( \Hom_{\Lambda_E} \left( F_{ \mathfrak m}^{\bullet} , E \right) \right) \cong H^* \left( \Hom_{\Lambda_E} \left( \left( R_{\rho}^{\ordi} \right)^{\oplus m} , E \right) \right), \] and the derived diamond actions of \[ \Hom \left( \mathrm T(\Zp)_p, E \right) \cong H^* \left( \Hom_{\Lambda_E} \left( E, E \right) \right) \] is the obvious one, by post compositions of homomorphisms:
\begin{equation} \label{derivedactiononExt} H^i \left( \Hom_{\Lambda_E} \left( \left( R_{\rho}^{\ordi} \right)^{\oplus m} , E \right) \right) \times H^j \left( \Hom_{\Lambda_E} \left( E, E \right) \right) \lra H^{i+j} \left( \Hom_{\Lambda_E} \left( \left( R_{\rho}^{\ordi} \right)^{\oplus m} , E \right) \right). \end{equation}

Since $R_{\rho}^{\ordi}$ is a power series ring over $E$ of dimension $r - l_0$, we can compatibily choose coordinates on $\Lambda_E$ and $R_{\rho}^{\ordi}$ so that the surjection $\psi$ kills the last $l_0$ coordinates (see lemma 7.4 of \cite{akshay}): \[ \psi: \Lambda_E \cong \Zp \left[ \left[ X_1, \ldots, X_r \right] \right] \otimes_{\Zp} E \lra \Zp \left[ \left[ Y_1, \ldots, Y_{r-l_0} \right] \right] \otimes_{\Zp} E \cong R_{\rho}^{\ordi} \qquad X_i \mapsto \left\{ \begin{array}{cc} Y_i & \textnormal{ if } i \le r - l_0 \\ 0 & \textnormal{ if } i > r -l_0 \end{array} \right. \]
The action in formula \ref{derivedactiononExt} can then be computed explicitly using the presentation of the $\Lambda_E$-module $R_{\rho}^{\ordi}$ afforded by $\psi$: the Koszul complex argument of appendix B of \cite{akshay} show that the image of $\mathfrak t_{R_{\rho}^{\ordi}}$ inside $\mathfrak t_{\Lambda_E} \cong H^* \left( \Hom_{\Lambda_E} \left( E, E \right) \right)$ acts trivially on $H^i \left( \Hom_{\Lambda_E} \left( R_{\rho}^{\ordi} , E \right) \right)$, which is precisely the statement of the theorem.
\end{proof}

\appendix

\section{Comparison with the work of Hansen-Thorne}\label{HaTho}

Hansen and Thorne construct in \cite{HT} an action of a $\mathrm{Tor}$-group which decreases the cohomological degree, and for them to construct this action it suffices that $F_{ \mathfrak m}^{\bullet}$ is quasi-isomorphic to the quotient of $\Lambda_E$ by a regular sequence (in this subsection $\mathfrak m$ is a maximal ideal of a certain Hecke algebra corresponding to $\pi$ as in section 4 of \cite{HT}).

The actions constructed in this paper and \cite{HT} are dual in the following sense, and their contrasting properties may be summarized qualitatively as follows:

1.  In \cite{HT}, a `` Tor action'' is constructed   assuming a certain  dimension conjecture and : (i) assuming the dimension conjecture, this gives   a big ``surjective''  action of  the exterior algebra of a certain vector space $V_0$ over $E$ of dimension $l_0$ on $H^*(\Gamma,\Q_P)_\pi$, and (ii)  $V_0$ can be identified with a  dual Selmer group  assuming smoothness
statements like the one above.

2.  In the present paper, the derived Hecke ``Ext action'' is constructed unconditionally, and  : (i) it  factors through dual of  dual Selmer assuming dimension conjectures (and in forthcoming work we show this without such an assumption), and (ii) it gives a  big ``surjective'' action assuming smoothness.

In fact, we expect the two actions to be compatible in the sense of Galatius and Venkatesh \cite{GV}, section 15. To make such a comparison  of the two action we assume that  $R^{\ordi}_{\rho}$ is smooth. Hansen and Thorne construct a degree-lowering action of $H^1_f \left( \Z \left[ \frac{1}{S} \right], \Lie \check {\mathrm G} \right)$ on $H^* \left( Y(K), E \right)_{\mathfrak m}$, while  we construct a degree-raising action of the dual space $H^1_f \left( \Z \left[ \frac{1}{S} \right], \Lie \check {\mathrm G} \right)^{\vee}$ on the same cohomology group $H^* \left( Y(K), E \right)_{\mathfrak m}$.
\begin{conj} Our action is compatible with that of Hansen and Thorne in the sense of \cite{GV}: for every $v \in H^1_f \left( \Z \left[ \frac{1}{S} \right], \Lie \check {\mathrm G} \right)^{\vee}$, $v^* \in H^1_f \left( \Z \left[ \frac{1}{S} \right], \Lie \check {\mathrm G} \right)$ and $f \in H^* \left( Y(K), E \right)_{\mathfrak m}$ we have \[ v^*. \left( v.f \right) + v. \left( v^*.f \right) = \langle v^*, v \rangle \cdot f. \]
\end{conj}
This can be proved along the lines of the proof of Theorem 15.2 in \cite{GV}.
As \cite{GV} mentions, when $H^* \left( Y(K), E \right)_{\mathfrak m}$ is generated by $\bigwedge^* H^1_f \left( \Z \left[ \frac{1}{S} \right], \Lie \check {\mathrm G} \right)^{\vee}$ over $H^{q_0} \left( Y(K), E \right)_{\mathfrak m}$ (as in theorem \ref{cohomologygeneratedbybottomdegree}), a formal consequence of this compatibility is that the action of $H^1_f \left( \Z \left[ \frac{1}{S} \right], \Lie \check {\mathrm G} \right)$ defined by \cite{HT} is uniquely determined.
In other words, the derived Hecke action and the action defined in \cite{HT}  determine uniquely one another.

\section{A rationality conjecture}

We make the analog of the rationality conjectures of \cite{akshay} for the action of the dual of the dual Selmer group on cohomology of arithmetic groups that we have defined in the previous sections.  
We have  an isomorphism \[ \frac{H^1 \left( \Qp, \Lie \check {\mathrm B} \right)}{H^1_f \left( \Qp, \Lie \check {\mathrm B} \right)} \lra \Hom \left( \Zp^*, \Lie \check{\mathrm T} \right). \]
Now the embedding $ \Lie \check {\mathrm B} \hookrightarrow  \Lie \check {\mathrm G}$ gives an injection of cohomology groups \[ \frac{H^1 \left( \Qp, \Lie \check {\mathrm B} \right)}{H^1_f \left( \Qp, \Lie \check {\mathrm B} \right)} \hookrightarrow \frac{H^1 \left( \Qp, \Lie \check {\mathrm G} \right)}{H^1_f \left( \Qp, \Lie \check {\mathrm G} \right)}. \] 

Tate duality yields that $\frac{H^1 \left( \Qp, \Lie \check {\mathrm G} \right)}{H^1_f \left( \Qp, \Lie \check {\mathrm G} \right)} \cong H^1_f \left( \Qp,  \left( \Lie \check {\mathrm G} \right)^* (1) \right)$.

The global dual Selmer group was defined in equation \ref{dualselmercrystalline}, and we have a restriction map \[ \res{\Gal_{\Q}}{\Gal_{\Qp}} H^1_f \left( \Z \left[ \frac{1}{S} \right],  \left( \Lie \check {\mathrm G} \right)^* (1) \right) \lra H^1_f \left( \Qp,  \left( \Lie \check {\mathrm G} \right)^* (1) \right) \] inducing a map between $\Qp$-duals: \[ H^1_f \left( \Qp,  \left( \Lie \check {\mathrm G} \right)^* (1) \right)^{\vee} \lra H^1_f \left( \Z \left[ \frac{1}{S} \right],  \left( \Lie \check {\mathrm G} \right)^* (1) \right)^{\vee}. \]

As explained in \cite{akshay}, section 1.2, to the coadjoint representation $\left( \Lie \check {\mathrm G} \right)^*$ one can conjecturally associate a weight zero Chow motive $M_{\coad}$ over $\Q$, with the property that \[ H^*_{\mathrm{et}} \left( M_{\coad} \times_{\Q} \overline \Q, E \right) \cong H^0_{\mathrm{et}} \left( M_{\coad} \times_{\Q} \overline \Q, E \right) \cong \left( \Lie \check {\mathrm G} \right)^*. \]
After Voevodsky and others, one has comparison maps from motivic cohomology to \`etale cohomology - the one we are interested in is: \[ \mathrm{reg}: H^1_{\mathcal M} \left( M_{\coad}, \Q(1) \right) \otimes_{\Q} E \lra H^1 \left( \Q, \left( \Lie \check {\mathrm G} \right)^* (1) \right). \]
In \cite{scholl}, Scholl defines a subspace of `integral classes' inside motivic cohomology (see in particular theorem 1.1.6 of loc cit.).
The above regulator map is then expected to map the integral classes inside the crystalline Selmer subgroup, and moreover the map \[ \mathrm{reg}: H^1_{\mathcal M} \left( (M_{\coad})_{\Z}, \Q(1) \right) \otimes_{\Q} E \lra H^1_f \left( \Z \left[ \frac{1}{S} \right], \left( \Lie \check {\mathrm G} \right)^* (1) \right) \] is conjecturally an isomorphism (see \cite{BK2} conjecture 5.3(ii)).

Taking the $\Qp$-linear duals we obtain an isomorphism: \[ \mathrm{reg}^{\vee} : H^1_f \left( \Z \left[ \frac{1}{S} \right], \left( \Lie \check {\mathrm G} \right)^* (1) \right)^{\vee} \lra \left( H^1_M \left( (M_{\coad})_{\Z}, \Q (1) \right) \otimes_{\Q} E \right)^{\vee} \cong \Hom_{\Q} \left( H^1_M \left( (M_{\coad})_{\Z}, \Q (1) \right), E \right). \]
Here's a diagram with all the maps defined so far: \begin{displaymath} \label{BIGMAP} \xymatrix{ H^1_f \left( \Qp, \left( \Lie \check {\mathrm G} \right)^* (1) \right)^{\vee} \ar[d]_{\mathrm{res}^{\vee}} & \frac{H^1 \left( \Qp, \Lie \check {\mathrm G}  \right)}{H^1_f \left( \Qp, \Lie \check {\mathrm G} \right)} \ar[l]_{\cong} & \frac{H^1 \left( \Qp, \Lie \check {\mathrm B} \right)}{H^1_f \left( \Qp, \Lie \check {\mathrm B} \right)} \ar[r]^{\cong} \ar@{_{(}->}[l] & \Hom \left( \mathrm T(\Zp)_p, E \right)  \\ H^1_f \left( \Z \left[ \frac{1}{S} \right], \left( \Lie \check {\mathrm G} \right)^* (1) \right)^{\vee} \ar[r]^{\mathrm{reg}^{\vee}}_{\cong} & \left( H^1_M \left( (M_{\coad})_{\Z}, \Q (1) \right) \otimes_{\Q} E \right)^{\vee} & &} \end{displaymath}
We obtain thus a map \[ \Hom \left( \mathrm T(\Zp)_p, E \right) \lra \left( H^1_M \left( (M_{\coad})_{\Z}, \Q (1) \right) \otimes_{\Q} E \right)^{\vee} \cong \Hom_{\Q} \left( H^1_M \left( (M_{\coad})_{\Z}, \Q (1) \right), \Q \right) \otimes_{\Q} E. \]
\begin{conj} \label{rationalactionconj}
The action of the derived Hecke algebra on $H^* ( Y(K), E )$ factors through motivic cohomology according to the map above.

Moreover, the resulting action of $\Hom_{\Q} \left( H^1_{\mathcal M} \left( \left( M_{\coad} \right)_{\Z}, \Q(1) \right) , \Q \right) \otimes_{\Q} E$ on $H^* \left( Y(K), E \right) \cong H^* \left( Y(K), \Q \right) \otimes_{\Q} E$ comes from a rational action of $\Hom_{\Q} \left( H^1_{\mathcal M} \left( \left( M_{\coad} \right)_{\Z}, \Q(1) \right) , \Q \right)$ on $H^* \left( Y(K), \Q \right)$.
\end{conj}
As a sanity check, we prove this in the case that $\mathrm G$ is a torus.

\begin{prop} Let $\mathrm G = \mathrm T$ be an anisotropic torus, then the conjecture \ref{rationalactionconj} holds.
\end{prop}
\begin{proof}
The proof follows almost verbatim section 9.1 of \cite{akshay}. For simplicity, we assume the coefficient field $E = \Qp$.
Let $F / \Q$ be a number field and let $\mathrm T$ be an anisotropic $F$-torus and fix a finite set of primes $S$ such that $\mathrm T$ admits a smooth $\OO_F \left[ \frac{1}{S} \right]$-model - which we also denote by $\mathrm T$.
The associated symmetric space is \[ \mathcal S = \mathrm T \left( F \otimes_{\Q} \R \right) / \textnormal{maximal compact} \] and for a good open compact subgroup $K \subset \mathrm T(\Af)$, the arithmetic manifold $Y(K)$ is a finite union of orbifolds $\mathcal S / \Delta$, where $\Delta \subset \mathrm T \left( \OO_F \right)$ is a congruence subgroup.

Our general setup assumes that $\mathrm T$ is split at any place $v$ above $p$, so in this case it coincides with its maximal, $F_v$-split subtorus, and the Satake homomorphism is the identity map between the same Hecke algebra $\mathcal H^* (G_v, K_v) \cong \mathcal H^* (T_v, T_v \cap K_v)$.

Any cohomological, cuspidal automorphic form for $\mathrm T$ factors through the identity component $\mathrm T^0 \left( F \otimes_{\Q} \R \right)$ - so that (as explained in \cite{akshay}, section 9.2) we are just considering finite order id\'ele class characters for $\mathrm T$.

As for the associated Galois representation $\rho: \Gal_F \lra \check {\mathrm T}(\Qp)$, we notice that the coadjoint representaion $\left( \Lie \check {\mathrm G} \right)^*$ is the trivial representation of $\Gal_F$ on $(\Lie \check {\mathrm T})^*$, because the coadjoint action of $\check {\mathrm T}$ on $(\Lie \check {\mathrm T})^*$ is trivial.
In particular, since \[ ( \Lie \check {\mathrm T})^* \cong \left( \Lie \Gm \otimes_{\Z} X_* ( \check {\mathrm T} ) \right)^* \cong X^* (\check {\mathrm T}) \otimes_{\Z} \Qp \cong X_* ( \mathrm T) \otimes_{\Z} \Qp \] we infer that the Tate twist of the coadjoint motive $M_{\coad}$ has $p$-adic realization \[ M_p= (M_{\coad})_p (1) \cong X_*( \mathrm T) \otimes_{\Z} \Qp (1). \]

As mentioned earlier, the Galois action on $\left( \Lie \check {\mathrm T} \right)^*$ is trivial, and upon applying the Tate twist we obtain $\left( \Lie \check {\mathrm T} \right)^* \cong X_*( \mathrm T) \otimes_{\Z} \Qp (1)$, where the Galois-action on the first tensor factor is the trivial one.

We have then \[ H^1_f \left( F, \left( \Lie \check {\mathrm T} \right)^* (1) \right) \cong H^1_f \left( F, X_*( \mathrm T) \otimes_{\Z} \Qp (1) \right) \cong H^1_f \left( F, \Qp(1) \right) \otimes_{\Z} X_*(\mathrm T), \] where we can pull out $X_*(\mathrm T)$ since the Galois action on it is trivial.

Now a standard computation with the Kummer sequence (see for example \cite{feng} or \cite{bellaiche}, proposition 2.12) gives that \[ H^1_f \left( F, \Qp(1) \right) \cong \OO_F^* \otimes_{\Z} \Qp \] so that \[ H^1_f \left( F, \Qp(1) \right) \otimes_{\Z} X_*(\mathrm T) \cong X_*(\mathrm T) \otimes_{\Z} \OO_F^* \otimes_{\Z} \Qp \cong \mathrm T \left( \OO_F \right) \otimes_{\Z} \Qp. \]

Again like in \cite{akshay}, the motivic cohomology $H^1_{\mathcal M} \left( M, \Q (1) \right)$ identifies with $\mathrm T (F) \otimes_{\Z} \Q$, and the integral classes are \[ H^1_{\mathcal M} \left( M_{\Z}, \Q(1) \right) \cong \Delta \otimes_{\Z} \Q \subset \mathrm T(F) \otimes_{\Z} \Q. \]
Moreover, in this case $\mathrm G = \mathrm B = \mathrm T$, so that the previous diagram reduces to \begin{displaymath} \label{BIGMAP2} \xymatrix{ H^1_f \left( F_v, \left( \Lie \check{\mathrm T} \right)^* (1) \right)^{\vee} \ar[d]_{\mathrm{res}^{\vee}} & \frac{H^1 \left( F_v, \Lie \check {\mathrm T} \right)}{H^1_f \left( F_v, \Lie \check{\mathrm T} \right)} \ar[r]^{\cong} \ar[l]_{\cong} & \Hom \left( \mathrm T \left( \OO_{F_v} \right), \Qp \right)  \\ \Hom_{\Qp} \left( T \left( \OO_F \right) \otimes_{\Z} \Qp , \Qp \right) \ar[r]^{\mathrm{reg}^{\vee}}_{\cong} & \Hom_{\Qp} \left( \Delta \otimes_{\Z} \Qp, \Qp \right) \ar[r]^{\cong} & \Hom \left( \Delta , \Qp \right) } \end{displaymath}
In other words, the map reduces to the restriction map $\Hom \left( \mathrm T \left( \OO_{F_v} \right) , \Qp \right) \stackrel{ \mathrm{res}^{\vee}}{\lra} \Hom \left( \Delta, \Qp \right)$.
Thus, to prove the claim it suffices to show that the derived diamond operators act on $H^* \left( Y(K) , \Qp \right)$ by pulling back via $\Delta \hookrightarrow \mathrm T \left( \OO_{F_v} \right)$.

This is described in \cite{akshay}, section 9.1, using the fact that the arithmetic manifold is a finite union of copies of $\mathcal S / \Delta$, so that the action of each derived diamond operator is obtained by pulling back the cohomology class through $\Delta \hookrightarrow \mathrm T \left( \OO_{F_v} \right)$, and then applying it to $H^* \left( \mathcal S / \Delta \right)$.
\end{proof}

{\bf Address of authors:} {Department of Mathematics, UCLA, Los Angeles, CA 90095-1555,
  USA}

{\tt email:} {\tt shekhar@math.ucla.edu, niccronc@math.ucla.edu}

\end{document}